\documentclass[12pt,Bold,letterpaper]{mcgilletdclassBob}
\usepackage{amsmath}
\usepackage{amsthm}
\usepackage{amssymb}
\usepackage{amsfonts}
\usepackage{mathrsfs}
\usepackage{verbatim}
\usepackage{float}
\usepackage{natbib}

\usepackage{nomencl}
\makenomenclature

\usepackage[dvips,final]{graphicx}
\SetTitle{\huge{Reduction in Solving Some \\ Integer Least Squares Problems}}%
\SetAuthor{Mazen Al Borno}%
\SetDegreeType{Master of Science}%
\SetDepartment{School of Computer Science}%
\SetUniversity{McGill University}%
\SetUniversityAddr{Montreal, Quebec}%
\SetThesisDate{February, 2011}%
\SetRequirements{A thesis submitted to McGill University\\ 
in partial fulfilment of the requirements of the degree of\\ 
Master of Science in Computer Science
}
\SetCopyright{\copyright Mazen Al Borno 2011}%

\renewcommand{\baselinestretch}{1.5}

\newtheorem{algorithm}{Algorithm}[section]
\newtheorem{THEOREM}{THEOREM}[section]

\newcommand{\be}{\begin{equation}}
\newcommand{\ee}{\end{equation}}
\newcommand{\beqy}{\begin{eqnarray}}
\newcommand{\eeqy}{\end{eqnarray}}
\newcommand{\beqynn}{\begin{eqnarray*}}
\newcommand{\eeqynn}{\end{eqnarray*}}
\newcommand{\bi}{\begin{itemize}}
\newcommand{\ei}{\end{itemize}}

\newcommand{\ba}{\begin{array}}
\newcommand{\ea}{\end{array}}
\newcommand{\bmx}{\begin{bmatrix}}
\newcommand{\emx}{\end{bmatrix}}
\newcommand{\bsmx}{\left[\begin{smallmatrix}}
\newcommand{\esmx}{\end{smallmatrix}\right]}

\newcommand{\bmxc}[1]{\left[\begin{array}{@{}#1@{}}}
\newcommand{\emxc}{\end{array}\right]}

\newcommand{\bc}{\begin{center}}
\newcommand{\ec}{\end{center}}

\newcommand{\diag}{\mathrm{diag}}

\newcommand{\nbn}{{n \times n}}

\newcommand{\A}{\boldsymbol{A}}
\newcommand{\B}{\boldsymbol{B}}

\newcommand{\D}{\boldsymbol{D}}

\newcommand{\G}{\boldsymbol{G}}

\newcommand{\I}{\boldsymbol{I}}

\renewcommand{\L}{\boldsymbol{L}}
\newcommand{\M}{\boldsymbol{M}}

\renewcommand{\P}{\boldsymbol{P}}
\newcommand{\Q}{\boldsymbol{Q}}
\newcommand{\R}{\boldsymbol{R}}

\newcommand{\U}{\boldsymbol{U}}

\newcommand{\W}{\boldsymbol{W}}

\newcommand{\Z}{\boldsymbol{Z}}

\renewcommand{\a}{\boldsymbol{a}}

\renewcommand{\c}{\boldsymbol{c}}

\newcommand{\e}{\boldsymbol{e}}

\renewcommand{\l}{\boldsymbol{l}}
\newcommand{\n}{\boldsymbol{n}}
\newcommand{\p}{\boldsymbol{p}}
\newcommand{\q}{\boldsymbol{q}}

\renewcommand{\u}{\boldsymbol{u}}
\renewcommand{\v}{\boldsymbol{v}}
\newcommand{\w}{\boldsymbol{w}}
\newcommand{\x}{{\boldsymbol{x}}}
\newcommand{\y}{{\boldsymbol{y}}}
\newcommand{\z}{\boldsymbol{z}}

\newcommand{\0}{{\boldsymbol{0}}}

\newcommand{\bd}{{\bar{d}}}
\newcommand{\bl}{{\bar{l}}}

\newcommand{\cbx}{{\check{\x}}}
\newcommand{\cbz}{{\check{\z}}}

\newcommand{\hz}{{\hat{z}}}

\newcommand{\hbx}{{\hat{\x}}}

\newcommand{\hbz}{{\hat{\z}}}

\newcommand{\bbz}{{\bar{\z}}}
\newcommand{\bz}{{\bar{z}}}
\newcommand{\bx}{{\bar{x}}}

\newcommand{\bbL}{{\bar{\L}}}
\newcommand{\bbD}{{\bar{\D}}}

\newcommand{\bbR}{{\bar{\R}}}
\newcommand{\tR}{{\tilde{\R}}}

\newcommand{\by}{{\bar{\y}}}
\newcommand{\bu}{{\bar{\u}}}
\newcommand{\bbl}{{\bar{\l}}}

\providecommand{\norm}[1]{\lVert#1\rVert}
\providecommand{\round}[1]{\lfloor#1\rceil}

\renewcommand{\nomname}{\BTitle{\ETDAbbreviationName}\vspace*{0.2in}}

\begin{document}

\maketitle

\begin{romanPagenumber}{2}%

\SetDedicationName{DEDICATION}
\SetDedicationText{\centering \IVisualEmphasis{To my loving parents,\\ 
Manal Abdo and Ahmed-Fouad Al Borno \\}
\vspace{70mm}
\centering \IVisualEmphasis{
God is the Light of the heavens and the earth.\\ The Parable of His Light is as if there were a Niche and within it a Lamp:\\ the Lamp enclosed in Glass: the glass as it were a brilliant star:\\ Lit from a blessed Tree, an Olive, neither of the east nor of the west,\\ whose oil is well-nigh luminous, though fire scarce touched it:\\ Light upon Light!} \\
{\bf Qur'an}, 24:35
}
\Dedication

\SetAcknowledgeName{ACKNOWLEDGMENTS}
\SetAcknowledgeText{
I am fortunate to acknowledge: 
\begin{itemize}
\item  Xiao-Wen Chang as an advisor; for his patience, constructive criticism and careful review of this thesis which went above and beyond the call of duty.  
\item  true friends at the School of Computer Science; I only wish I was as good a friend to you, as you to me; Stephen Breen, Zhe Chen, Arthur Guez, Sevan Hanssian, Wen-Yang Ku, Mathieu Petitpas, Shaowei Png, Milena Scaccia, Ivan Savov, Ben Sprott, David Titley-Peloquin, Yancheng Xiao, Xiaohu Xie, and especially Omar Fawzi;  
\item  a wonderful course on Quantum Information Theory by Patrick Hayden;
\item a younger brother for his limitless generosity and a loving family whose the source of my joy and the removal of my sorrow.  
\end{itemize}
}
\Acknowledge

\SetAbstractEnName{ABSTRACT}
\SetAbstractEnText{
Solving an integer least squares (ILS) problem usually consists of two stages:
reduction and search. This thesis is concerned with the reduction process
for the ordinary ILS problem and the ellipsoid-constrained ILS problem.
For the ordinary ILS problem, we dispel common misconceptions on the reduction stage
in the literature and show what is crucial to the efficiency of the
search process. The new understanding allows us to design a new reduction
algorithm which is more efficient than the well-known LLL reduction algorithm.
Numerical stability is taken into account in designing the new reduction algorithm.
For the ellipsoid-constrained ILS problem, we propose a new
reduction algorithm which, unlike existing algorithms, uses all the available
information. Simulation results indicate that new algorithm can greatly reduce
the computational cost of the search process when the measurement noise is large.
}
\AbstractEn
\SetAbstractFrName{ABR\'{E}G\'{E}}
\SetAbstractFrText{
La r\'{e}solution de probl\`{e}mes de moindres carr\'{e}s en nombres entiers (ILS) comprend habituellement deux stages: la r\'{e}duction et la recherche.
Cette th\`{e}se s'int\'{e}resse \`{a} la r\'{e}duction pour le probl\`{e}me ILS ordinaire et le probl\`{e}me ILS sous contrainte d'ellipse.  Pour le probl\`{e}me ILS ordinaire, nous  dissipons des erreurs communes de compr\'{e}hension \`{a} propos de la r\'{e}duction dans la litt\'{e}rature et nous montrons ce qui est r\'{e}ellement crucial pour l'efficacit\'{e} de la recherche.   Ce r\'{e}sultat nous permet de d\'{e}velopper un nouvel algorithme de r\'{e}duction plus efficace que le c\'{e}l\`{e}bre algorithme LLL.  La stabilit\'{e} num\'{e}rique est prise en compte dans le d\'{e}veloppement du nouvel algorithme.
Pour le probl\`{e}me ILS sous contrainte d'ellipse, nous proposons un nouvel algorithme de r\'{e}duction qui, contrairement aux algorithmes existants, utilise toute l'information disponible.  Les r\'{e}sultats de simulations indiquent que le nouvel algorithme r\'{e}duit consid\'{e}rablement les co\^{u}ts de calcul de la recherche lorsque la variance du bruit est large dans le mod\`{e}le lin\'{e}aire.   
}
\AbstractFr

\SetDeclarationName{DECLARATION}
\SetDeclarationText{
The initial version of this thesis was submitted before finding \cite{LinN07}.
Here, we point out the similarities and differences with the results 
derived in Chapter 4 of this thesis.
First, they show that integer Gauss transformations do not affect the Babai point, while we show that integer Gauss transformations do not affect the entire search process.
Second, their modified reduction algorithm uses Gram-Schmidt orthogonalization.
Our modified reduction algorithm uses Householder reflections and Givens rotation. The latter is more stable.   
Third, our algorithm  does not apply an integer Gauss transformation if no column permutation is needed as it is unnecessary.
Finally, they do not take numerical stability into account, while we do.
Specifically, our reduction algorithm uses extra integer Gauss transformations to prevent the serious rounding errors that can occur due to the increase of the off-diagonal elements. 
}

\tableofcontents
\listoftables
\listoffigures

\nomenclature{BILS:}{box-constrained integer least squares}
\nomenclature{CLLL:}{constrained LLL}
\nomenclature{EILS:}{ellipsoid-constrained integer least squares}
\nomenclature{GNSS:}{global navigation satellite systems}
\nomenclature{i.i.d.:}{independent and identically distributed}
\nomenclature{IGT:}{integer Gauss transformation}
\nomenclature{ILS:}{integer least squares}
\nomenclature{LAMBDA:}{Least-squares AMBiguity Decorrelation Adjustment}
\nomenclature{LS:}{least squares}
\nomenclature{MLAMBDA:}{Modified Least-squares AMBiguity Decorrelation Adjustment}
\nomenclature{PREDUCTION:}{Partial Reduction}
\nomenclature{PLLL:}{partial LLL}
\printnomenclature
\Abbreviation

\chapter{Introduction}
Suppose we have the following linear model
\begin{equation}\label{eq:linearmodel}
\y = \A \x + \v, 
\end{equation}
where $\y \in \mathbb{R}^m$ is a \emph{measurement vector}, $\A \in \mathbb{R}^{m \times n}$ is a \emph{design matrix} with full column rank, $\x \in \mathbb{R}^n$ is an unknown \emph{parameter vector}, and $\v \in \mathbb{R}^n$ is a \emph{noise vector} that follows a normal distribution with mean $\0$ and covariance $\sigma^2 \I$.  
We need to find an estimate $\hbx$ of the unknown vector $\x$.  One approach is to solve the following \emph{real least squares (LS)} problem
\begin{equation}\label{eq:realLE}
\min_{\x \in \mathbb{R}^n} \| \y-\A \x \|_2^2.
\end{equation}
We refer to \eqref{eq:realLE} as the standard form of the LS problem. 
The real least squares solution is given by (see, e.g., \cite[Section 5.3]{GolV96})
\begin{equation}\label{eq:realS}
\hbx = (\A^T \A)^{-1} \A^T \y.
\end{equation}
Substituting \eqref{eq:linearmodel} into \eqref{eq:realS}, we get
\begin{equation}\label{eq:bLCP}
\hbx = \x + (\A^T \A)^{-1} \A^T \v.
\end{equation} 
Let $\W_{\hbx}\in \mathbb{R}^{n\times n}$ be the covariance matrix of $\hbx$.
Using the law of covariance propagation on \eqref{eq:bLCP} (see, e.g., \cite[p.\ 329]{StrB97}), it 
follows that $\W_{\hbx} = \sigma^2(\A^T \A)^{-1}$.  

\end{romanPagenumber}
In many applications, $\x$ is constrained to some discrete integer set $\mathcal{D}$, e.g., a box constraint.  Then, one wants to solve the 
\emph{integer least squares (ILS) problem} 
\begin{equation}\label{eq:gILSintro}
\min_{\x \in \mathcal{D}} \| \y-\A \x \|_2^2.
\end{equation}
If $\mathcal{D}$ is the whole integral space $\mathbb{Z}^n$, then we refer to  
\begin{equation}\label{eq:OILSintro}
\min_{\x \in \mathbb{Z}^n} \| \y-\A \x \|_2^2
\end{equation}
as the \emph{ordinary integer least squares (OILS) problem}.
In lattice theory, the set $\mathcal{L}(\A) = \{ \A \x : \x \in \mathbb{Z}^n \}$ is referred to as the lattice generated by $\A$.   The OILS problem is to find the point in $\mathcal{L}(\A)$ which is closest to $\y$.  For this reason,  the OILS problem is also called the \emph{closest point problem}.   
Since the residual $\y-\A \hbx$ is orthogonal to the range of $\A$, we have
\begin{equation}\label{eq:dev}
 \begin{split}
    \| \y-\A \x \|_2^2  & =   \| \y-\A \hbx - \A(\x - \hbx) \|_2^2 \\
                        & =  \| \y-\A \hbx \|_2^2 + \| \A(\x - \hbx) \|_2^2 \\
                        & =   \| \y-\A \hbx \|_2^2 + (\x - \hbx)^T \A^T\A (\x - \hbx) \\
                        & =   \| \y-\A \hbx \|_2^2 + \sigma^2(\x - \hbx)^T \W_{\hbx}^{-1} (\x - \hbx).  
  \end{split}
\end{equation} 
Thus, the OILS problem \eqref{eq:OILSintro} is equivalent to
\begin{equation}\label{eq:qilsintro}
\min_{\x\in \mathbb{Z}^n} (\x-\hbx)^T\W_{\hbx}^{-1}(\x-\hbx).
\end{equation} 
We refer to \eqref{eq:qilsintro} as the quadratic form of the OILS problem. 

ILS problems arise from many applications, such as global navigation satellite systems, communications, bioinformatics, radar imaging, cryptography, Monte Carlo second-moment estimation, lattice design, etc., see, e.g., \cite{AgrEVZ02} and \cite{HasB98}.  Unlike the real least squares problem 
\eqref{eq:realLE}, the ILS problem (\ref{eq:gILSintro}) is NP-hard (see \cite{Boa81} and \cite{Mic01}).  Since all known algorithms to solve the ILS problem have exponential complexity, designing efficient algorithms is crucial for real-time applications. A typical approach for solving an ILS problem consists of two stages: reduction and search.  The reduction process transforms (\ref{eq:gILSintro}) to a new ILS problem, where $\A$ is reduced to an upper triangular matrix.  The search process searches for the solution of the new ILS problem in a geometric region.  The main goal of the reduction process is to make the search process more efficient. 

For solving (\ref{eq:OILSintro}), two typical reduction strategies are employed in practice.  One is the Korkine-Zolotareff (KZ) reduction (see \cite{KorZ73}), which transforms the original OILS problem to a new one which is optimal for the search process.  The other is the Lenstra-Lenstra-Lov\'{a}sz (LLL) reduction (see \cite{LenLL82}), which transforms the original OILS problem to a new one which is approximately optimal for the search process. 
Unlike the KZ reduction, it is known how to compute the LLL reduction efficiently in polynomial time.  For this reason, the LLL reduction is more widely used in practice.  Simulations in \cite{AgrEVZ02} suggest to use the KZ reduction only if we have to solve many OILS problems with the same generator matrix $\A$.  Otherwise, the LLL reduction should be used.  For the search process, two common search strategies are the Pohst enumeration strategy (see \cite{FinP85}) and the Schnorr-Euchner enumeration strategy (see \cite{SchE94}).  Both examine the lattice points lying inside a hyper-sphere, but in a different order.  Simulations in \cite{AgrEVZ02} indicate that the Schnorr-Euchner strategy is more efficient than the Pohst strategy.        

In high precision relative global navigation satellite systems (GNSS) positioning, a key component is to resolve the unknown so-called 
double differenced cycle ambiguities of the carrier phase data as integers. 
The most successful method of ambiguity resolution in the GNSS literature
is the well-known LAMBDA (Least-squares AMBiguity Decorrelation Adjustment) method
presented by Teunissen (see, e.g., \cite{Teu93, Teu95b, Teu95,  Teu98, Teu99}).
This method solves the OILS problem \eqref{eq:qilsintro}.
A detailed description of the LAMBDA algorithm and implementation is given by \cite{JonT96}.
Its software (Fortran version  and MATLAB version) is available from
Delft University of Technology. Frequently asked questions and misunderstanding
about the LAMBDA method are addressed by \cite{JooT02}.
Recently, a modified method called MLAMBDA was proposed by Chang etc in \cite{ChaYZ05},
which was then further modified and extended to handle mixed ILS problems
by using orthogonal transformations, resulting in the MATLAB package MILS (see \cite{ChaZ07}).

In some applications, the point $\A \x$ in (\ref{eq:gILSintro}) is constrained to be inside a given hyper-sphere.  The ILS problem becomes
\begin{equation}\label{eq:EILSintro}
\min_{\x \in \mathcal{E}} \| \y-\A \x \|_2^2, \quad 
\mathcal{E} = \{ \x \in \mathbb{Z}^n: \| \A \x \|_2^2 \le \alpha^2 \},
\end{equation}
see \cite{ChaG09} and \cite{DamEC03}.
We refer to (\ref{eq:EILSintro}) as the \emph{ellipsoid-constrained integer least squares (EILS) problem}.  To solve the EILS problem, \cite{ChaG09} proposed the LLL reduction in the reduction stage and modified the Schnorr-Euchner search strategy to handle the ellipsoidal constraint in the search stage.  

The contribution of this thesis is two-fold.  The first is to dispel common misconceptions on the reduction process appearing in the ILS literature. 
The second is to present more efficient algorithms to solve the ILS and EILS problem.  The thesis is organized as follows. 

In Chapter \ref{s:itILS},
we review the typical methods to solve an OILS problem.  Specifically, we introduce the LLL reduction method and the Schnorr-Euchner search strategy.

In Chapter \ref{s:GPS}, we discuss the typical methods to solve the quadratic form of the OILS problem.  Our focus will be on LAMBDA reduction and the modified reduction (MREDUCTION) given in \cite{ChaYZ05}.  

According to the literature, there are two goals the reduction process should achieve to make the search process efficient.  One of them is to transform the generator matrix $\A$ in (\ref{eq:OILSintro}) or the covariance matrix $\W_{\hbx}$ in (\ref{eq:qilsintro}) to a matrix which is as close to diagonal as possible.  
A covariance matrix which is close to diagonal means that there is little correlation between its random variables.  
In other words, the goal of the reduction in the GNSS context is to decorrelate the ambiguities as much as possible.
To achieve this goal, the reduction process uses so-called integer Gauss transformations.  
In Chapter \ref{s:misconceptions}, we show that, contrary to common belief, this goal will not make the search more efficient.  We provide a new explanation on the role of integer Gauss transformations in the reduction process.   This new understanding results in modifications to the existing reduction methods.  Numerical simulations indicate that our new algorithms are more efficient than the existing algorithms.  Finally, we discuss another misconception in some GNSS literature where it is believed that the reduction process should reduce the condition number of the covariance matrix.  We provide examples that show that this goal should be discarded.              
   
In Chapter \ref{s:EILS}, we discuss the EILS problem. One drawback with the existing reduction algorithms is that the search time becomes more and more prohibitive as the noise $\v$ in (\ref{eq:linearmodel}) gets larger.  We present a new reduction algorithm which, unlike the LLL reduction, uses the information of the input vector $\y$ and the ellipsoidal constraint.  Then, we provide simulations that show that our proposed approach greatly improves the search process for large noise.   

Finally, we summarize our results and mention our future work in Chapter \ref{s:summary}.                 

We now describe the notation used in this thesis.
The sets of all real and integer $m\times n$ matrices are denoted by $\mathbb{R}^{m\times n}$
and $\mathbb{Z}^{m\times n}$, respectively,
and the set of real and integer $n$-vectors are denoted by $\mathbb{R}^n$
and $\mathbb{Z}^n$, respectively.  
Bold upper case letters and bold lower case letters denote matrices and vectors, respectively.
The identity matrix is denoted by $\I$ and its $i$th column is denoted by $\e_i$.
Superscript $T$ denotes the transpose.  The 2-norm of a vector or a matrix is denoted by $\| \cdot \|_2$. 
MATLAB notation is used to denote a submatrix.
Specifically, if $\A=(a_{ij}) \in \mathbb{R}^{m\times n}$, then
$\A(i,:)$ denotes the $i$th row, $\A(:,j)$ the $j$th column,
and $\A(i_1\!:\!i_2,j_1\!:\!j_2)$ the submatrix formed by rows $i_1$ to $i_2$ and
columns $j_1$ to $j_2$.
For the $(i,j)$ element of $\A$,  we denote it by $a_{ij}$ or $\A(i,j)$.
For the $i$th entry of a vector $\a$, we denote it by $a_i$.
For a scalar $z \in \mathbb{R}$, we use $\lfloor z\rceil$ to denote its nearest integer.
If there is a tie, $\lfloor z\rceil$ denotes the one with smaller magnitude.  The operation sign(z) returns 
$-1$ if $z \le 0$ and $1$ if $z > 0$.
For a random vector $\x \in \mathbb{R}^n$, $\x \sim \mathcal{N}(0,\sigma^2 I)$ means that $\x$ follows a normal distribution with mean $\0$ 
and covariance matrix $\sigma^2 \I$.    
We use i.i.d. to abbreviate ``independently and identically distributed".

\chapter{The OILS problem in the standard form}\label{s:itILS}  
The OILS problem is defined as
\begin{equation}\label{eq:ILS}
\min_{\x \in \mathbb{Z}^n} \| \y-\A \x \|_2^2,
\end{equation}
where $\y \in \mathbb{R}^n$ and $\A \in \mathbb{R}^{m \times n}$ has full column rank.
In this chapter, we review the typical methods to solve the OILS problem. 
A typical approach for solving an OILS problem consists of two stages: reduction and search.  The main goal of the reduction process is to make the search process efficient.  In order to better understand the aims of the reduction, we first introduce the search process and the Schnorr-Euchner search strategy in Section \ref{s:search}.  Then, we present the reduction process and the well-known LLL reduction in Section \ref{s:reduction}.
 
\section{Search process}\label{s:search}

Suppose that after the reduction stage, the OILS problem (\ref{eq:ILS}) is transformed to
\begin{equation}\label{eq:rILS}
\min_{\z \in \mathbb{Z}^n} \| \by-\R \z \|_2^2,
\end{equation}
where $\R \in \mathbb{R}^{n \times n}$ is nonsingular upper triangular and $\by \in \mathbb{R}^{n}$. 
Assume that the solution of (\ref{eq:rILS}) satisfies the bound
\begin{equation*}\label{eq:rILS2}
\| \by-\R \z \|_2^2 < \beta^2,
\end{equation*}
or equivalently
\begin{equation}\label{eq:expand}
\sum_{k=1}^n (\bar{y}_k - \sum_{j=k+1}^n r_{kj}z_j - r_{kk} z_k)^2 < \beta^2.
\end{equation}

Since (\ref{eq:expand}) is a hyper-ellipsoid, we refer to it as a search ellipsoid. 
The goal of the search process is to find an integer point inside the search ellipsoid which minimizes the left-hand side of (\ref{eq:expand}). 
 
Let
\begin{equation}\label{eq:c_k}
c_n = \bar{y}_n/r_{nn}, \quad c_k = (\bar{y}_k - \sum_{j=k+1}^n r_{kj}z_j)/r_{kk}, \quad k=n-1\!:\!-1\!:\!1.
\end{equation}
Note that $c_k$ depends on $z_{k+1}, \ldots, z_n$.
Substituting (\ref{eq:c_k}) in (\ref{eq:expand}), we have
\begin{equation}\label{eq:objf}
\sum_{k=1}^n r_{kk}^2 (z_k - c_k)^2 < \beta^2.
\end{equation}
If $\z$ satisfies the bound, then it must also satisfy inequalities
\begin{align}
& \mbox{level }n : r_{nn}^2 (z_n - c_n)^2 < \beta^2, \label{eq:ineq_n}\\ 
& \quad \vdots \nonumber \\
& \mbox{level }k : r_{kk}^2 (z_k - c_k)^2 < \beta^2 - \sum_{i=k+1}^n r_{ii}^2 (z_i - c_i)^2 \label{eq:ineq_k}\\ 
& \quad \vdots \nonumber \\
& \mbox{level }1 : r_{11}^2 (z_1 - c_1)^2 < \beta^2 - \sum_{i=2}^n r_{ii}^2 (z_i - c_i)^2.  \label{eq:ineq_1}
\end{align}

The search process starts at level $n$ and moves down to level 1.  At level $k$, $z_{k}$ is determined for $k=n\!:\!-1\!:\!1$.  
From (\ref{eq:ineq_k}), the range of $z_k$ is $[l_k,u_k]$, where
\begin{equation*}
l_k = \Big \lceil c_k - (\beta^2 - \sum_{i=k+1}^n r_{ii}^2 (z_i - c_i)^2)^{1/2}/|r_{kk}| \Big \rceil 
\end{equation*} 
and
\begin{equation*}
u_k = \Big \lfloor c_k + (\beta^2 - \sum_{i=k+1}^n r_{ii}^2 (z_i - c_i)^2)^{1/2}/|r_{kk}| \Big \rfloor.
\end{equation*}

There are two typical strategies to examine the integers inside $[l_k,u_k]$.
In the Pohst strategy (see \cite{FinP85}), the integers are chosen in the ascending order
\begin{equation*}
l_k,l_k+1, l_{k}+2, \ldots,  u_k.
\end{equation*}
In the Schnorr-Euchner strategy (see \cite{SchE94}), the integers are chosen in the zig-zag order
\begin{equation}\label{eq:seq}
z_k = \left\{ 
\begin{array}{l l}
  \lfloor c_k \rceil ,\lfloor c_k \rceil-1,
\lfloor c_k \rceil+1,\lfloor c_k \rceil-2,\hdots,\text{ if } c_k \leq \lfloor c_k \rceil\\
\lfloor c_k \rceil ,\lfloor c_k \rceil+1,
\lfloor c_k \rceil-1,\lfloor c_k \rceil+2,\hdots,\text{ if } c_k \geq \lfloor c_k \rceil.
\end{array} \right.
\end{equation} 

Observe that in Schnorr-Euchner strategy, once an integer $z_k$ does not satisfy (\ref{eq:ineq_k}), all the following integers in the sequence will not satisfy it.  These integers can be pruned from the search process.  Such a property does not exist in the Pohst strategy.  Another benefit with the Schnorr-Euchner enumeration order is that the first points examined are more likely to minimize (\ref{eq:objf}) than the last points examined.  As will be seen in the next paragraph, this allows to shrink the search ellipsoid faster.  Simulations in \cite{AgrEVZ02} confirm that the Schnorr-Euchner strategy is more efficient than the Pohst strategy.

We now describe the search process using the Schnorr-Euchner strategy.  At level $n$, we compute $c_n$ by (\ref{eq:c_k}) and set $z_n =  \round{c_n}$.  If (\ref{eq:ineq_n}) is not satisfied, no integer can satisfy (\ref{eq:objf}).  Otherwise, we go to level $n-1$, compute $c_{n-1}$ and set $z_{n-1} = \round{c_{n-1}}$. If (\ref{eq:ineq_k}) does not hold, we go back to level $n$ and choose $z_n$ to be the second nearest integer to $c_n$.  Otherwise, we move down to level $n-2$.  When we reach level $1$, we compute $c_1$ and set $z_1 = \round{c_1}$. Then, if (\ref{eq:ineq_1}) is satisfied, we set $\hbz = [z_1, \ldots, z_n]^T$, where $\hbz$ is a full integer point inside the search ellipsoid. We update $\beta$ by setting $\beta^2 = \sum_{k=1}^n r_{kk}^2 (\hat{z}_k - c_k)^2$.  This step allows to eliminate more points by ``shrinking'' the search ellipsoid.  Now we search for a better point than $\hbz$. If one is found, we update $\hbz$.  We move up to level $2$ and choose $z_2$ to be the next nearest integer to $c_2$, where ``next" is relative to $\hat{z}_2$.  If inequality  (\ref{eq:ineq_k}) holds at level 2, we move down to level 1 and update $z_1$; otherwise, we move up to level 3 and update $z_3$.  The procedure continues until we reach level $n$ and (\ref{eq:ineq_n}) is not satisfied.  The last full integer point found is the OILS solution. 
The described search process is a depth-first search, see an example of a search tree in Fig.\ \ref{fig:searchtree}.  The root node
does not correspond to an element in $\z$, but is used to unite the branches into a tree.
 Note that the integers are enumerated according to order (\ref{eq:seq}) when we move up one level in the search.     
 If the initial value of $\beta$ is 
$\infty$, the first integer point found is called the Babai integer point. 
\begin{figure}
\centering
{\includegraphics[trim = 10mm 140mm 60mm 5mm, clip, scale = 0.5]{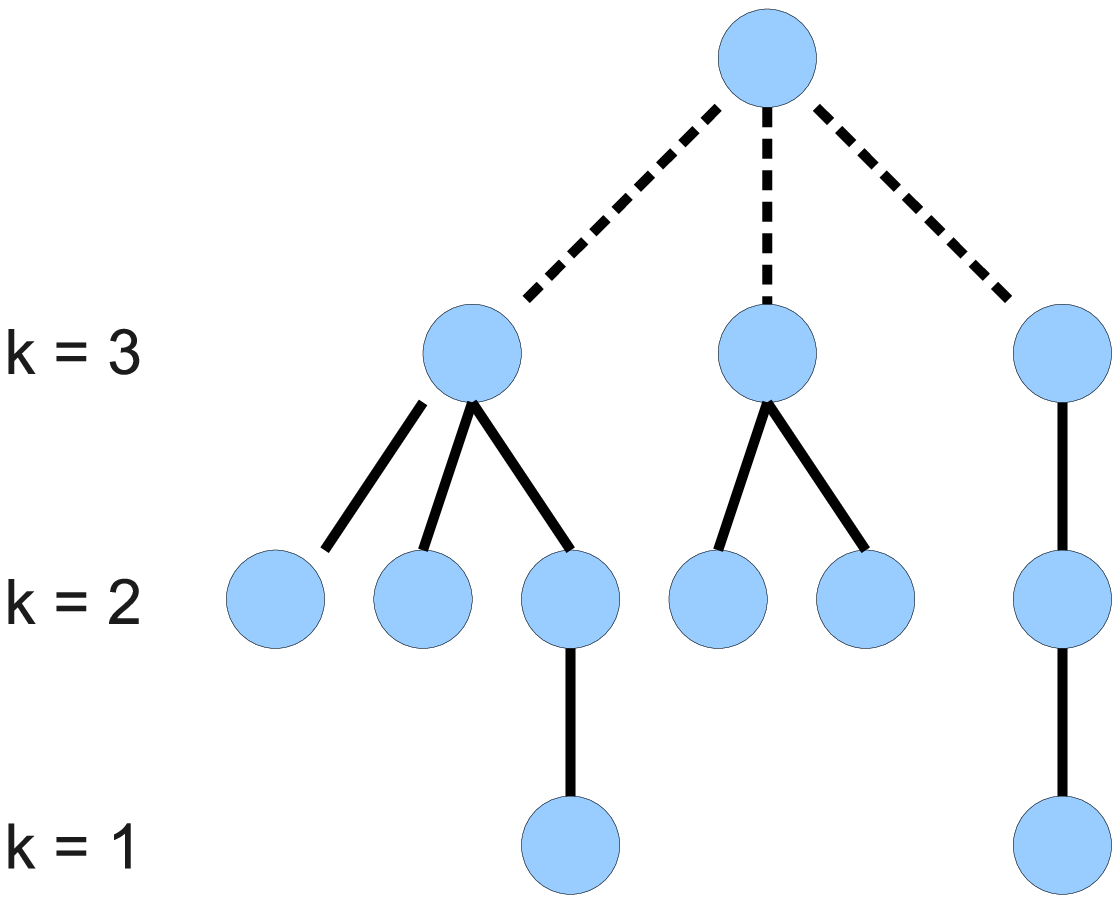}}
\caption{A three-dimensional example of a search tree}
\label{fig:searchtree}
\end{figure}
We describe the process as an algorithm, see \cite{Cha05}.
\begin{algorithm}
\textnormal{(SEARCH)
Given nonsingular upper triangular matrix $\R \in \mathbb{R}^{n \times n}$ and $\by \in \mathbb{R}^{n}$.  The Schnorr-Euchner search algorithm finds the optimal solution to $\min_{\z \in \mathbb{Z}^n} \| \by-\R \z \|_2^2$.
\begin{tabbing}
$\quad$ \= {\bf function}: $\z = \mathrm{SEARCH}(\R, \by)$ \\
\> 1. (Initialization) Set $k=n, \beta = \infty$. \\
\> 2. \= Compute $c_k$ from ($\ref{eq:c_k}$). Set $z_k = \round{c_k}$, $\Delta_k = \mbox{sgn}(c_k - z_k)$ \\
\> 3. (Main step) \\
\>   \> {\bf if } \= $r_{kk}^2 (z_k - c_k)^2 > \beta^2 - \sum_{i=k+1}^n r_{ii}^2 (z_i - c_i)^2$ \\
\>   \>    \>    go to Step 4 \\ 
\>   \> {\bf else if}  $k > 1$ \\ 
\>   \>    \>  $k=k-1$, go to Step 2  \\ 
\>   \>  {\bf else} $\quad$ // case k = 1 \\  
\>   \>     \>  go to Step 5 \\
\>   \>  {\bf end} \\ 
\> 4. (Invalid point) \\
\>    \> {\bf if} \= $k = n $ \\
\>    \>  \>    terminate \\
\>    \> {\bf else } \\
\>    \>    \> $k = k+1$, go to Step 6 \\
\>    \> {\bf end} \\	
\> 5. (Found valid point) \\  
\>   \> Set $\hbz = \z$, $\beta^2 = \sum_{k=1}^n r_{kk}^2 (\hat{z}_k - c_k)^2$  \\
\>    \> $k = k + 1$, go to Step 6 \\   
\> 6. (Enumeration at level k) \\
\>  \> Set $z_k = z_k + \Delta_k$, $\Delta_k = -\Delta_k - \mbox{sgn}(\Delta_k)$ \\
\>  \>  go to Step 3. 
\end{tabbing}
}
\end{algorithm}

\section{Reduction process}\label{s:reduction}

In the reduction process, we transform $\A$ to an upper triangular matrix $\R$ which has properties that make the search process more efficient.
Since (\ref{eq:ILS}) uses the 2-norm, we can apply an orthogonal transformation $\Q^T$ to the left of $\A$ as long as it is also applied to the left of $\y$, i.e., $\|\y - \A \x \|^2_2 = \|\Q^T \y - \Q^T \A \x \|^2_2$. In order to maintain the integer nature of $\x$, the transformation $\Z$ applied to the right of $\A$ must be an integer matrix, whose inverse is also an integer matrix.  It is easy to verify that $|\mbox{det}(\Z)| = 1$, as both 
$\mbox{det}(\Z)$ and $\mbox{det}(\Z^{-1})$ are integers, and $\mbox{det}(\Z)\mbox{det}(\Z^{-1}) = 1$.  Such integer matrices are referred to as unimodular matrices.  

The transformations on $\A$ can be described as a QRZ factorization of $\A$:
\begin{equation}\label{eq:QRZ1}
\Q^T \A \Z = 
\begin{bmatrix}
\R \\
{\bf 0} 
\end{bmatrix}  \quad
\mbox{or } \A = \Q_{1}^T \R \Z^{-1}, 
\end{equation}
where $\Q = [\Q_1, \Q_2] \in \mathbb{R}^{m \times m}$ is orthogonal, $\R \in \mathbb{R}^{n}$ is nonsingular upper triangular and 
$\Z \in \mathbb{Z}^{n \times n}$ is unimodular (see \cite{ChaG09}). Using this factorization,
\begin{equation}\label{eq:QRZ2}
\| \y - \A \x \|_2^2 = \| \Q_1^T \y - \R \Z^{-1} \x \|_2^2 + \| \Q_2^{T} \y \|_2^2. 
\end{equation}
Let 
\begin{equation}\label{eq:QRZ3}
\by = \Q_1^T \y, \quad \z = \Z^{-1} \x.
\end{equation}
Then, we can rewrite (\ref{eq:ILS}) as
\begin{equation}\label{eq:rrILS}
\min_{\z \in \mathbb{Z}^n} \| \by-\R \z \|_2^2.
\end{equation}
If $\hbz$ is the solution of the transformed OILS problem (\ref{eq:rrILS}), then $\hbx = \Z \hbz$ is the solution of the original OILS problem (\ref{eq:ILS}). 

Note that $\R$ in (\ref{eq:QRZ1}) is not unique.  A different $\Z$ usually leads to a different $\R$.  Without loss of generality, we assume that the diagonal entries of $\R$ are positive in this thesis. 
Certain properties of $\R$ can make the search process much more efficient. 
If $\R$ is diagonal, then simply rounding all $c_k$ to the nearest integer gives the optimal solution. In \cite{MurGDC06}, it is claimed that as $\R$
gets closer to a diagonal matrix, the complexity of the search process decreases.  In Section \ref{s:proof}, we show that this claim is not true.  The crucial property that $\R$ must strive for is that $r_{kk}$ should be as large as possible for large $k$.  
We motivate this property by a two dimensional case ($n = 2$). Let $r_{22} \ll r_{11}$.  In the search process, the bound (\ref{eq:ineq_n}) at level 2 is very loose, implying that there are many valid integers $z_2$. However, the bound (\ref{eq:ineq_1}) at level 1 is very tight. Hence, after $z_2$ is fixed, there is a high probability that no valid integer $z_1$ exists.  We must enumerate many integers at level 2 before we can find a valid integer at level 1.  This is the so-called search halting problem (see \cite{Teu95}).  
Note that $\mbox{det}(\A^T\A)= \mbox{det}(\R^T\R) = r_{11}^2 \ldots r_{nn}^2$ is constant, independent of the choice of $\Z$.
Making $r_{kk}$ as large as possible for large $k$ implies making $r_{kk}$ as small as possible for small $k$.   
Hence, we can say that the reduction process should strive for
\begin{equation}\label{eq:rgoal}
r_{11} \ll \ldots \ll r_{nn}.
\end{equation}

The typical reduction method for the OILS problem is the LLL reduction (see \cite{LenLL82} and \cite{HasB98}).  The LLL reduction can be written in the form of a QRZ factorization, where $\R$ satisfies the following criteria
\begin{align}
 |r_{k-1,j}|  & \le \frac{1}{2} r_{k-1,k-1}, \quad  j = k, \ldots, n  \label{eq:LLLprop1}\\
 r_{k-1,k-1} & \le \delta \sqrt{r_{k-1,k}^2 + r_{kk}^2}, \quad 1 \le \delta < 2 \quad \mbox{and} \quad k = 2, \ldots, n. \label{eq:LLLprop2}
\end{align}
Notice that taking $\delta = 1$ is best to strive for (\ref{eq:rgoal}).  In this thesis, we always take $\delta = 1$. The LLL reduction cannot guarantee $r_{11} \le \ldots \le r_{nn}$,  but substituting (\ref{eq:LLLprop1}) into (\ref{eq:LLLprop2}), we see that it can guarantee
\begin{equation*}
r_{k-1,k-1} \le \frac{2}{\sqrt{3}} r_{kk}, \quad k = 2, \ldots, n.
\end{equation*} 

In our implementation of the LLL reduction, we use two types of unimodular transformations to ensure properties (\ref{eq:LLLprop1}) and (\ref{eq:LLLprop2}). 
These are integer Gauss transformations and permutation matrices.  In the following, we present the effect of these transformations on the R factor of the QR factorization of $\A$.   

\subsection{Integer Gauss transformations}\label{s:igt}

Suppose we are given an upper triangular matrix $\R$ with positive diagonal entries.
An integer Gauss transformation (IGT) $\Z_{ij}$ has the following form
\begin{equation}\label{eq:Zij}
\Z_{ij} = \I - \mu \e_i\e_j^T, \qquad \mbox{$\mu$ is an integer}.
\end{equation}
Applying $\Z_{ij}$ ($i<j$) to $\R$ from the right gives
$$
\bbR = \R\Z_{ij}=\R-\mu \R\e_i\e_j^T.
$$
Thus $\bbR$ is the same as $\R$, except that
$$
\bar{r}_{kj}=r_{kj}-\mu r_{ki}, \quad k=1,\ldots,i.
$$
Taking $\mu=\lfloor r_{ij}/r_{ii} \rceil$ ensures that $|\bar{r}_{ij}|\leq \frac{1}{2} r_{ii}$.
Similarly, an IGT $\Z_{ij}$ ($i>j$) can be applied to a unit lower triangular matrix $\L$ from the right to ensure that $|\bar{l}_{ij}| \leq 1/2$.

\subsection{Permutations}\label{s:permutations}
For $\R$ to satisfy (\ref{eq:LLLprop2}) with $\delta=1$, we sometimes need to permute its columns.  In the reduction process, if
$r_{k-1,k-1} > \sqrt{r^2_{k-1,k} + r_{kk}^2},$ we permute columns $k$ and $k-1$.  

\begin{equation*}
\R \P_{k-1,k} =
\begin{bmatrix}
\R_{11} & \bbR_{12} & \R_{13} \\
        & \tR_{22}  & \R_{23}  \\
        &           & \R_{33}     
\end{bmatrix},
\end{equation*}
where
\begin{equation*}
\P_{k-1,k} =
\begin{bmatrix}
\I_{k-2} &    &        \\
         & \P &        \\
         &    & \I_{n-k} 
\end{bmatrix}, \quad \P =
\begin{bmatrix}
0 & 1 \\
1 & 0
\end{bmatrix}, \quad \tR_{22} =
\begin{bmatrix}
r_{k-1,k} & r_{k-1,k-1} \\
r_{kk}    & 0
\end{bmatrix},
\end{equation*}
\begin{equation*}
\bbR_{12} = [\R(1:k-2,k-1), \R(1:k-2,k)].
\end{equation*} 
As a result, $\R$ is no longer upper triangular.  To make it upper triangular again, we apply a Givens rotation $\G$ to zero element $r_{kk}$ in $\tR_{22}$.

\begin{equation*}
\G \tR_{22} = \bbR_{22}, \quad \mbox{or }
\begin{bmatrix}
c & s \\
-s & c
\end{bmatrix}
\begin{bmatrix}
r_{k-1,k} & r_{k-1,k-1} \\
r_{kk}    & 0
\end{bmatrix} =
\begin{bmatrix}
\bar{r}_{k-1,k-1} & \bar{r}_{k-1,k} \\
0                 &   \bar{r}_{kk}
\end{bmatrix},
\end{equation*}
where 
\begin{equation}\label{eq:r_k-1}
\bar{r}_{k-1,k-1} = \sqrt{r_{k-1,k}^2 + r_{kk}^2}, \quad c = \frac{r_{k-1,k}}{\bar{r}_{k-1,k-1}}, \quad s = \frac{r_{kk}}{\bar{r}_{k-1,k-1}}
\end{equation}
\begin{equation}\label{eq:r_kk}
\bar{r}_{k-1,k} = c r_{k-1,k-1}, \quad \bar{r}_{kk} = -s r_{k-1,k-1}.
\end{equation}
Therefore, we have
\begin{equation*}
\Q_{k-1,k} \R \P_{k-1,k} = \bbR =
\begin{bmatrix}
\R_{11} & \bbR_{12} & \R_{13} \\
        & \bbR_{22}  & \bbR_{23}  \\
        &           & \R_{33}     
\end{bmatrix}, \quad \Q_{k-1,k} =
\begin{bmatrix}
\I_{k-2} &    &        \\
         & \G &        \\
         &    & \I_{n-k} 
\end{bmatrix},  
\end{equation*}
\begin{equation*}
\bbR_{23} = \G \R_{23}.
\end{equation*}
After the permutation, inequality $\bar{r}_{k-1,k-1} \le \sqrt{\bar{r}^2_{k-1,k} + \bar{r}_{kk}^2}$ now holds.
While a permutation does not guarantee $\bar{r}_{k-1,k-1} \le \bar{r}_{kk}$, it does guarantee that
\begin{equation*}
\bar{r}_{k-1,k-1} < r_{k-1,k-1} \quad \mbox{ and } \quad \bar{r}_{kk} > r_{kk}.
\end{equation*}  Such a permutation is useful since the diagonal elements of $\R$ are now closer to (\ref{eq:rgoal}).    

\subsection{LLL reduction}\label{s:LLL}
 
The LLL reduction algorithm given in \cite{ChaG09} starts by finding the QR decomposition of $\A$ by Householder transformations, then computes $\by$ and works with $\R$ from left to right. At the $k$th column of $\R$, the algorithm applies IGTs to ensure that $|r_{ik}| < \frac{1}{2} r_{ii}$ for $i=k-1\!:\!-1\!:\!1$.  Then, if inequality (\ref{eq:LLLprop2}) holds, it moves to column $k+1$; otherwise it permutes columns $k$ and $k-1$, applies a Givens rotation to $\R$ from the left to bring $\R$ back to an upper triangular form, simultaneously applies the same Givens rotation to $\by$, and moves to column $k-1$.  We describe our implementation of the LLL reduction as follows (see \cite{ChaG09}).

\begin{algorithm} \label{a:decor}
\textnormal{(LLL Reduction).
Given the generator matrix $\A \in \mathbb{R}^{m \times n}$ and the input vector $\y \in \mathbb{R}^{m}$.
The algorithm returns the reduced upper triangular matrix $\R \in \mathbb{R}^{n \times n}$, the 
unimodular matrix $\Z \in \mathbb{Z}^{n \times n}$, and the vector $\by \in \mathbb{R}^{n}$. 
\begin{tabbing}
$\quad$ \= {\bf function}: $[\R, \Z,\by] = \mathrm{LLL}(\A, \y)$ \\
\> Compute QR factorization of $\A$ and set $\by = \Q_{1}^T\y$ \\
\> $\Z=\I$ \\
\> $k=2$ \\
\> {\bf while} \= $k \le n$  \\
\>            \>	{\bf for} \= $i=k-1:-1:1$ \\
\>            \>        \> \mbox{Apply IGT $\Z_{ik}$ to $\R$, .i.e., $\R=\R \Z_{ik}$} \\
\>            \>        \> \mbox{Update $\Z$, i.e., $\Z= \Z\Z_{ik}$} \\
\>	      \>        {\bf end} \\
\>            \> {\bf if}   $r_{k-1,k-1} > \sqrt{r_{k-1,k}^2 + r_{kk}^2}$ \\
\>	      \>                     \> \mbox{Interchange columns $k$ and $k-1$ of $\R$ and $\Z$} \\ 
\>	      \>	             \> \mbox{Transform $\R$ to an upper triangular matrix by a Givens rotation} \\
\>	      \>                     \> \mbox{Apply the same Givens rotation to $\by$ } \\
\>	      \>		     \> {\bf if} \= $k > 2$ \\
\>            \>                     \> 	\> $k = k - 1$ \\
\>	      \>		     \> {\bf end} \\
\>            \> {\bf else}\\
\>            \>         \> $k=k+1$ \\
\>            \> {\bf end} \\
\> {\bf end}
\end{tabbing}
}
\end{algorithm}

\chapter{The OILS problem in the quadratic form}\label{s:GPS}

A prerequisite for high precision relative GNSS positioning is
to resolve the unknown double differenced cycle ambiguities of the carrier phase data as integers. 
This turns out to be an OILS problem. 
Suppose $\hbx \in \mathbb{R}^n$ is the real-valued least squares estimate of the
integer parameter vector $\x \in \mathbb{Z}^n$ (i.e., the double differenced integer ambiguity vector
in the GNSS context)
and $\W_{\hbx}\in \mathbb{R}^{n\times n}$ is its covariance matrix,
which is symmetric positive definite.
The OILS estimate $\cbx$ is the solution of
the minimization problem:
\begin{equation}\label{eq:qils}
\min_{\x\in \mathbb{Z}^n} (\x-\hbx)^T\W_{\hbx}^{-1}(\x-\hbx).
\end{equation}

Although (\ref{eq:qils}) is in the form of an integer quadratic optimization problem, it is easy to rewrite it in the standard OILS form (\ref{eq:ILS}). 
We refer to (\ref{eq:qils}) as the quadratic form of the OILS problem.  
In Section \ref{s:reductionGPS}, we discuss the reduction process used in the GNSS literature.  In Section \ref{s:lambda},  we review the reduction stage in the LAMBDA (Least-squares AMBiguity Decorrelation Adjustment) method (e.g., \cite{Teu93, Teu95b, Teu95,  Teu98, Teu99}).  In Section \ref{s:mlambda}, we introduce the improvements to the reduction provided by the MLAMBDA  (Modified LAMBDA) method (see \cite{ChaYZ05}).  Finally in Section \ref{s:searchGPS}, we briefly review the search process in the quadratic form of the OILS problem.  

\section{Reduction process}\label{s:reductionGPS}

The reduction step uses a unimodular matrix $\Z$ to transform \eqref{eq:qils} into
\begin{equation}\label{eq:fobjGPS}   
\min_{\z\in \mathbb{Z}^n} (\z-\hbz)^T\W_{\hbz}^{-1}(\z-\hbz),
\end{equation}
where $\z = \Z^T \x$, $\hbz = \Z^T \hbx$ and $ \W_{\hbz} = \Z^T \W_{\hbx} \Z$.  
If $\cbz$ is the integer minimizer of (\ref{eq:fobjGPS}),
then $\cbx = \Z^{-T} \cbz$ is the integer minimizer of (\ref{eq:qils}).
The benefit of the reduction step is that the search in the new optimization problem (\ref{eq:fobjGPS}) can be much more efficient. 
If $\W_{\hbz}$ is a diagonal matrix, then the transformed ambiguities $z_1,\ldots,z_n$  are uncorrelated to each other.  In this case, simply setting $z_i = \round{\hat{z}_i}$, for $i=1:n$, would minimize the objective function.

Let the $\mathrm{\L^T\D\L}$ factorization of $\W_{\hbz}$ be
\be
\W_{\hbz} = \L^T\D\L,
\label{eq:qazldl}
\ee
where $\L$ is unit lower triangular and
$\D=\diag(d_1,\ldots,d_n)$ with
$d_i>0$.
These factors have a statistical interpretation. 
Let $\bz_{i}$ denote the least-squares estimate of $z_i$ when $z_{i+1},\ldots,z_n$ are fixed.
As shown in \cite[p.\ 337]{Teu98},  $d_i$ is the variance of $\bz_{i}$, which is denoted by $\sigma_{\bz_{i}}^{2}$.  Furthermore, $l_{ij} = \sigma_{\hat{z}_i \bz_{j}} \sigma_{\bz_{j}}^{-2}$ for $i > j$, where  $\sigma_{\hat{z}_i \bz_{j}}$ denotes the covariance between $\hat{z}_i$ and $\bz_{j}$.

In the literature (see, e.g., \cite{JonT96}, \cite[p.\ 498]{StrB97} and \cite[p.\ 369]{Teu98}), it is often mentionned that the following two goals should be pursued in the reduction process because they are crucial for the efficiency of the search process:
\begin{enumerate}
\item[(i)] $\W_{\hbz}$ is as diagonal as possible.   From (\ref{eq:qazldl}),
for $i \ne j$, making $\L(i+1\!:\!n,i)$ and $\L(j+1\!:\!n,j)$ closer to $0$ makes $\W_{\hbz}(i,j)$ closer to $0$.
Hence, making the absolute values of the off-diagonal entries of $\L$ as small as possible makes $\W_{\hbz}$ as diagonal as possible.  A covariance matrix which is close to diagonal means that there is little correlation between its random variables.  
In other words, the goal of the reduction is to decorrelate the ambiguities as much as possible.
\item[(ii)] The diagonal entries of $\D$ are distributed in
decreasing order if possible, i.e., one strives for
\begin{equation}\label{eq:order}
d_1 \gg d_{2} \gg \cdots \gg d_n.
\end{equation}
\end{enumerate}  

Note that we want $d_1$ to be as large as possible and $d_n$ to be as small as possible. 
We can show that striving for \eqref{eq:order} is equivalent to striving for \eqref{eq:rgoal}, where $d_{i}$ corresponds to $r_{ii}^{-2}$ for $i=1\!:\!n$.
In the reduction process of the LAMBDA method, the unimodular matrix $\Z$ is constructed by a sequence of integer Gauss transformations and permutations. 
The reduction process starts with the $\mathrm{L^TDL}$ factorization of $\W_{\hbx}$
and updates the factors to give the $\mathrm{L^TDL}$ factorization of  $\W_{\hbz}$.
The main contribution of this thesis is to show that, contrary to common belief,  the first goal will not make the search process more efficient.
While lower triangular integer Gauss transformations are used to make the absolute values of the off-diagonal entries of $\L$ as small as possible, we argue that they are 
useful because they help achieve the second goal.

In \cite{LiuHZO99}, \cite{LouG03} and \cite{Xu01}, instead of (i) and (ii), the condition number of $\W_{\hbz}$ is used to evaluate the reduction process.  In Section \ref{s:conditionNumber}, we show that this criterion can be misleading and that it is not as effective as (ii).         

\subsection{Integer Gauss transformations}\label{s:GPSigt}

Integer Gauss transformations (IGTs) were first introduced in Section \ref{s:igt}.  
We apply $\Z_{ij}$ with $\mu=\lfloor l_{ij} \rceil$  (see (\ref{eq:Zij})) to $\L$ from the right, i.e., $\bbL = \L \Z_{ij}$, 
to make $|\bl_{ij}|$ as small as possible.  This ensures that
\be
|\bl_{ij}|\leq 1/2, \quad i>j.
\label{eq:inequality}
\ee
We use the following algorithm to apply the IGT $\Z_{ij}$ to transform the OILS problem (see \cite{ChaYZ05}).
\begin{algorithm}\label{a:gauss}
\textnormal{(Integer Gauss Transformations).
Given a  unit lower triangular $\L\in \mathbb{R}^\nbn$, index pair $(i,j)$,
$\hbx\in \mathbb{R}^n$ and $\Z\in \mathbb{Z}^\nbn$.
This algorithm applies the integer Gauss transformation $\Z_{ij}$ to $\L$ such that
$|(\L\Z)(i,j)|\leq 1/2$,
then computes $\Z_{ij}^T\hbx$ and $\Z\Z_{ij}$, which overwrite $\hbx$ and $\Z$, respectively.
\begin{tabbing}
$\quad$ \= {\bf function}: $[\L, \hbx, \Z] = \mathrm{GAUSS}(\L, i, j, \hbx, \Z)$ \\
\> $\mu=\lfloor \L(i,j)\rceil$ \\
\> {\bf if} $\mu$ \= $\neq 0$ \\
\> \> $\L(i:n,j)=\L(i:n,j)-\mu \L(i:n,i)$ \\
\> \> $\Z(1:n,j)=\Z(1:n,j)-\mu \Z(1:n,i)$  \\
\> \> $\hat{\x}(j)=\hat{\x}(j)-\mu\hat{\x}(i)$ \\
\> {\bf end} 
\end{tabbing}
}
\end{algorithm}

\subsection{Permutations}\label{s:reorder}
In order to strive for order (\ref{eq:order}), symmetric permutations of the covariance matrix $\W_{\hbx}$ are needed in the reduction.
After a permutation, the factors $\L$ and $\D$ of the $\mathrm{L^TDL}$ factorization have to be updated.

If we partition the $\mathrm{L^TDL}$ factorization of $\W_{\hbx}$ as follows
$$
\W_{\hbx}
=\L^T\D\L
=\bmx \L_{11}^T &  \L_{21}^T & \L_{31}^T \\  & \L_{22}^T & \L_{32}^T \\ & & \L_{33}^T \emx
\bmx \D_1 & \\ & \D_2 & \\ & & \D_3 \emx
\underset{\;\;\;\; k-1 \quad \;\;  2 \;\; \quad n-k-1}
{\bmx \L_{11} \\ \L_{21} & \L_{22} \\ \L_{31} & \L_{32} & \L_{33} \emx}
\hspace{-1mm}\small{\ba{l} k\!-\!1 \\ 2  \\ \small{n\!-\!k\!-\!1} \ea}.
$$
Let
$$
\P=\bmx  0 & 1 \\ 1 & 0\emx, \quad
\P_{k,k+1}=\bmx \I_{k-1} & \\ & \P & \\ & & \I_{n-k-1} \emx .
$$
It can be shown that $\P_{k,k+1}^T\W_{\hbx}\P_{k,k+1}$
has the $\mathrm{L^TDL}$ factorization
\be
\P_{k,k+1}^T\W_{\hbx}\P_{k,k+1}
= \bmx \L_{11}^T & \bbL_{21}^T & \L_{31}^T \\ & \bbL_{22}^T & \bbL_{32}^T \\ & & \L_{33}^T \emx
\bmx \D_1 & \\ & \bbD_2 & \\ & & \D_3 \emx
\bmx\L_{11} & \\ \bbL_{21} & \bbL_{22} & \\ \L_{31} & \bbL_{32} & \L_{33} \emx,
\label{eq:zszldl}
\ee
where
\begin{align}
& \bbD_2=\bmx  \bd_{k} & \\ & \bd_{k+1} \emx, \quad
  \bd_{k+1}= d_{k}+l_{k+1,k}^2 d_{k+1}, \quad \bd_{k} = \frac{d_{k}}{\bd_{k+1}}d_{k+1},
  \label{eq:D2} \\
& \bbL_{22} \equiv \bmx  1 & \\ \bl_{k+1,k} & 1 \emx, \quad
  \bl_{k+1,k} = \frac{d_{k+1}l_{k+1,k}}{\bd_{k+1}},
  \label{eq:L22} \\
& \bbL_{21}=\bmx  -l_{k+1,k} & 1 \\  \frac{d_k}{\bd_{k+1}} & \bl_{k+1,k}\emx \L_{21}
           =\bmx  -l_{k+1,k} & 1 \\  \frac{d_k}{\bd_{k+1}} & \bl_{k+1,k}\emx \L(k\!:\!k+1,1\!:\!k-1),
  \label{eq:L21} \\
& \bbL_{32}= \L_{32}\P=\bmx \L(k+2\!:\!n,k+1) & \L(k+2\!:\!n,1\!:\!k)\emx.
  \label{eq:L32}
\end{align}
We refer to such an operation as a permutation between pair ($k,k+1$).
We describe the process as an algorithm (see \cite{ChaYZ05}).
\begin{algorithm}\label{a:permute}
\textnormal{(Permutations).
Given the $\L$ and $\D$ factors of the $\mathrm{L^TDL}$ factorization of $\W_{\hbx}\in \mathbb{R}^\nbn$,
index $k$, scalar $\delta$ which is equal to $\bd_{k+1}$ in (\ref{eq:D2}),
$\hat \x \in  \mathbb{R}^n$, and  $\Z\in\mathbb{Z}^\nbn$.
This algorithm computes the updated $\L$ and $\D$ factors in (\ref{eq:zszldl})
after $\W_{\hbx}$'s $k$th row and $(k+1)$th row, and $k$th column and $(k+1)$th column are interchanged, respectively.
It also interchanges $\hat \x$'s $k$th entry and $(k+1)$th entry
and $\Z$'s $k$th column and $(k+1)$th column.
\begin{tabbing}
$\quad$ \= {\bf function}: $[\L, \D, \hbx, \Z] = \mbox{PERMUTE}(\L, \D, k, \delta, \hbx, \Z)$ \\
\> $\eta=\D(k,k)/\delta$ \qquad // see (\ref{eq:D2})  \\
\> $\lambda=\D(k+1,k+1)\L(k+1,k)/\delta$ \qquad // see (\ref{eq:L22})  \\
\> $\D(k,k)=\eta \D(k+1,k+1)$ \qquad // see (\ref{eq:D2})  \\
\> $\D(k+1,k+1)=\delta$ \\
\> $\L(k\!:\!k+1,1\!:\!k-1) = \bmx -\L(k+1,k) & 1 \\ \eta & \lambda \emx \L(k\!:\!k+1,1\!:\!k-1)$
          \qquad // see (\ref{eq:L21})   \\
\> $\L(k+1,k)=\lambda$ \\
\> swap columns $\L(k+2\!:\!n,k)$ and $\L(k+2\!:\!n,k+1)$ \qquad // see (\ref{eq:L32})  \\
\> swap columns $\Z(1\!:\!n,k)$ and $\Z(1\!:\!n,k+1)$ \\
\> swap entries $\hbx(k)$ and $\hbx(k+1)$ 
\end{tabbing}
}
\end{algorithm}

\section{LAMBDA reduction}\label{s:lambda}

We now describe the reduction process of the LAMBDA method (see \cite[Sect.\ 3]{JonT96}).
First, it computes the $\mathrm{L^TDL}$ factorization of $\W_{\hbx}$. 
The algorithm starts with column $n-1$ of $\L$.  At column $k$, IGTs are applied to ensure that the absolute values of the entries below the ($k,k$)th entry are as small as possible. 
Then, if $\bd_{k+1} \ge d_{k+1}$ holds (see (\ref{eq:D2})), it moves to column $k-1$; otherwise it permutes pair ($k,k+1$) and moves 
back to the initial position $k=n-1$.
The algorithm uses a variable ($k1$ in Algorithm \ref{a:decor} below) to track down the columns whose off-diagonal entries in magnitude are already bounded above by 1/2 due to previous integer Gauss transformations. We present the implementation given in \cite{ChaYZ05}.

\begin{algorithm} \label{a:decor}
\textnormal{(LAMBDA REDUCTION).
Given the covariance matrix $\W_{\hbx}$ and real-valued LS estimate $\hbx$ of $\x$.
This algorithm computes an integer unimodular matrix $\Z$ and the $\mathrm{L^TDL}$ factorization
$\W_{\hbz}=\Z^T\W_{\hbx}\Z=\L^T\D\L$, where
$\L$ and  $\D$ are updated from the factors of the $\mathrm{L^TDL}$ factorization of $\W_{\hbx}$.
This algorithm also computes $\hbz=\Z^T\hbx$, which overwrites $\hbx$.
\begin{tabbing}
$\quad$\= {\bf function}: $[\Z, \L,\D,\hbx] = \mathrm{REDUCTION}(\W_{\hbx}, \hbx)$ \\
\> Compute the $\mathrm{L^TDL}$ factorization of $\W_{\hbx}$: $\W_{\hbx}=\L^T\D\L$ \\
\> $\Z=\I$ \\
\> $k=n-1$ \\
\> $k1 = k$ \\
\> {\bf while} \= $k>0$  \\
\>	      \> {\bf if} \= $k \le k1$ \\ 
\>            \>	\>  {\bf for} \= $i=k+1\!:\!n$ \\
\>            \>        	\>	\> $[\L, \hbx, \Z] = \mbox{GAUSS}(\L, i, k, \hbx, \Z)$ \\
\>	      \>	\>  {\bf end} \\
\>            \>  {\bf end} \\
\>            \> $\bbD(k+1,k+1)=\D(k,k)+\L(k+1,k)^2\D(k+1,k+1)$  \\
\>            \> {\bf if}   $\bbD$\=$(k+1,k+1)<\D(k+1,k+1)$ \\
\>            \>                     \>   $[\L, \D, \hbx, \Z] = \mbox{PERMUTE}(\L, \D, k, \bbD(k+1,k+1), \hbx, \Z)$ \\
\>	      \>		     \> $k1 = k$ \\
\>	      \>		     \> $k = n-1 $ \\
\>            \> {\bf else}\\
\>            \>         \> $k=k-1$ \\
\>            \> {\bf end} \\
\> {\bf end}
\end{tabbing}
}
\end{algorithm}

When the reduction process is finished, we have 
\begin{align}
|l_{kj}| & \le 1/2, \quad j = 1, \ldots, k-1, \label{eq:lambdag1} \\
d_{k+1}  & \le d_k + l_{k+1,k}^2 d_{k+1},  \quad k = 1,2, \ldots, n-1. \label{eq:lambdag2}
\end{align}
Note that (\ref{eq:lambdag1}) and (\ref{eq:lambdag2}) are the lower-triangular equivalent of the LLL reduction properties (\ref{eq:LLLprop1}) and (\ref{eq:LLLprop2}) with $\delta=1$, respectively. 

\section{MLAMBDA reduction}\label{s:mlambda}

In the MLAMBDA method, several strategies are proposed to reduce the computational complexity of the reduction process.     

\subsection{Symmetric pivoting strategy}\label{s:sp}
In striving for (\ref{eq:order}), LAMBDA reduction performs symmetric permutations of the covariance matrix $\W_{\hbx}$. 
After each permutation, an IGT is needed to update the $\L$ and $\D$ factors of the $\mathrm{L^TDL}$ factorization of $\W_{\hbx}$.
One idea is to apply permutations to $\W_{\hbx}$ before computing its $\mathrm{L^TDL}$ factorization.   
This way, the cost of the IGT associated with a permutation is saved.
Once the $\mathrm{L^TDL}$ factorization is computed, new permutations are usually needed to strive for (\ref{eq:order}).
Nevertheless, this strategy usually reduces the number of permutations done after we have the $\L$ and $\D$ factors. 
First, we show how to compute the $\mathrm{L^TDL}$ of $\W_{\hbx}$ without pivoting.  We partition the $\W_{\hbx} = \L^T \D \L$ as follows
\begin{equation*}
\begin{bmatrix}
\tilde{\W}_{\hbx} & \q \\
\q^T & \q_{nn}
\end{bmatrix} =
\begin{bmatrix}
\tilde{\L}^T & \l \\
             & 1
\end{bmatrix}
\begin{bmatrix}
\tilde{\D} & \\
      & d_n
\end{bmatrix}
\begin{bmatrix}
\tilde{\L} & \\
\l^T & 1
\end{bmatrix}.
\end{equation*}  
We can see that
\begin{equation}\label{eq:ldl}
d_n = q_{nn}, \quad \l = \q/d_n, \quad \tilde{\W}_{\hbx} - \l d_n \l^T = \tilde{\L}^T \tilde{\D} \tilde{\L}.
\end{equation}
We recurse on $\tilde{\W}_{\hbx} - \l d_n \l^T$ to find the complete factorization. 
Now we introduce the symmetric pivoting strategy.  Since we strive for (\ref{eq:order}), we first symmetrically permute the smallest diagonal entry of 
$\W_{\hbx}$ to position $(n,n)$.  With (\ref{eq:ldl}), we compute $d_n$ and $\l$. We continue this procedure with $\tilde{\W}_{\hbx} - \l d_n \l^T$. Finally, we get the 
$\mathrm{L^TDL}$ factorization of a permuted $\W_{\hbx}$.
The implementation of this strategy is described in \cite{ChaYZ05}.
\begin{algorithm}\label{a:sympiv}
\textnormal{($\mathrm{L^TDL}$ factorization with symmetric pivoting).
Suppose $\W_{\hbx} \in \mathbb{R}^\nbn$ is symmetric positive definite.
This algorithm computes a permutation $\P$, a unit lower triangular matrix $\L$
and a diagonal $\D$ such that $\P^T\W_{\hbx}\P=\L^T\D\L$.
The strict lower triangular part of $\W_{\hbx}$ is overwritten by that of $\L$
and the diagonal part of $\W_{\hbx}$ is overwritten by that of $\D$.
\begin{tabbing}
$\quad$ \= $\P=\I_n$ \\
        \> {\bf for}  \= $k=n\!:\!-1\!:\!1$ \\
        \>            \> $q=\arg\min_{1\leq j\leq k} \W_{\hbx}(j,j)$ \\
        \>            \> swap $\P(:,k)$ and $\P(:,q)$ \\
        \>            \> swap $\W_{\hbx}(k,:)$ and $\W_{\hbx}(q,:)$ \\
        \>            \> swap $\W_{\hbx}(:,k)$ and $\W_{\hbx}(:,q)$ \\
        \>            \> $\W_{\hbx}(k,1\!:\!k-1)=\W_{\hbx}(k,1\!:\!k-1)/\W_{\hbx}(k,k)$ \\
        \>            \> $\W_{\hbx}(1\!:\!k\!-\!1,1\!:\!k\!-\!1)=$ \= $\W_{\hbx}(1\!:\!k\!-\!1,1\!:\!k\!-\!1)-\W_{\hbx}(k,1\!:\!k\!-\!1)^T*$\\
        \>            \>   	\> $\W_{\hbx}(k,k)*\W_{\hbx}(k,1\!:\!k\!-\!1)$ \\
\> {\bf end}
\end{tabbing}
}
\end{algorithm}

\subsection{Greedy selection strategy}\label{s:gs}
The reduction process starts with the $\mathrm{L^TDL}$ factorization with pivoting.  In order to further reduce the number of permutations, a greedy selection strategy is proposed.  As shown in Section \ref{s:lambda}, the reduction process of the LAMBDA method permute pairs ($k,k+1$) from right to left.  If for some index $k$, we have $d_{k+1} \gg d_{k}$ and $\bar{d}_{k+1} < d_{k+1}$, then we permute pair ($k,k+1$) and we move to column $k+1$ of $\L$.  Now, it is likely that we also have to permute pair ($k+1,k+2$), and so on.  As a result, it is possible that some of the permutations done before reaching index $k$ are wasted.       
To avoid these unnecessary permutations, instead of looping $k$ from $n-1$ to 1 as in Algorithm \ref{a:decor}, we choose the index $k$ such that 
$d_{k+1}$ decreases most after a permutation for pair ($k,k+1$) is performed.  In other words, we first permute the pairs ($k,k+1$) for which we are most confident of the order.  We define $k$ by
\begin{equation}\label{eq:k}
k = \arg \min_{1 \le j \le n-1} \{ \bar{d}_{j+1}/d_{j+1}: \bar{d}_{j+1} < d_{j+1} \}. 
\end{equation}  
If no $k$ can be found, no more permutations are applied. 

\subsection{Lazy transformation strategy}\label{s:lt}
In LAMBDA reduction, IGTs can be applied to the same entries in $\L$ numerous times.  We now explain how this can occur.  When we permute pair ($k,k+1$), the entries of $\L(k\!:\!k+1,1\!:\!k-1)$ are modified (see (\ref{eq:L22})).  If the absolute values of the entries of $\L(k\!:\!k+1,1\!:\!k-1)$ are bounded above by 1/2 before the permutation, then these bounds may no longer hold after the permutation.  Hence, new IGTs have to be applied.  To avoid this extra work, we want to defer as much IGTs as possible to the end of the reduction process.  From (\ref{eq:D2}), if $d_k < d_{k+1}$, we need $|l_{k+1,k}|$ to be as small as possible to determine the order of the ambiguities.  Therefore, at first, we apply IGTs only on some of the subdiagonal entries of $\L$.  Then, when no more permutations occur, IGTs are applied to all the entries in the strictly lower triangular part of $\L$.  This strategy is called a ``lazy'' transformation strategy in \cite{ChaYZ05}.

\subsection{The reduction algorithm}

We present the modified reduction algorithm (MREDUCTION) given in \cite{ChaYZ05}.  It uses an $(n+1)$-dimensional vector ChangeFlag to track if $l_{k+1,k}$ is modified by the last permutation. 

\begin{algorithm}
\textnormal{(MREDUCTION)
Given the covariance matrix $\W_{\hbx}$ and real-valued LS estimate $\hbx$ of $\x$.
This algorithm computes an integer unimodular matrix $\Z$ and the $\mathrm{L^TDL}$ factorization
$\W_{\hbz}=\Z^T\W_{\hbx}\Z=\L^T\D\L$, where
$\L$ and  $\D$ are updated from the factors of the $\mathrm{L^TDL}$ factorization of $\W_{\hbx}$.
This algorithm also computes $\hbz=\Z^T\hbx$, which overwrites $\hbx$.
\begin{tabbing}
$\quad$ \= {\bf function}: $[\Z, \L,\D,\hbx] = \mathrm{MREDUCTION}(\W_{\hbx}, \hbx)$ \\
	\> Compute the $\mathrm{L^TDL}$ factorization of $\W_{\hbx}$ \\
        \> with symmetric pivoting $\P^T\W_{\hbx}\P=\L^T\D\L$ \\
        \>  $\hbx = \P^T \hbx$  \\  
	\> $\Z = \P$ \\
	\> Set all elements of ChangeFlag(1:n+1) to ones \\
	\> {\bf while}  \= true \\
	\>              \> minratio = 1 \\
	\> 	        \> {\bf for}  \= $k$ = $1:n-1$ \\
	\>		\> 	\> \bf{if } \= $\frac{\D(k,k)}{\D(k+1,k+1)} < 1$ \\
	\>		\>	\>	\>  \bf{if }  \= ChangeFlag$(k+1) = 1$ \\
	\>		\>	\>	\>	\> $[\L, \hbx, \Z] = \mbox{GAUSS}(\L, k+1, k, \hbx, \Z)$  \\
	\>		\>	\>	\>	\>  $\bbD(k+1,k+1)=\D(k,k)+\L(k+1,k)^2\D(k+1,k+1)$ \\
	\>		\>	\>	\>	\>  ChangeFlag$(k+1) = 1$	\\
	\>		\>	\>	\>  \bf{end } \\
	\>		\>	\>	\>  tmp = $\frac{\bbD(k+1,k+1)}{\D(k+1,k+1)}$  \\
	\>		\>	\>	\>  \bf{if } \= tmp $<$ minratio \\
	\>		\>	\>	\>	\>  $i = k \quad$  // see (\ref{eq:k}) \\
	\>		\>	\>	\>	\> minratio = tmp \\
	\>		\>	\>	\>	\> $\tilde{d}  = \bbD(k+1,k+1)$ \\
	\>		\>	\>	\> \bf{end }\\	
	\>		\> 	\> \bf{end } \\
	\>		\> {\bf end} \\
	\>		\> {\bf if} minratio $ = 1$ \\
	\>		\>	\> break while loop \\
	\>		\> \bf{end }\\
	\>		\> $[\L, \D, \hbx, \Z] = \mbox{PERMUTE}(\L, \D, i,\tilde{d}, \hbx, \Z)$  \\
	\>		\> Set ChangeFlag$(i:i+2)$ to ones \\
	\> {\bf end} \\
	\> // Apply IGTs to $\L$'s strictly lower triangular part \\
	\> {\bf for} $k$ = $1:n-1$ \\
	\>	\> {\bf for} $i=k+1:n$ \\
	\>	\>  \>  $[\L, \hbx, \Z] = \mbox{GAUSS}(\L, i, k, \hbx, \Z)$ \\
	\>	\> {\bf end} \\
	\> {\bf end}
\end{tabbing}
}
\end{algorithm}

Numerical simulations in \cite{ChaYZ05} show that MLAMBDA can be much faster than LAMBDA implemented in Delft's LAMBDA package (MATLAB, version 2.0) for 
high dimensional problems. In Section \ref{s:simulations}, our numerical simulations indicate that MREDUCTION can be numerically unstable on some problems.  In these problems, MLAMBDA finds a worse solution to the OILS problem than LAMBDA.

\section{Search process}\label{s:searchGPS}

After the reduction process, the search process starts. 
The critical step in understanding the search process is to rewrite the objective function \eqref{eq:fobjGPS} in terms of a sum-of-squares, similar to (\ref{eq:objf}).  
Substituting the $\mathrm{\L^T\D\L}$ factorization in (\ref{eq:fobjGPS}), we get
\begin{equation}\label{eq:fobjGPS2}   
\min_{\z\in \mathbb{Z}^n} (\z-\hbz)^T\L^{-1}\D^{-1}\L^{-T}(\z-\hbz).
\end{equation}
Define $\bbz$ as
\begin{equation}\label{eq:bbz}
\bbz =  \z - \L^{-T}(\z-\hbz),
\end{equation} 
or equivalently
\begin{equation*}
\L^T(\z - \bbz) = \z -\hbz,
\end{equation*}
which can be expanded to
\begin{equation}\label{eq:seqDefN}
\bz_{j} = \hat{z}_j + \sum_{i=j+1}^n l_{ij}(z_i - \bz_{i}), \quad j = n:-1:1. 
\end{equation}
Observe that $\bz_{j}$ depends on $z_{j+1}, \ldots, z_n$.
With \eqref{eq:bbz}, we can rewrite the optimization problem \eqref{eq:fobjGPS2} as follows
\begin{equation}\label{eq:fobjGPS3}
\min_{\z\in \mathbb{Z}^n} (\z-\bbz)^T\D^{-1}(\z-\bbz),
\end{equation}
or equivalently
\begin{equation}\label{eq:GPSsearch}
\min_{\z \in \mathbb{Z}^n} \sum_{j=1}^n \frac{(z_j - \bz_{j})^2}{d_j}.
\end{equation} 
Assume that the solution of \eqref{eq:GPSsearch} satisfies the bound
\begin{equation}\label{eq:ambspace}
\sum_{j=1}^n \frac{(z_j - \bz_{j})^2}{d_j} < \beta^2.
\end{equation} 
Note that \eqref{eq:ambspace} is a hyper-ellipsoid, which we refer to as an ambiguity search space.
If $\z$ satisfies (\ref{eq:ambspace}), then it must also satisfy inequalities
\begin{align}
& \mbox{level }n : \frac{(z_n - \bz_{n})^2}{d_n} < \beta^2, \nonumber\\ 
& \quad \vdots \nonumber \\
& \mbox{level }k : \frac{(z_k - \bz_{k})^2}{d_k} < \beta^2 - \sum_{i=k+1}^n \frac{(z_i - \bz_{i})^2}{d_i} \label{eq:validintk} \\ 
& \quad \vdots \nonumber \\
& \mbox{level }1 : \frac{(z_1 - \bz_{1})^2}{d_1} < \beta^2 -  \sum_{i=2}^n \frac{(z_i - \bz_{i})^2}{d_i}. \nonumber
\end{align}
The search process starts at level $n$ and moves down to level 1.
From \eqref{eq:validintk}, the range of $z_k$ is $[l_k,u_k]$, where
\begin{equation}
l_k = \Big \lceil \bz_k - d_k^{1/2}(\beta^2 - \sum_{i=k+1}^n (z_i - \bz_{i})^2/d_i)^{1/2} \Big \rceil  \label{eq:lGPS}
\end{equation} 
and
\begin{equation}
u_k = \Big \lfloor \bz_k + d_k^{1/2}(\beta^2 - \sum_{i=k+1}^n (z_i - \bz_{i})^2/d_i)^{1/2} \Big \rfloor. \label{eq:uGPS}
\end{equation}
With the inequalities at each level, the search for the OILS solution can be done with the same procedure shown in Section \ref{s:search}.

\chapter{Reduction Misconceptions and New Reduction Algorithms}\label{s:misconceptions}

Basically there are two communities studying ILS problems: the information theory and communications community and the GNSS community.  
Typically, the former uses the OILS problem in the standard form, while the later uses the quadratic form.
In Section \ref{s:itILS}, we presented the OILS problem in the standard form and the LLL reduction method.  In Section \ref{s:GPS}, we presented the OILS problem in the quadratic form and the LAMBDA reduction and the MREDUCTION methods.
It appears that there are two misconceptions about the reduction process in the literature.
The first  is that the reduction process should decorrelate the covariance matrix
of the real least squares estimate as far as possible, i.e., make the off-diagonal entries
of the covariance matrix as small as possible (see, e.g., \cite{JonT96}, \cite[p.\ 498]{StrB97} and \cite[p.\ 369]{Teu98}). 
This misconception also appears in the communications literature, where it is claimed that the search process will be faster if the reduction process makes the off-diagonal entries of the triangular matrix $\R$ as small as possible (see, e.g., \cite{ChaG09} and \cite{MurGDC06}).  
The second is that the reduction process should reduce the condition number of the the covariance matrix
(see, e.g., \cite{LiuHZO99}, \cite{LouG03} and \cite{Xu01}).
In this Chapter, we show that both are incorrect in Sections \ref{s:proof} and \ref{s:conditionNumber}, respectively.  Our results will provide insight on the role of lower triangular IGTs in the reduction process. In Section \ref{s:plambda}, this new understanding leads us to develop PREDUCTION, a new reduction algorithm which is more efficient and numerically stable than LAMBDA reduction and MREDUCTION. In Section \ref{s:simulations}, we present simulation results.  Finally, in Section \ref{s:miscLLL}, we discuss the implications of these results to the standard form of the OILS problem and to the LLL reduction algorithm.  

\section{Impact of decorrelation on the search process}\label{s:proof}

As seen in Section \ref{s:reductionGPS}, according to the literature,
one of the two goals of the reduction process is to decorrelate the
ambiguities as much as possible.  Decorrelating the ambiguities as
much as possible implies making the covariance matrix as diagonal as
possible, i.e., making the absolute values of the off-diagonal entries of $\L$ as small as
possible.  In the following, we show that to solely
make the absolute values of the off-diagonal entries of $\L$ as small as possible will have 
no impact on the search process. 

\begin{THEOREM}\label{thm:decorrelation}
Given the OILS problem \eqref{eq:qils} and the reduced OILS problem
\eqref{eq:fobjGPS}. If the transformation matrix $\Z$
is a product of lower triangular IGTs, then the search trees for problems
\eqref{eq:qils} and \eqref{eq:fobjGPS} are identical.
\end{THEOREM}
\begin{proof}
Let the $\mathrm{L^TDL}$ factorization of $\W_{\hbx}$ and $\W_{\hbz}$ be
\begin{equation*}
\W_{\hbx} = \L^T \D \L, \quad \W_{\hbz} = \bbL^T \bbD \bbL.
\end{equation*} 
As shown in Section \ref{s:reductionGPS}, the OILS problems
(\ref{eq:qils}) and \eqref{eq:fobjGPS} can be written in the form (c.f.\ \ref{eq:GPSsearch})
\begin{equation}\label{f}
f(\x) = \sum_{j=1}^n (x_j - \bx_{j})^2/d_j, \quad
f(\z) = \sum_{j=1}^n (z_j - \bz_{j})^2/\bd_j,
\end{equation}
where
\begin{eqnarray}\label{seqDef}
\bx_{j}  & = & \hat{x}_j + \sum_{i=j+1}^n l_{ij}(x_i - \bx_{i}), \quad \bz_{j} =  \hat{z}_j + \sum_{i=j+1}^n \bl_{ij}(z_i - \bz_{i}).
\end{eqnarray}
We first consider the case where $\Z$ is a single lower triangular IGT $\Z_{kj}$ ($k > j$), which is applied
to $\L$ from the right to make $|l_{kj}|$ as small as possible (see Section \ref{s:GPSigt}).  We have 
\begin{equation*}
\bbL = \L \Z_{kj} = \L - \mu \L \e_k \e_j^T,
\end{equation*}
where the modified entries in $\L$ are
\begin{equation}\label{Lchange}
\bar{l}_{tj} = l_{tj} - \mu l_{tk}, \quad t = k, \ldots, n.
\end{equation}
Let $\hbz = \Z_{kj}^T \hbx$. 
Thus,
\begin{equation}\label{ah}
\hat{z}_i = \left\{
\begin{array}{l l}
 \hat{x}_i, & \quad \text{if $i \ne j,$}\\
 \hat{x}_j - \hat{x}_k \mu, & \quad \text{if $i = j.$}\\
\end{array} \right.
\end{equation}
From \eqref{seqDef} and \eqref{ah}, we know that 
\begin{equation}\label{a}
\bz_{i}  = \bx_{i}, \quad \text{for }i>j.
\end{equation}
Lower triangular IGTs do not affect the $\D$ factor, meaning
\begin{equation}\label{eq:Dfactor}
\bd_i = d_i, \quad \forall i.
\end{equation}
We want to compare the enumerated points in the search process of problems \eqref{eq:qils} and \eqref{eq:fobjGPS}.
The search process starts at level $n$ and moves down to level 1.
When it moves down to level $i$, it chooses $x_i = \round{\bx_{i}}$ and $z_i = \round{\bz_{i}}$.  From \eqref{a} and \eqref{eq:Dfactor},
if the chosen integer $x_i$ is not valid, i.e., it does not satisfy bound \eqref{eq:validintk} at level $i$, then the chosen integer $z_i$ is also not valid.
In this case, the search trees for problems \eqref{eq:qils} and \eqref{eq:fobjGPS} will both move up to level $i+1$.
Therefore, before we reach level $j$ in the search process, we have
\begin{equation*}
z_{i}  = x_{i}, \quad \text{for }i>j.
\end{equation*}
At level $j$, 
\begin{eqnarray}
\bx_{j} - \bz_{j} & = &  \hat{x}_j + \sum_{i=j+1}^n
l_{ij}(x_i - \bx_{i}) - \hat{z}_j - \sum_{i=j+1}^n
\bar{l}_{ij}(z_i - \bz_{i})
\nonumber \\
& = &  \hat{x}_k \mu + \sum_{i=j+1}^n l_{ij}(x_i - \bx_{i}) -
\sum_{i=j+1}^n \bar{l}_{ij}(z_i - \bz_{i}) \quad \text{(using
\eqref{ah})} \nonumber \\
& = &  \hat{x}_k \mu + \sum_{i=j+1}^n (l_{ij} - \bar{l}_{ij})(x_i - \bx_{i}) \quad \text{(using \eqref{a})} \nonumber \\
& = & \hat{x}_k \mu + \sum_{i=k}^n (l_{ij} - \bar{l}_{ij})(x_i - \bx_{i})  \quad \text{(using \eqref{Lchange})} \nonumber \\
& = & \hat{x}_k \mu + \sum_{i=k}^n \mu l_{ik}(x_i - \bx_{i})
\quad \text{(using \eqref{Lchange})} \nonumber \\
& = & \hat{x}_k \mu + \mu l_{kk}(x_k - \bx_{k}) + \sum_{i=k+1}^n
\mu l_{ik}(x_i - \bx_{i}) \nonumber \\
& = & x_k \mu + \mu[\hat{x}_k + \sum_{i=k+1}^n
l_{ik}(x_i - \bx_{i}) -\bx_{k}] \quad \text{(since $l_{kk} = 1$)} \nonumber  \\
& = & x_k \mu \quad \text{(using \eqref{seqDef}).} \label{eq:resultc}
\end{eqnarray}
Since $x_k$ and $\mu$ are integers,  $\bz_{j}$ is an integer distance from $\bx_{j}$. 
This means that if integer $x_j$ is chosen when we move down to level $j$ in the search, then the chosen integer $z_j$ is (see \eqref{eq:resultc})
\begin{equation}
z_j = x_j - x_k \mu. \label{eq:z_j}
\end{equation} 
From \eqref{eq:resultc} and
\eqref{eq:z_j}, we obtain
\begin{equation}\label{eq:leveljs}
z_j - \bz_{j} = x_j - \bx_{j}.
\end{equation}
Using \eqref{ah} and \eqref{eq:leveljs} in \eqref{seqDef}, we get 
\begin{equation*}
\bz_{i} = \bx_{i}, \quad \text{for }i < j.
\end{equation*}
In other words, while $z_j$ and $x_j$ have different values, they have the same impact on the lower levels of the search process.
Hence, the enumerated points in the search process satisfy
\begin{equation*}
z_i = x_i, \quad \text{for }i<j.
\end{equation*}
For each level $i$ in the search process, the enumerated points $\x$ and $\z$ satisfy
\begin{align*}
& z_i = x_i, \quad \forall i \ne j, \\
& z_j = x_j - x_k \mu, \quad \text{for }i = j. 
\end{align*} 
This shows that the search trees 
for problems \eqref{eq:qils} and \eqref{eq:fobjGPS} are identical. 
Consider the case where $\Z$ is a product of lower triangular IGTs, i.e., $\Z = \Z_1, \ldots, \Z_n$, used to make the absolute values of the other off-diagonal entries of $\L$ as small as possible.  As shown, applying $\Z_1$ to \eqref{eq:qils} will transform the ILS problem, but not modify the search tree.  Applying $\Z_2$ to this transformed ILS problem will also not modify the search tree, and so on.
Thus, if $\Z$ is a product of lower triangular IGTs, the search trees for problems \eqref{eq:qils} and \eqref{eq:fobjGPS} are identical.
\end{proof}
Since the search trees are identical, lower triangular IGTs by themselves have no impact on the search process.
Hence, it is not true that the search process is more efficient when the off-diagonal entries of $\L$ are as small as possible.

We provide a 2 by 2 example.
Let the covariance matrix $\W_{\hbx}$ be
\begin{equation*}
\W_{\hbx} =
\left(
\begin{array}{cc}
11026 & 1050 \\
1050 & 100 \\
\end{array}
\right).
\end{equation*}
Its $\L$ and $\D$ factors are
\begin{equation*}
\L =
\left(
\begin{array}{cc}
1 & 0 \\
10.5 & 1 \\
\end{array}
\right), \quad
\D = 
\left(
\begin{array}{cc}
1 & 0 \\
0 & 100 \\
\end{array}
\right).
\end{equation*}
Let the real least-squares estimate be $\hbx = (5.38, 18.34)^T$.
We can make $|l_{21}|$ as small as possible with the following IGT
\begin{equation*}
\Z =
\left(
\begin{array}{cc}
1 & 0 \\
-10 & 1 \\
\end{array}
\right). \\
\end{equation*}
The covariance matrix and its $\L$ factor become
\begin{align*}
\W_{\hbz} = \Z^T \W_{\hbx} \Z =
\left(
\begin{array}{cc}
26 & 50 \\
50 & 100 \\
\end{array}
\right), \quad
\bbL & = \L \Z = \left(
\begin{array}{cc}
1 & 0 \\
0.5 & 1 \\
\end{array}
\right). 
\end{align*}
 
In \cite{Teu98}, the correlation coefficient $\rho$ and the elongation of the search space $e$ are used to quantify the correlation between the ambiguities. The correlation coefficent $\rho$ between random variables $s_1$ and $s_2$ is defined as (see \cite[p.\ 322]{StrB97})
\begin{equation*}
\rho =  \sigma_{s_1 s_2}/\sigma_{s_1} \sigma_{s_2}.   
\end{equation*}
The elongation of the search space $e$ is given by square of the condition number of the covariance matrix (see Section \ref{s:conditionNumber}). 
For the original ambiguities $\x$, we have $\rho = 0.999$ and $e = 1.113 \times 10^3$.   
For the transformed ambiguities $\z$, we have $\rho = 0.981$ and $e = 12.520$.  These measurements indicate that the transformed ambiguities are more decorrelated. 
The points $(x_1,x_2)^T$ and $(z_1,z_2)^T$ encountered during the search process are shown in Table \ref{t:WOD} and \ref{eq:IGTt}, where $-$ indicates that no valid integer is found.  In both cases, the first point encountered is valid, while the others points are invalid.  The OILS solution is $\cbx = (2,18)^T.$
As expected, we observe that the lower triangular IGT did not reduce the number of points encountered in the search process. 
\begin{table}
\caption{Search process without IGT}\label{t:WOD}
\begin{center}
\begin{tabular}[h!]{|c c|}
\hline
$x_1$ & $x_2$ \\
\hline
2  & 18  \\
$-$ & 18 \\
 $-$ & 19 \\
$-$  & 17  \\
$-$ &  20 \\
$-$ & $-$   \\
\hline
\end{tabular}
\end{center}
\end{table}
\begin{table}
\caption{Search process with IGT}\label{eq:IGTt}
\begin{center}
\begin{tabular}{|c c|}
\hline
$z_1$ & $z_2$ \\
\hline
-178  & 18  \\
$-$  & 18  \\
 $-$ & 19 \\
$-$  & 17  \\
$-$ &  20 \\
$-$ & $-$   \\
\hline
\end{tabular}
\end{center}
\end{table}

\subsection{Implications to some reduction strategies}
In \cite{LiuHZO99}, a united ambiguity decorrelation approach is proposed: unlike the usual pairwise decorrelation, all the ambiguities are decorrelated at once.  This allows for faster, but not maximum, ambiguity decorrelation.  Their approach can be divided in two stages: (i) reordering the ambiguities (ii) decorrelating them.  The reduction process can be written as $\M^T\P^T\W_{\hbx}\P \M$, where $\P$ is a permutation matrix and $\M$ is a product of lower triangular IGTs.     In this contribution, we have shown that  stage (ii) will not improve the search process.  This means that the search process would have been identical if the reduction strategy consisted of (i) only.  The united ambiguity decorrelation approach can be iterated until no more decorrelation is possible.  In this case, only the last decorrelation step can be removed. 

In the lazy transformation strategy of the MREDUCTION algorithm (see Section \ref{s:mlambda}), when no more permutations occur in the reduction process, lower triangular IGTs are applied to the off-diagonal entries of $\L$.  Our current understanding shows that these IGTs are unnecessary; removing them reduces the computational cost of the reduction process, without affecting the search process. 

\section{Partial reduction}\label{s:partialDecorrelation}
In the literature, it is often conjectured that when the ambiguities get more decorrelated, the computational cost of the search process decreases. We have already shown that to solely decorrelate the ambiguities by applying lower triangular IGTs to the $\L$ factor of the $\mathrm{L^TDL}$ factorization of the covariance matrix will not help the search process. However, as in LAMBDA reduction, lower triangular IGTs combined with permutations can significantly reduce the cost of the search process.   This indicates that the accepted explanation is, to say the least, incomplete.  We now provide a new explanation on the role of lower triangular IGTs in the reduction process.

We claim that the computational cost of the search depends mainly on the $\D$ factor.  The off-diagonal entries of $\L$ are only 
important when they affect $\D$.  In the reduction process, when we permute pair $(k,k+1)$, $\D$ is modified according to \eqref{eq:D2}.
We strive for (\ref{eq:order}).  In order to make $\bar{d}_{k+1}$ as small as possible, from (\ref{eq:D2}), we observe that $|l_{k+1,k}|$ should be made as small as possible. An example would be helpful to show this.
Let the $\L$ and $\D$ factors of a 2 by 2 covariance matrix $\W_{\hbx}$ be
\begin{equation*}
\L =
\left(
\begin{array}{cc}
1 & 0 \\
0.8 & 1 \\
\end{array}
\right), \quad
\D = 
\left(
\begin{array}{cc}
1 & 0 \\
0 & 100 \\
\end{array}
\right).
\end{equation*}
We have $d_{2} = 100$ and $d_{1} = 1$.  Let the real least-squares estimate be $\hbx = (13.5, 1.2)^T$.
If we permute the two ambiguities without first applying an IGT, 
i.e.,
\begin{equation*}
\Z =
\left(
\begin{array}{cc}
0 & 1 \\
1 & 0 \\
\end{array}
\right), \\ 
\end{equation*}
using \eqref{eq:D2}, we have 
$\bar{d}_{2} = 65$ and $\bar{d}_{1} = 1.54$.  The search process will be more efficient after this transformation because  
$\bd_2 < d_2$ allows more pruning to occur (see \eqref{eq:order}). 
The integer pairs $(z_1,z_2)^T$ encountered during the search are $(2,14)^T, (-,14)^T, (-,13)^T$ and $(-,-)^T$.
The OILS solution is $\cbx = \Z^{-T}(2,14)^T =(14,2)^T.$
However, we can make $\bd_2$ even smaller by applying a lower triangular IGT before the permutation, which means that 
\begin{equation*}
\Z =
\left(
\begin{array}{cc}
0 & 1 \\
1 & -1 \\
\end{array}
\right). \\ 
\end{equation*} 
In this case, we have $\bd_2 = 5$ and $\bd_{1} = 20$.  Now, three and not four integer pairs are encountered during the search, namely $(2,12)^T, (-,12)^T$ and $(-,-)^T$.  The OILS solution is $\cbx = \Z^{-T}(2,12)^T =(14,2)^T.$  This example illustrates how a lower triangular IGT, followed by a permutation, can prune more nodes from the search tree. 

It is useful to make $|l_{k+1,k}|$ as small as possible because of its effect on $\D$.
However, making $|l_{jk}|$ as small as possible, where $j > k+1$, will have no effect on $\D$ since \eqref{eq:D2} only involves $l_{k+1,k}$. 
Hence, even if $|l_{jk}|$ is very large, making it smaller will not improve the search process. This means that making all the off-diagonal entries of $\L$ as close to 0 as possible is unnecessary.  
It is only necessary to make $|l_{k+1,k}|$ as close to 0 as possible before permuting pair ($k,k+1$) in order to strive for \eqref{eq:order}.
We call this strategy a ``minimal'' reduction (MINREDUCTION) strategy. The MINREDUCTION algorithm is exactly like Algorithm \ref{a:decor}, except that the first ``for loop'' is replaced with the following statement: $[\L, \hbx, \Z] = \mbox{GAUSS}(\L, k+1, k, \hbx, \Z)$. 

Large off-diagonal entries in $\L$ indicate that the ambiguities
are not decorrelated as much as possible, which  contradicts the claim that it is one of the goals of the reduction process. 
In the following, we provide a 3 by 3 example which illustrates the issue. 
Let the covariance matrix and the real least-squares estimate of the ambiguity 
vector $\x$ be
\begin{equation*}
\W_{\hbx} =
\left(
\begin{array}{ccc}
 2.8376 &  -0.0265 &  -0.8061 \\
 -0.0265 &  0.7587 &  2.0602 \\
 -0.8061 &  2.0602 &   5.7845
\end{array}
\right),
\quad
\hbx =
\left(
\begin{array}{c}
 26.6917 \\
 64.1662 \\
 42.5485
\end{array}
\right).
\end{equation*}
Let $\psi(\W_{\hbx})$ denote the sum of the absolute values of the correlation coefficients of $\W_{\hbx}$, which is
used to quantify the entire correlation between the ambiguities.
It is defined as follows
\begin{equation*}
\psi(\W_{\hbx}) = \sum_{i,j > i}^n \Big|\W_{\hbx}(i,j)/ \sqrt{\W_{\hbx}(i,i)\W_{\hbx}(j,j)}\Big|.
\end{equation*}
For the original ambiguities $\x$, we have $\psi(\W_{\hbx})$ = 1.2005.  
With LAMBDA reduction or MREDUCTION, the transformation matrix is 
\begin{equation*}
\Z =
\left(
\begin{array}{ccc}
  4 &   -2 &  1 \\
  -43 &  19 &  -11 \\
  16 &   -7 &  4
\end{array}
\right).
\end{equation*}
The covariance matrix and the real least-squares estimate become
\begin{equation*}
\W_{\hbz} = \Z^T\W_{\hbx}\Z =
\left(
\begin{array}{ccc}
0.2282 &   0.0452 &  -0.0009 \\
0.0452 &   0.1232 &  -0.0006 \\
-0.0009 &  -0.0006 &   0.0327
\end{array}
\right), \quad
\z  =  \Z^T \hbx 
    =  
\left(
\begin{array}{c}
-1971.6 \\
867.9 \\
-508.9
\end{array}
\right). 
\end{equation*}
We have $\psi(\W_{\hbz})$ = 0.2889, which indicates that the transformed ambiguities $\z$ are less correlated than the original ambiguities.
With MINREDUCTION, the transformation matrix is
\begin{equation*}
\Z =
\left(
\begin{array}{ccc}
  0 &    0 &  1 \\
  1 &   -3 &  -11 \\
  0 &   1 &  4
\end{array}
\right).
\end{equation*}
The covariance matrix and the real least-squares estimate become
\begin{equation*}
\W_{\hbz} = \Z^T\W_{\hbx}\Z =
\left(
\begin{array}{ccc}
   0.7587 &  -0.2160  & -0.1317 \\
  -0.2160 &   0.2518 &   0.0649 \\
  -0.1317 &   0.0649 &   0.0327
\end{array}
\right), \quad
\z  =  \Z^T \hbx 
    =  
\left(
\begin{array}{c}
64.1662 \\
-149.9499 \\
-508.9418
\end{array}
\right).  
\end{equation*}
Now, we have $\psi(\W_{\hbz})$ = 2.0449, which means that the transformed ambiguities are more correlated than the original ambiguities.

\begin{table}[h!]
\caption{Search process with NOREDUCTION}\label{t:NOREDUCTION}
\begin{center}
\begin{tabular}{|c c c|}
\hline
$z_1$ & $z_2$ & $z_3$ \\
\hline
23 & 64 & 43 \\
$-$ & 64 & 42 \\
27 & 64 & 42 \\
$-$ & 64 & 42 \\
$-$ & $-$ & 44 \\
$-$ & $-$ & 41 \\
$-$ & $-$ & $-$ \\
\hline
\end{tabular}
\end{center}
\end{table}

\begin{table}[h!]
\caption{Search process with LAMBDA reduction or MREDUCTION}\label{t:LAMBDAREDUCTION}
\begin{center}
\begin{tabular}{|c c c|}
\hline
$z_1$ & $z_2$ & $z_3$ \\
\hline
-1972 & 868 & -509 \\
$-$ & 868 & -509 \\
$-$ & $-$ & $-$ \\
\hline
\end{tabular}
\end{center}
\end{table}

\begin{table}[h!]
\caption{Search process with MINREDUCTION}\label{t:MINREDUCTION}
\begin{center}
\begin{tabular}{|c c c|}
\hline
$z_1$ & $z_2$ & $z_3$ \\
\hline
64 & -150 & -509 \\
$-$ & -150 & -509 \\
$-$ & $-$ & $-$ \\
\hline
\end{tabular}
\end{center}
\end{table}

We refer to the reduction process which consists of only finding the $\mathrm{L^TDL}$ factorization of $\W_{\hbx}$ as NOREDUCTION.  
The integer triples encountered during the search when NOREDUCTION, LAMBDA reduction and MINREDUCTION are used are shown in Tables \ref{t:NOREDUCTION}, \ref{t:LAMBDAREDUCTION} and \ref{t:MINREDUCTION}, where $-$ indicates that no valid integer is found.
The ILS solution is $\check{\x} = (27,64,42)^T$.
Observe that MINREDUCTION causes the search to encounter the same number of integer points as LAMBDA reduction.  
Furthermore, NOREDUCTION causes the search to encounter four extra integer points than MINREDUCTION, although the ambiguities transformed by the latter are more correlated than the original ambiguities.
This indicates that it is not true that the reduction process should decorrelate the ambiguities as much as possible in order for the search process to be more efficient. 

The significance of this result is threefold: 
\begin{itemize} 
\item It indicates that contrary to common belief, the computational cost of the search is largely independent of the off-diagonal entries of $\L$
and of the correlation between the ambiguities. 
\item It provides a different explanation on the role of lower triangular IGTs in the reduction process.  
\item It leads to a more efficient reduction algorithm, see Section \ref{s:plambda}.
\end{itemize}

\section{Geometric interpretation}\label{s:geometry}

In Section \ref{s:proof}, we have shown that solely decorrelating the ambiguities will not improve the search process.  We now illustrate this result geometrically.
Let the real LS estimate of $\x$ and the covariance matrix be
\begin{equation*}
\hbx =
\left(
\begin{array}{c}
\hat{x}_1 \\
\hat{x}_2 \\
\end{array}
\right),
\quad
\W_{\hbx} =
\left(
\begin{array}{cc}
\sigma_{\hat{x}_1}^2 & \sigma_{\hat{x}_1 \hat{x}_2}  \\
\sigma_{\hat{x}_2 \hat{x}_1} & \sigma_{\hat{x}_2}^2  \\
\end{array}
\right).
\end{equation*}
We assume that $|\sigma_{\hat{x}_1 \hat{x}_2}| > \frac{1}{2} \sigma_{\hat{x}_1}^2$; otherwise, no further decorrelation is possible. 
From \eqref{eq:ambspace}, we observe that the ambiguity search space is centered at $\hbx$.
To decorrelate the two ambiguities, we use the following IGT
\begin{equation*}
\Z =
\left(
\begin{array}{cc}
1 & 0 \\
-[\sigma_{\hat{x}_1 \hat{x}_2} \sigma_{\hat{x}_1}^{-2}] & 1 \\
\end{array}
\right).
\end{equation*}
Note that the $\Z$-transformation reduces $\sigma_{\hat{x}_1}^2$ but does not affect $\sigma_{\hat{x}_2}^2$ in $\W_{\hbx}$. 
Therefore, as explained in \cite[p.\ 365]{Teu98}, the result of this transformation is to push the vertical tangents of the ambiguity search space, where the vertical axis is $\hat{x}_2$ and the horizontal axis $\hat{x}_1$ (see Fig.\ \ref{fig:SearchSpace}).
Since $|\mbox{det}(\Z)| = 1$, a $\Z$-transformation does not change the area of the search space.  

\begin{figure}[h!]
\centering
\scalebox{1}{\includegraphics*[105,560][580,715]{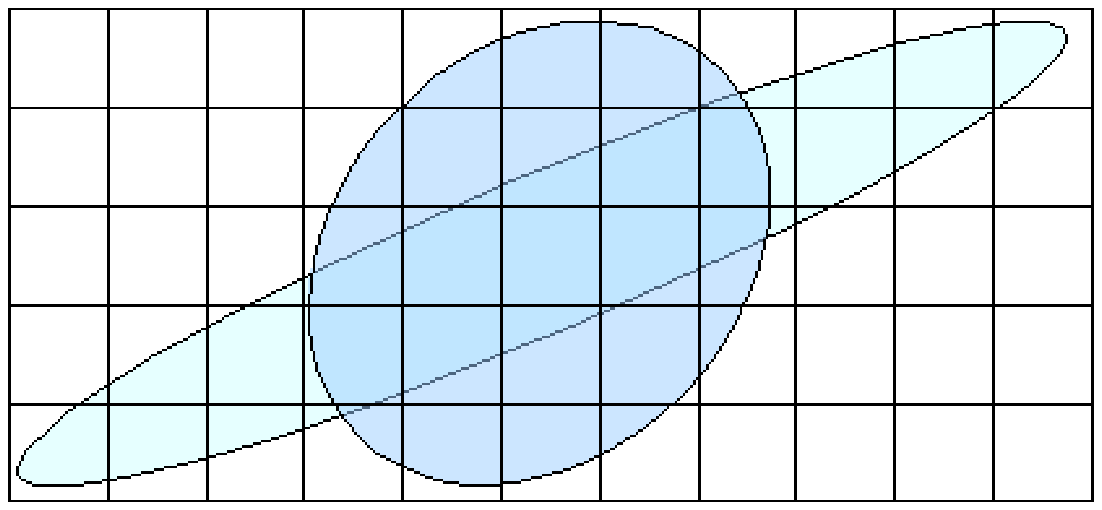}}
\caption{Original and transformed search space.  The transformed search space is less elongated. }
\label{fig:SearchSpace}
\end{figure}

In Fig.\ \ref{fig:SearchSpace}, we observe that if the integer at level 2 is determined (vertical axis), the number of valid integers at level 1 (horizontal axis) is the same in the orginal and transformed search space.  
We prove this result as follows.  
In the search process, the lower bound $l_k$ and the upper bound $u_k$ of the valid integers at a given level $k$ are given by \eqref{eq:lGPS} and \eqref{eq:uGPS}.  In the proof of Theorem \ref{thm:decorrelation}, we
have shown that after applying an IGT $\Z_{jk}$ for $j > k$, we have 
\begin{equation*}
\frac{(x_i - \bx_{i})^2}{d_i}= \frac{(z_i - \bz_{i})^2}{d_i}, \quad \forall i > k.
\end{equation*}
We have also shown that $\bz_{k}$ is an integer distance $\delta$ of $\bx_{k}$.
These results imply that interval $[l_k,u_k]$ for $x_k$ and interval $[\bar{l}_k,\bar{u}_k]$ for $z_k$ 
satisfy (see \eqref{eq:lGPS} and \eqref{eq:uGPS})
\begin{equation*}
[\bar{l}_k,\bar{u}_k] = [l_k + \delta,u_k + \delta].
\end{equation*} 
Hence in the search process, the number of valid integers $x_k$ and the number of valid integers $z_k$ are the same.  
In other words, lower triangular IGTs by themselves do not reduce search halting (see Sect.\ \ref{s:reduction}).

In the literature (see, e.g., \cite{LiuHZO99} and \cite{Teu97}), it is commonly stated that the elongated shape of the ambiguity search space (see \eqref{eq:ambspace}), which is an hyper-ellipsoid, causes the search to be highly inefficient.  It is then said that the role of the $\Z$-transformation is to make the search space less elongated. We now clarify this explanation.  
The elongation of the search space is defined to be the ratio of the major and minor principal axes.
While a lower triangular IGT will not improve the search, it will, however, reduce the elongation of the search space (see Fig.\ \ref{fig:SearchSpace} and the example given in Sect.\ \ref{s:proof}).  This is explained by the fact that changes in the principal axes can occur without changes in the conditional variances. This shows that using the elongation of the search space to evaluate the effectiveness of the reduction can be misleading.
See Section \ref{s:conditionNumber} for an example and the relationship between the elongation of the search space and the condition number of the covariance matrix.

\section{A new reduction method}\label{s:plambda}

This new understanding on the role of IGTs in the reduction process led us to a new algorithm: Partial Reduction (PREDUCTION).  
The motivation of PREDUCTION is to eliminate the unnecessary IGTs applied in LAMBDA reduction. 
Our results indicate that lower triangular IGTs have two roles in the reduction process.
\begin{description}
\item[Efficiency for search.] 
In Section \ref{s:partialDecorrelation}, we showed that if pair $(k,k+1)$ will be permuted, then we first need to make $|l_{k+1,k}|$ as small as possible in order to improve the search efficiency.
\item[Stability for reduction.] 
When only making $|l_{k+1,k}| \le 1/2$,  it is possible that large entries will appear in $\L(k+2\!:\!n,k)$.  This effect may accumulate during the reduction process and introduce huge numbers, which can cause serious rounding errors.    This problem does not occur if every time after we reduce $|l_{k+1,k}|$, we also reduce  $|l_{k+2,k}|,\ldots,|l_{nk}|$.  
\end{description}
The MINREDUCTION strategy mentioned in Section \ref{s:partialDecorrelation} did not include the IGTs that are necessary to ensure numerical stability. 
This means that on some ILS problems, MINREDUCTION yields a different and worse ILS solution than LAMBDA reduction. 
The PREDUCTION algorithm can be summarized as follows.  
It starts with column $n-1$ of $\L$.  At column $k$, it computes the new value of $l_{k+1,k}$ if an IGT were to be applied.  
Then, using this new value, if $\bd_{k+1} \ge d_{k+1}$ holds (see (\ref{eq:D2})), then permuting pair ($k,k+1$) will not help strive for \eqref{eq:order}, therefore it moves to column $k-1$ without applying any IGT; otherwise it
first applies IGTs to make $|l_{ik}| \le 1/2 \mbox{ for } i = k+1\!:\!n$, then it permutes pair ($k,k+1$) and moves to column $k+1$.     
In the LAMBDA reduction algorithm (see Sect.\ \ref{s:lambda}), when a permutation occurs at column $k$, the algorithm restarts, i.e., it goes back to the initial position $k=n-1$.
From (\ref{eq:zszldl}),  no new permutation occurs in the last columns $n-k-1$ of $\L$.  Hence, the new algorithm does not restart, but simply moves to column $k+1$.
In order to further reduce the computational costs, we use the symmetric pivoting strategy presented in Section \ref{s:sp}.  
 
We now present the complete PREDUCTION algorithm:
\begin{algorithm} \label{a:pdecor}
\textnormal{(PREDUCTION).
Given the covariance matrix $\W_{\hbx}$ and real-valued LS estimate $\hbx$ of $\x$.
This algorithm computes an integer unimodular matrix $\Z$ and the $\mathrm{L^TDL}$ factorization
$\W_{\hbz}=\Z^T\W_{\hbx}\Z=\L^T\D\L$, where
$\L$ and  $\D$ are updated from the factors of the $\mathrm{L^TDL}$ factorization of $\W_{\hbx}$.
This algorithm also computes $\hbz=\Z^T\hbx$, which overwrites $\hbx$.
\begin{tabbing}
$\quad$\= {\bf function}: $[\Z, \L,\D,\hbx] = \mathrm{PREDUCTION}(\W_{\hbx}, \hbx)$ \\
\> Compute the $\mathrm{L^TDL}$ factorization of $\W_{\hbx}$ \\
\> with symmetric pivoting $\P^T\W_{\hbx}\P=\L^T\D\L$ \\
\> $\hbx = \P^T \hbx$  \\ 
\> $\Z=\P$ \\
\> $k=n-1$ \\
\> $k1 = k$ \\
\> {\bf while} \= $k>0$  \\
\>            \>  $l = \L(k+1,k) - \round{\L(k+1,k)} \L(k+1,k+1)$  \\
\>            \> $\bbD(k+1,k+1)=\D(k,k)+l^2\D(k+1,k+1)$  \\
\>            \> {\bf if} \=  $\bbD$\=$(k+1,k+1) < \D(k+1,k+1)$ \\
\>	      \>	   	\>  {\bf if} $k \le k1$  \\
\>            \>	 	\>	\> {\bf for} \= $i=k+1:n$ \\
\>            \>              	\>	\> 	\> // See Alg.\ \ref{a:gauss} \\
\>            \>              	\>	\> 	\> $[\L, \hbx, \Z] = \mbox{GAUSS}(\L, i, k, \hbx, \Z)$  \\
\>            \>  		\>	\> {\bf end} \\
\>	      \>		\> {\bf end} \\
\>            \>                     \> // See Alg.\ \ref{a:permute} \\
\>            \>                     \>   $[\L, \D, \hbx, \Z] = \mbox{PERMUTE}(\L, \D, k, \bbD(k+1,k+1), \hbx, \Z)$ \\
\>	      \>		     \> $k1 = k$ \\
\>	      \>		     \> {\bf if} \= $k < n-1$ \\
\>            \>                     \> 	\> $k = k + 1$ \\
\>	      \>		     \> {\bf end} \\
\>            \> {\bf else}\\
\>            \>         \> $k=k-1$ \\
\>            \> {\bf end} \\
\> {\bf end}
\end{tabbing}
}
\end{algorithm}

Note that our final $\L$ might not be LLL-reduced since we do not ensure property \eqref{eq:lambdag1}. 
Consider PREDUCTION without the symmetric pivoting strategy.
Then, Theorem \ref{thm:decorrelation} implies that PREDUCTION will have the same impact on the search process as LAMBDA reduction. With the symmetric pivoting strategy, the initial ordering of the columns $\L$ is different than in LAMBDA reduction.  For this reason, it is no longer true that the search process will be identical.  Nevertheless, we do not expect significant differences in the computational cost, which is confirmed by simulations 
(see Sect.\ \ref{s:simulations}). 
Unlike MREDUCTION, PREDUCTION ensures that we do not create large off-diagonal entries in $\L$ during the reduction process.
This is necessary to avoid serious rounding errors.  

In the following, we give a 3 by 3 example to illustrate the differences between LAMBDA reduction, MREDUCTION and PREDUCTION. 
Let the covariance matrix and the real least-squares estimate of the ambiguity 
vector $\x$ be
\begin{equation*}
\W_{\hbx} =
\left(
\begin{array}{ccc}
1.3616 & 1.7318 & 0.9696 \\
1.7318 & 2.5813 & 1.4713 \\
0.9696 & 1.4713 & 0.8694
\end{array}
\right),
\quad
\hbx =
\left(
\begin{array}{c}
27.6490 \\
10.3038 \\
5.2883
\end{array}
\right).
\end{equation*}

For this example, we get the same transformation matrix from LAMBDA reduction and MREDUCTION, which is
\begin{equation*}
\Z =
\left(
\begin{array}{ccc}
-1 & 1 & 0 \\
2 & 0 & 1 \\
-2 & -1 & -2
\end{array}
\right).
\end{equation*}
The covariance matrix and the real least-squares estimate become
\begin{equation*}
\W_{\hbz} = \Z^T\W_{\hbx}\Z =
\left(
\begin{array}{ccc}
0.3454 & -0.0714 & 0.0200 \\
-0.0714 & 0.2918 & 0.0601 \\
0.0200 & 0.0601 & 0.1738
\end{array}
\right), \quad
\z  =  \Z^T \hbx 
    =  
\left(
\begin{array}{c}
-17.6179 \\
22.3607 \\
-0.2727
\end{array}
\right). 
\end{equation*}

The $\L$ and $\D$ factors of the $\mathrm{L^TDL}$ factorization of $\W_{\hbz}$ are
\begin{equation*}
\L =
\left(
\begin{array}{ccc}
1 & 0 & 0 \\
-0.2889 & 1 & 0 \\
 0.1149 & 0.3459 & 1
\end{array}
\right), \quad
\D =
\left(
\begin{array}{ccc}
0.3205 & 0 & 0\\
0 & 0.2710 & 0 \\
0 & 0 & 0.1738
\end{array}
\right).
\end{equation*}
The integer triples encountered during the search are shown in Table \ref{t:LAMBDA}, where $-$ indicates that no valid integer is found.
The last full integer point found is $\z = (-18,23,0)^T$.
The integer least-squares solution $\check{\x}$ for the original ambiguities is $\check{\x} =  \Z^{-T}\z = (28,10,5)^T$.
\begin{table}[h!]
\caption{Search process with LAMBDA reduction or MREDUCTION}\label{t:LAMBDA}
\begin{center}
\begin{tabular}{|c c c|}
\hline
$z_1$ & $z_2$ & $z_3$ \\
\hline
-17 & 22 & 0 \\
 $-$  & 22 & 0 \\
-18 & 23 & 0 \\
$-$ & $-$ & $-$ \\
\hline
\end{tabular}
\end{center}
\end{table}
Now, we compare the results with our PREDUCTION method.  The transformation matrix is 
\begin{equation*}
\Z =
\left(
\begin{array}{ccc}
0 & 1 & 0 \\
0 & 0 & 1 \\
1 & -1 & -2
\end{array}
\right)
\end{equation*}
The covariance matrix and the real least-squares estimate become
\begin{equation*}
\W_{\hbz} = \Z^T\W_{\hbx}\Z =
\left(
\begin{array}{ccc}
0.8694 & 0.1002 & -0.2676 \\
0.1002 & 0.2918 & 0.0601 \\
-0.2676 & 0.0601 & 0.1738
\end{array}
\right), \quad
\z  =  \Z^T \hbx 
    =  
\left(
\begin{array}{c}
5.2883 \\
22.3607 \\
-0.2727
\end{array}
\right). 
\end{equation*}
The $\L$ and $\D$ factors of the $\mathrm{L^TDL}$ factorization of $\W_{\hbz}$ are
\begin{equation*}
\L =
\left(
\begin{array}{ccc}
1 & 0 & 0 \\
0.7111 & 1 & 0 \\
-1.5392 & 0.3459 & 1
\end{array}
\right), \quad
\D =
\left(
\begin{array}{ccc}
0.3205 & 0 & 0\\
0 & 0.2710 & 0 \\
0 & 0 & 0.1738
\end{array}
\right).
\end{equation*}
The integer triples encountered during the search are shown in Table \ref{t:PLAMBDA}. 
The last full integer point found is $\z = (5,23,0)^T$.
The integer least-squares solution $\check{\x}$ for the original ambiguities is $\check{\x} =  \Z^{-T}\z = (28,10,5)^T$.
Notice that in this example, the search process is identical whether the $\Z$-transformation comes from LAMBDA reduction, MREDUCTION or PREDUCTION.
\begin{table}[h!]
\caption{Search process with PREDUCTION}\label{t:PLAMBDA}
\begin{center}
\begin{tabular}{|c c c|}
\hline
$z_1$ & $z_2$ & $z_3$ \\
\hline
5 & 22 & 0 \\
 $-$  & 22 & 0 \\
5 & 23 & 0 \\
$-$ & $-$ & $-$ \\
\hline
\end{tabular}
\end{center}
\end{table}
     
\section{Numerical simulations}\label{s:simulations}

We implemented the PREDUCTION method given in Section \ref{s:plambda}. We did numerical simulations
to compare its running time with LAMBDA reduction and MREDUCTION.  
All our computations were performed in MATLAB 7.9 on a Pentium-4, 2.66 GHz machine with 501 MB memory running Ubuntu 8.10.

\subsection{Setup}
We performed simulations for different cases. Cases 1-8 are test examples given in \cite{ChaYZ05}. With the exception of case 9, the real vector $\hat \x$ was constructed as follows:
\begin{equation}
\hat \x=100*\texttt{randn}(n,1),
\end{equation}
where $\texttt{randn}(n,1)$ is a MATLAB built-in function to generate a vector of $\n$
random entries which are normally distributed.

The first four cases are based on $\W_{\hbx} = \L^T\D\L$ where
$\L$ is a unit lower triangular matrix with each $l_{ij}$ (for $i>j$) being a random
number generated by $\texttt{randn}$,
and $\D$ is generated in four different ways:
\begin{itemize}
\item Case 1: $\D = \diag(d_i), \; d_i = \texttt{rand}$, where $\texttt{rand}$
is a MATLAB built-in function to generate uniformly distributed random numbers
in $(0,1)$.
\item Case 2: $\D = \diag(n^{-1},(n-1)^{-1}, \ldots,1^{-1})$.
\item Case 3: $\D = \diag(1^{-1}, 2^{-1}, \ldots, n^{-1})$.
\item Case 4: $\D = \diag(200, 200, 200, 0.1, 0.1, \ldots, 0.1)$.
\end{itemize}

The last five cases are as follows:

\begin{itemize}
\item Case 5: $\W_{\hbx}=\U\D\U^T$, $\U$ is a random orthogonal matrix
obtained by the QR factorization of a random matrix generated by $\texttt{randn}(n,n)$,
$\D = \diag(d_i), \; d_i = \texttt{rand}$.

\item Case 6: $\W_{\hbx}=\U\D\U^T$, $\U$ is generated in the same way as in Case 5,
$d_1 = 2^{-\frac{n}{4}}$, $d_n = 2^{\frac{n}{4}}$, other diagonal elements of $\D$
is randomly distributed between $d_1$ and $d_n$, $n$ is the dimension of $\W_{\hbx}$.
Thus the condition number of $\W_{\hbx}$ is $2^{\frac{n}{2}}$

\item Case 7: $\W_{\hbx}=\A^T\A$, $\A=\texttt{randn}(n,n)$.

\item Case 8: $\W_{\hbx}=\U\D\U^T$, the dimension of $\W_{\hbx}$ is fixed to 20,
$\U$ is generated in the same way as in Case 5, $d_1 = 2^{-\frac{k}{2}}$, $d_n = 2^{\frac{k}{2}}$,
other diagonal elements of $\D$ are randomly distributed between $d_1$ and $d_n$, $k = 5,6,\ldots,20$.
Thus the range of the condition number of $\W_{\hbx}$ is  from $2^5$ to $2^{20}$.

\item Case 9:  We assume the linear model
\begin{equation*}
\y = \A \x + \v,
\end{equation*}
where $\A=\texttt{randn}(2n,n)$, $\x = \round{100*\texttt{randn}(n,1)}$ and $\v \sim \mathcal{N}(\0,0.05 \I)$.
Then, we solve the following OILS problem
\begin{equation*}
\min_{\x \in \mathbb{Z}^n} \| \y-\A \x \|_2^2,
\end{equation*} 
which we can rewrite in terms of \eqref{eq:qils}; see \eqref{eq:dev}.
\end{itemize}

Case 4 is motivated by the fact that the covariance matrix $\W_{\hbx}$ in GNSS usually has
a large gap between the third conditioned standard deviation and the forth one
(see \cite[Sect.\ 8.3.3]{Teu98}).  The motivation for case 9 is that in typical GNSS applications, the variance of the noise vector $\v$ is small.
For the reduction process, we took dimensions $n=5\!:\!40$ and performed 40 runs for the all cases.
For the search process, we took dimensions $n=5\!:\!30$ and performed 40 runs for all cases.
The results about the average running time (in seconds) are given in Figs.\ \ref{f:case1} to \ref{f:case9}.
For each case, we give two plots, corresponding to the average reduction time and the average search time, respectively.
Note that $\Z^T \W_{\hbx} \Z = \L^T \D \L$.  Thus, $\W_{\hbx} = \Z^{-T} \L^T \D \L \Z^{-1}$ is a factorization of $\W_{\hbx}$.
We use the relative backward error to check the numerical stability of the factorization, which is 
\begin{equation*}
\frac{\| \W_{\hbx} - \Z^{-T}_c \L^T_c \D_c \L_c \Z^{-1}_c \|_2}{\| \W_{\hbx} \|_2},
\end{equation*}
where $\Z_c, \L_c$ and $\D_c$ are the computed values of $\Z,\L$ and $\D$.
The results for the three reduction algorithms are displayed in Figs.\ \ref{f:case1e} to \ref{f:case9e}. 

\subsection{Comparison of the reduction strategies}
From the simulation results, we observe that PREDUCTION improves the computational efficiency of the reduction stage for all cases.  
Usually, the improvement becomes more significant when the dimension $n$ increases.
For example, in Case 3, PREDUCTION has about the same running time as MREDUCTION and LAMBDA reduction when $n=5$, but is almost 10 times faster when $n = 40$. 
  
Below, we show that the MREDUCTION is not numerically stable.  For this reason, we did not compare the effectiveness of MREDUCTION with the other reduction algorithms. With LAMBDA reduction and PREDUCTION,  we obtain the same computed solution for the same OILS problem.

In Section \ref{s:plambda}, we showed that PREDUCTION without the symmetric pivoting strategy and LAMBDA reduction have exactly the same impact on the search process.  With the symmetric pivoting strategy, PREDUCTION and LAMBDA reduction can give different $\D$ factors.  The $\D$ factor which satisfies order \eqref{eq:order} better depends on the specific OILS problem.     
Nevertheless, Figs.\ \ref{f:case1} to \ref{f:case8} show that there is no significant difference in the search process whether we use PREDUCTION or the LAMBDA reduction.   

In our simulations, we found that MREDUCTION sometimes causes the search to find a different and worse OILS solution than if LAMBDA reduction or PREDUCTION was used. For instance, this occured twice out of 180 runs for case 7 at $n = 35$.  For these problems, the relative backward error of MREDUCTION was in the order of $10^2$ and $10^{12}$, while the relative backward error of LAMBDA reduction and PREDUCTION was in the order of $10^{-14}$.   
Observe that the relative backward error of MREDUCTION is particularly large for cases 2 and 7.  
This means that the MREDUCTION algorithm is not backward stable.   The stability problem is caused by the lazy transformation strategy (see Section \ref{s:mlambda}).  The deferred IGTs in the reduction process can cause large off-diagonal entries in $\L$ to appear, which can lead to big rounding errors. Such an issue is avoided in PREDUCTION.
For all the cases, the relative backward error in PREDUCTION is less than in LAMBDA reduction and in MREDUCTION.  For some cases, the difference is of several orders of magnitude.  For example, for case 3 at $n = 40$, the relative backward error in LAMBDA reduction and MREDUCTION is around $10^{-12}$, while it is $10^{-16}$ with PREDUCTION. This indicates that PREDUCTION is more computationally efficient and stable than both MREDUCTION and LAMBDA reduction. 
In some applications, the computational cost of the search process can be prohibitive. In order to reduce the search time, one might opt to find an approximate OILS solution.  In these cases, the savings in the reduction time provided by PREDUCTION become particularly important.    

\begin{figure}[ht!]
\centering
{\includegraphics[scale=0.80]{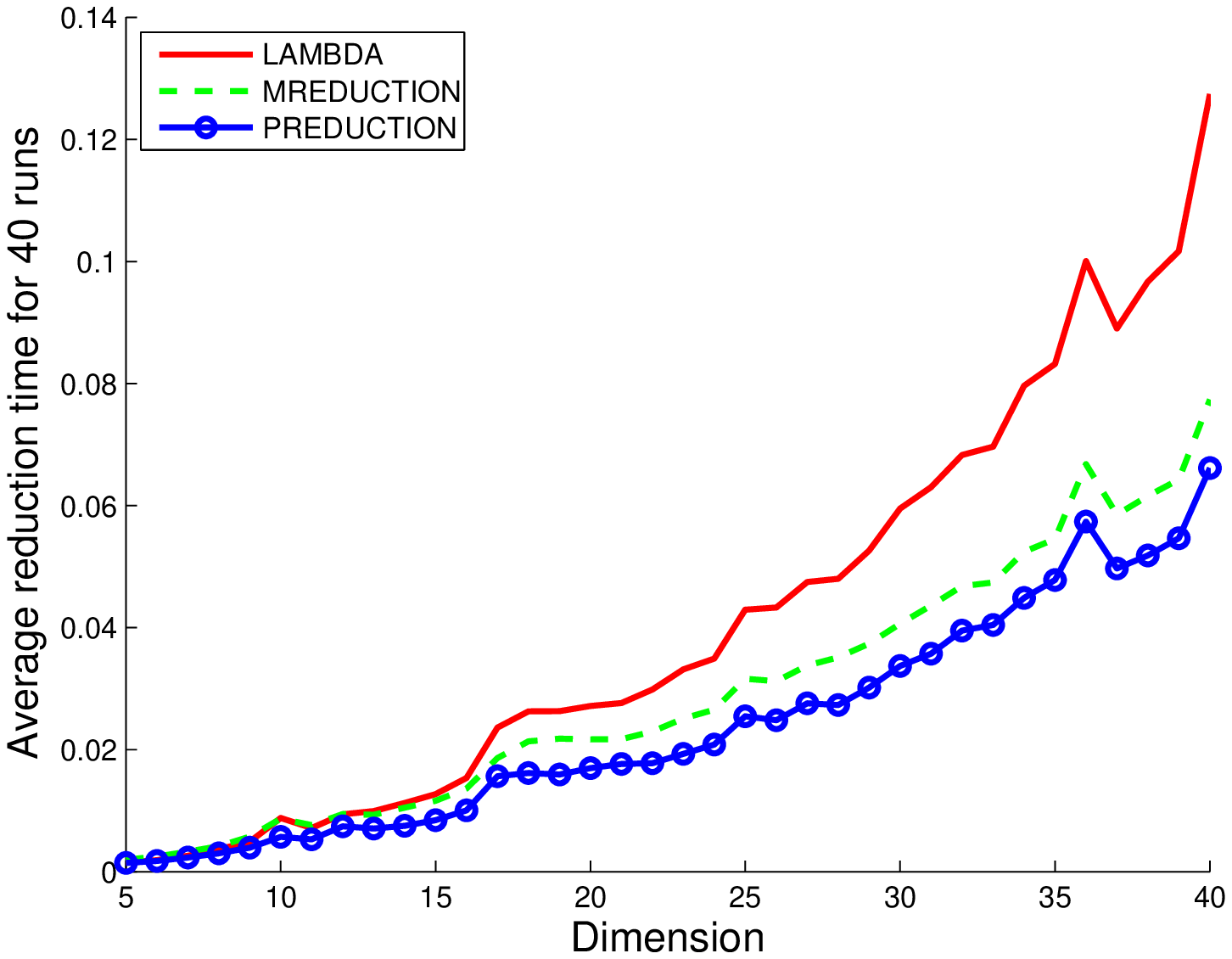}}
{\includegraphics[scale=0.80]{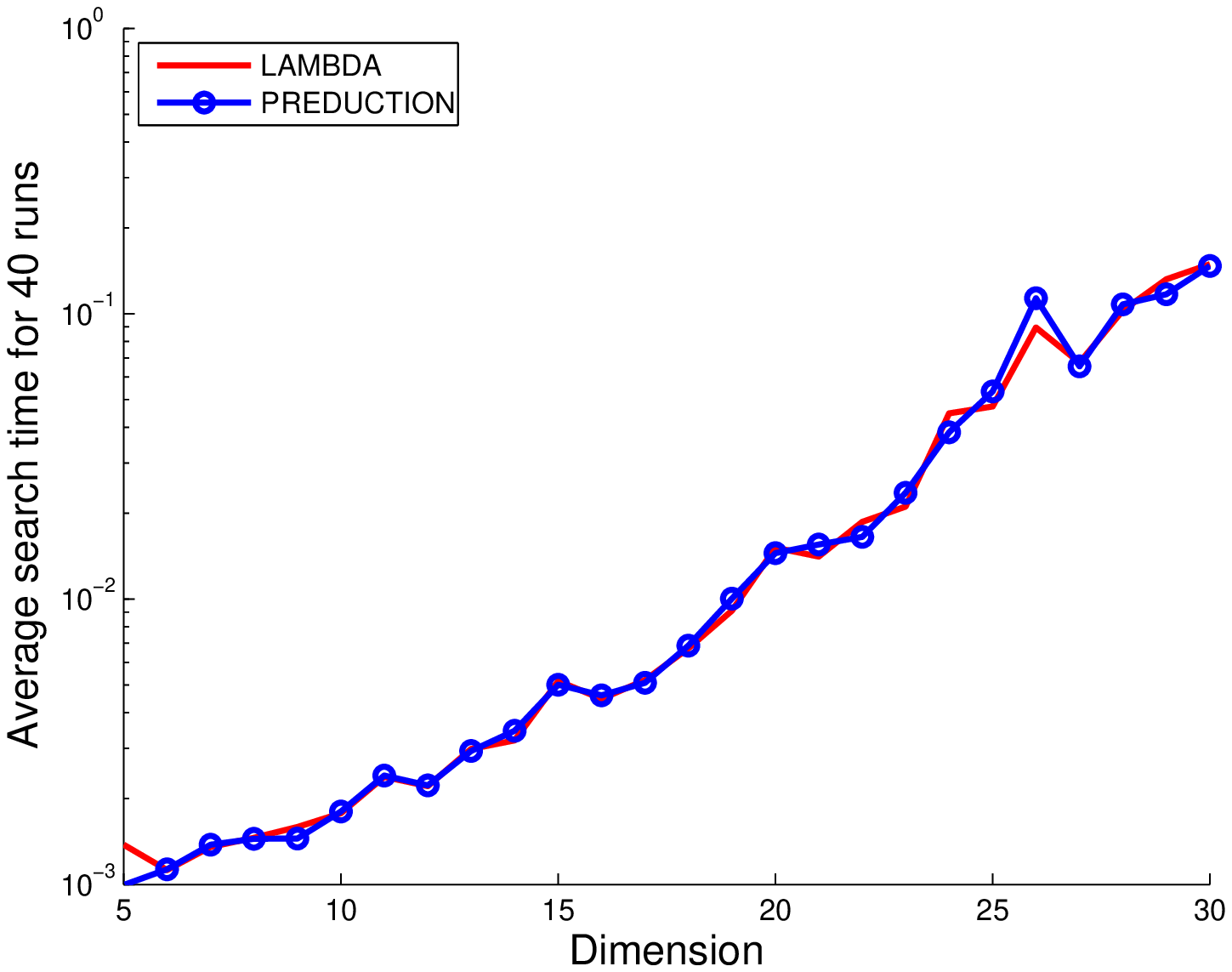}}
\caption{Running time for Case 1} \label{f:case1}
\end{figure}

\begin{figure}[ht!]
\centering
{\includegraphics[scale=0.80]{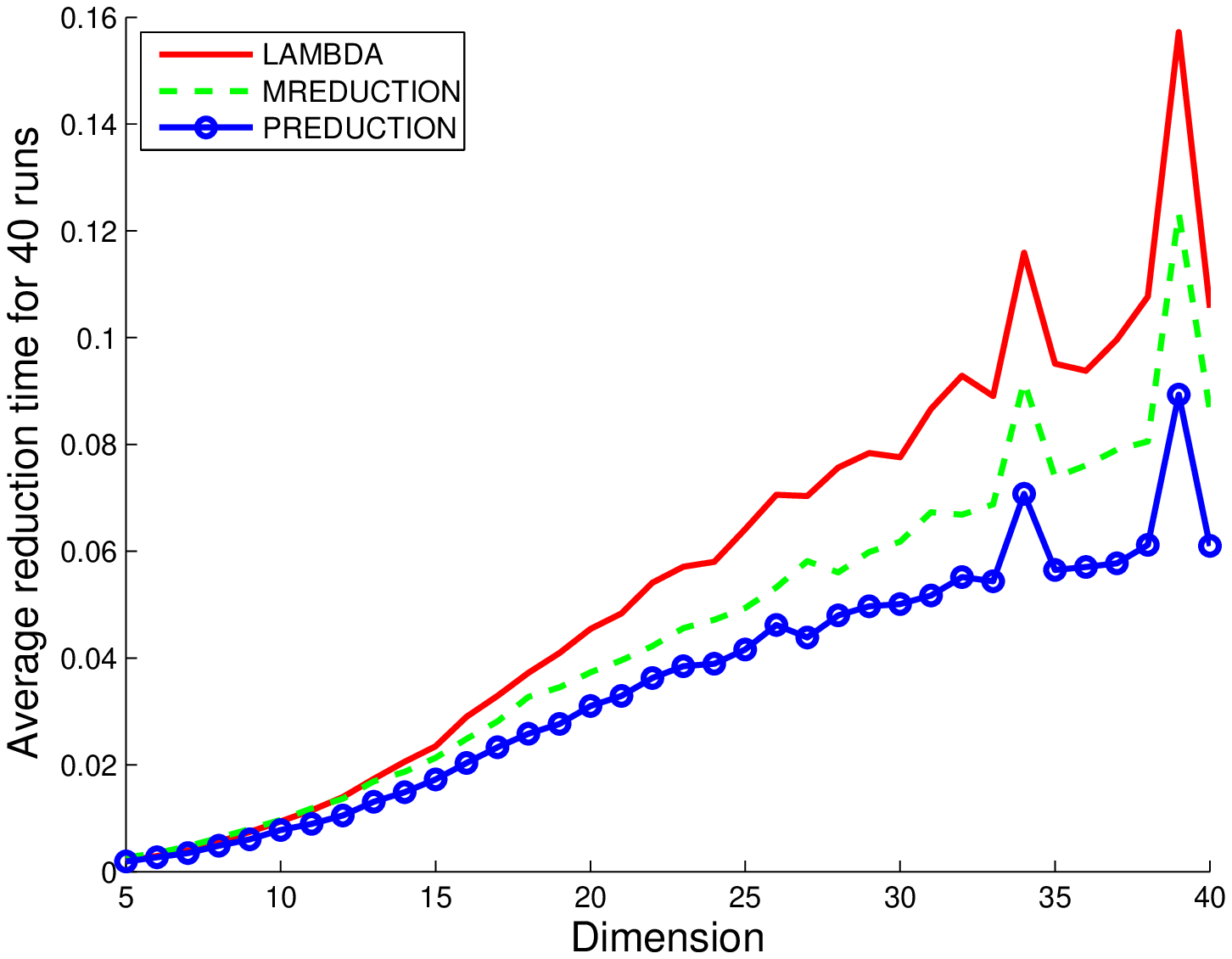}}
{\includegraphics[scale=0.80]{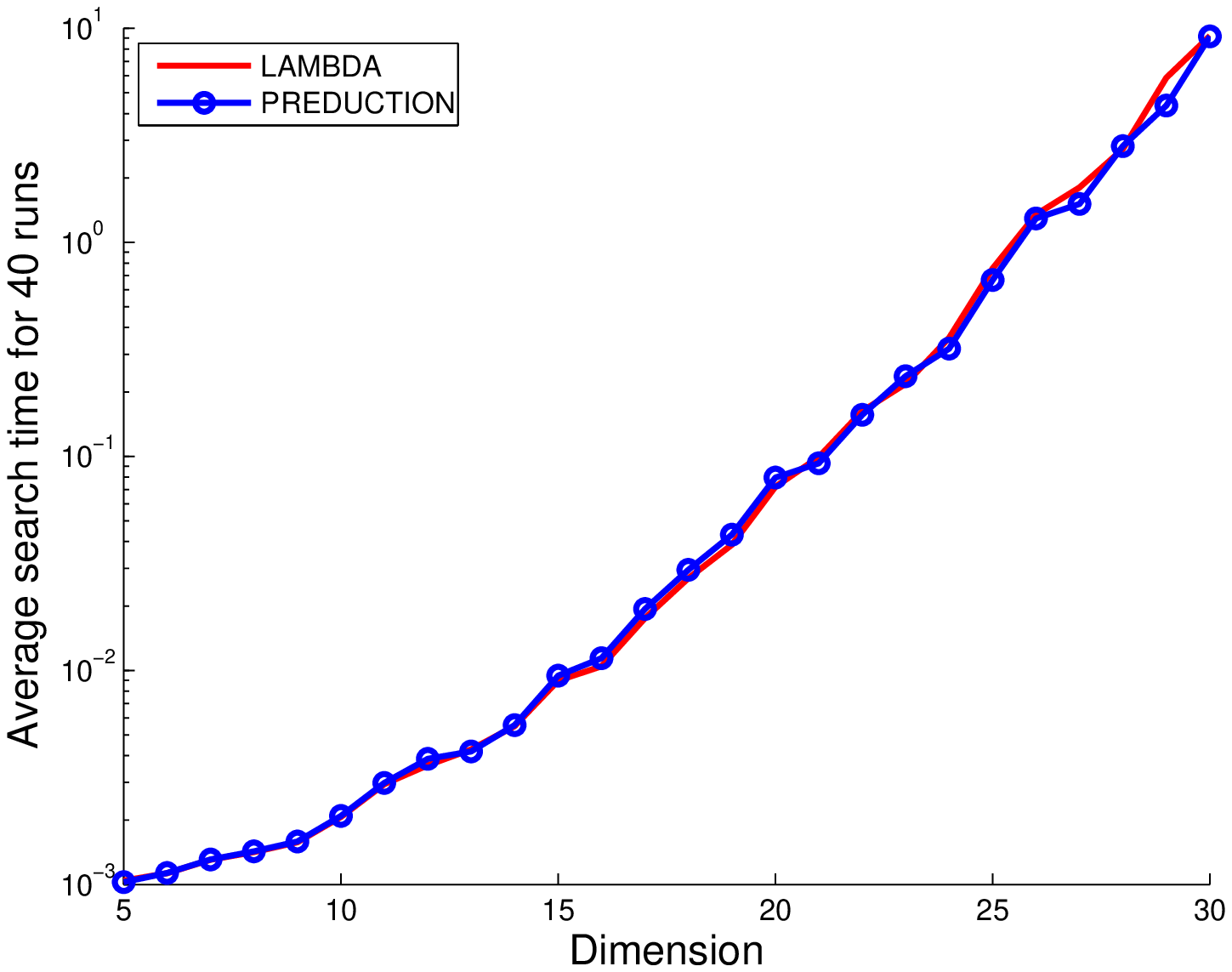}}
\caption{Running time for Case 2} \label{f:case2}
\end{figure}

\begin{figure}[ht!]
\centering
{\includegraphics[scale=0.80]{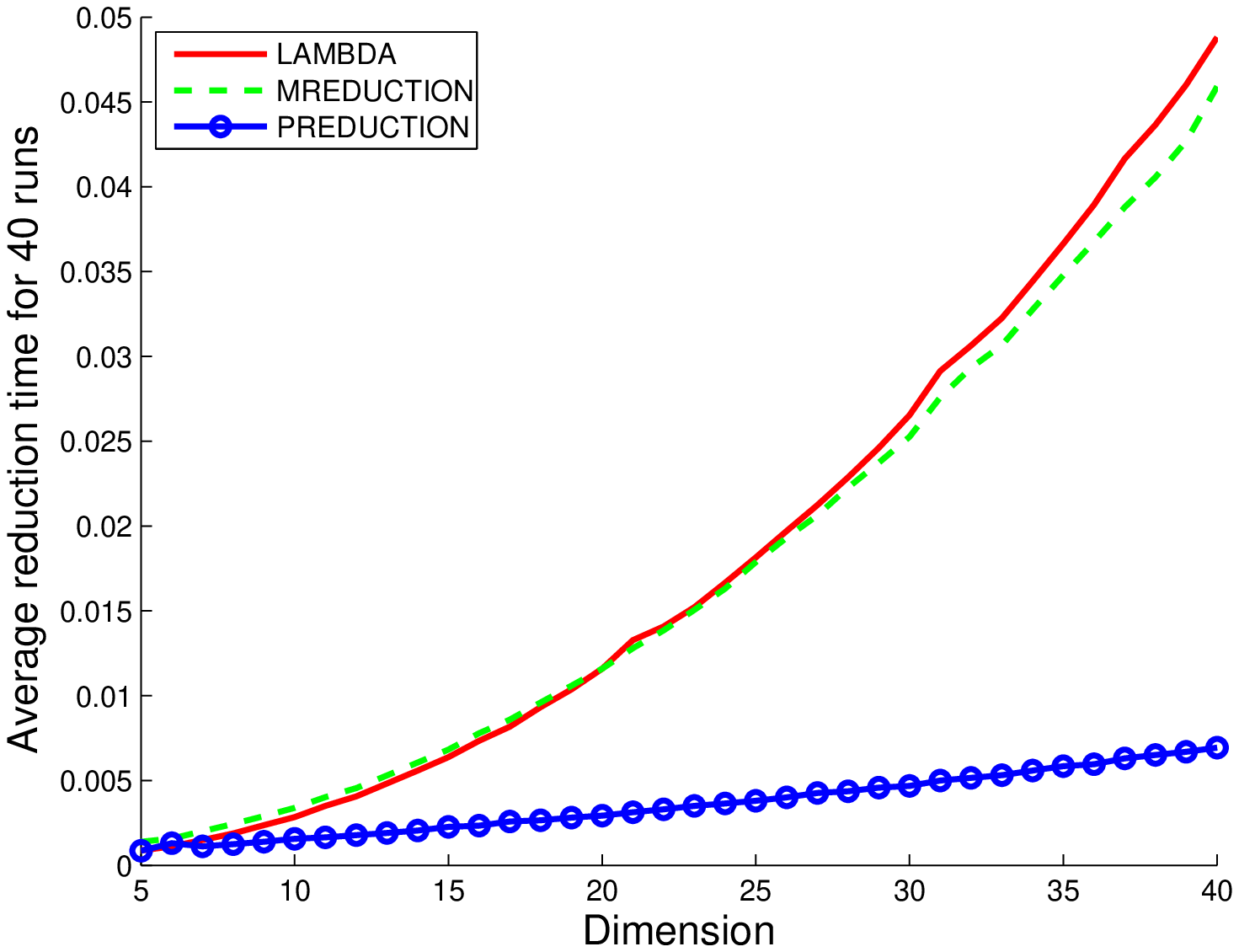}}
{\includegraphics[scale=0.80]{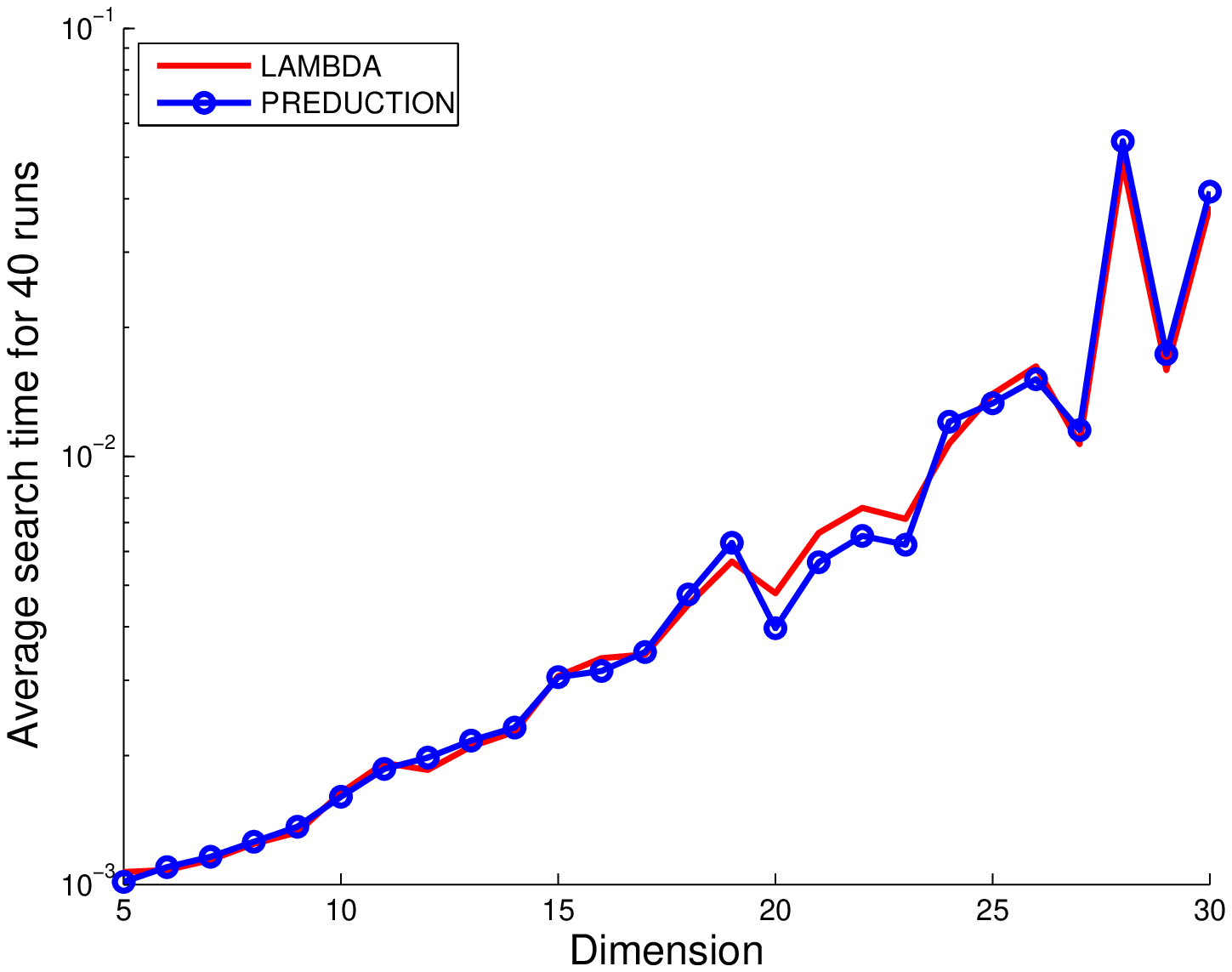}}
\caption{Running time for Case 3}\label{f:case3}
\end{figure}

\begin{figure}[ht!]
\centering
{\includegraphics[scale=0.80]{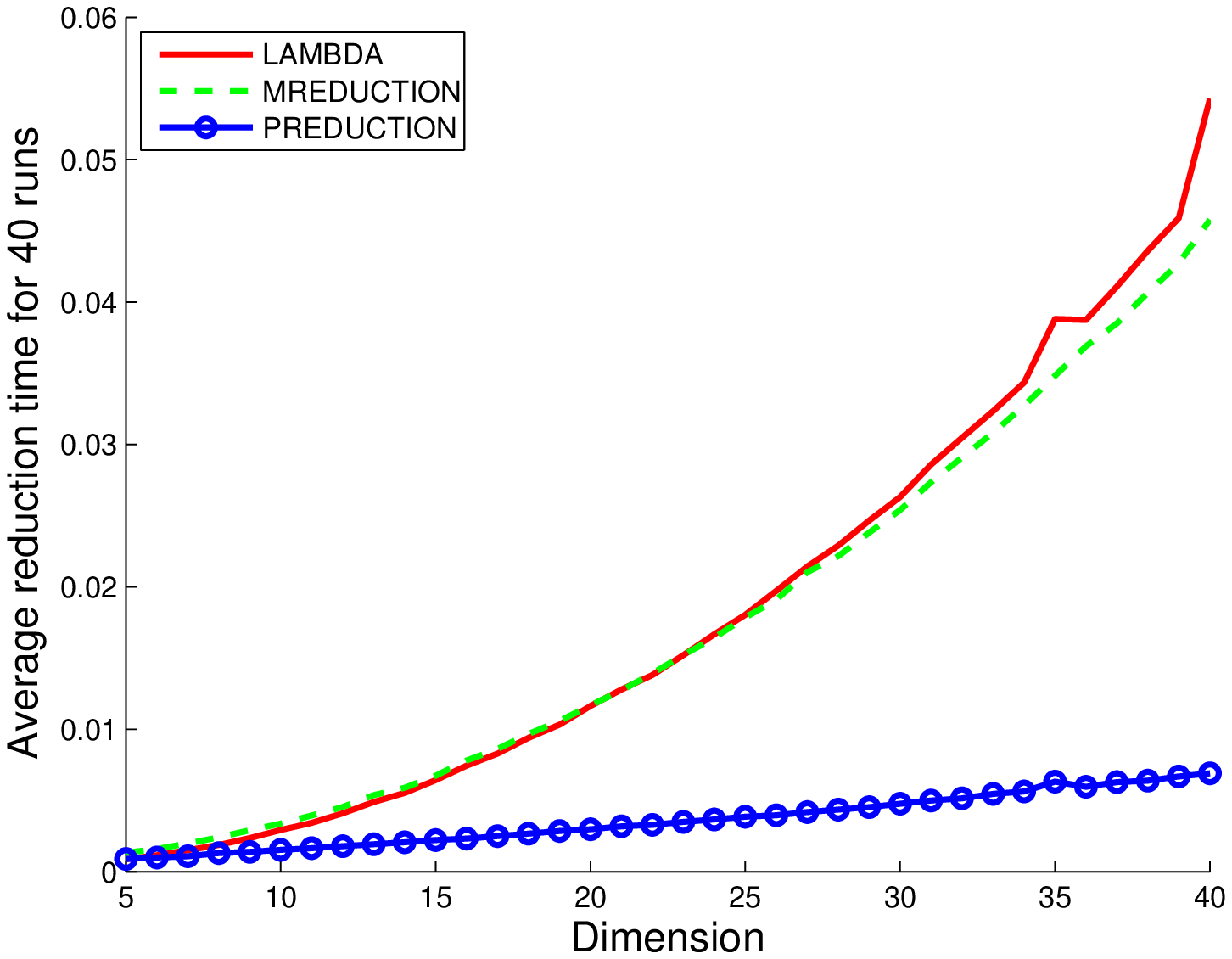}}
{\includegraphics[scale=0.80]{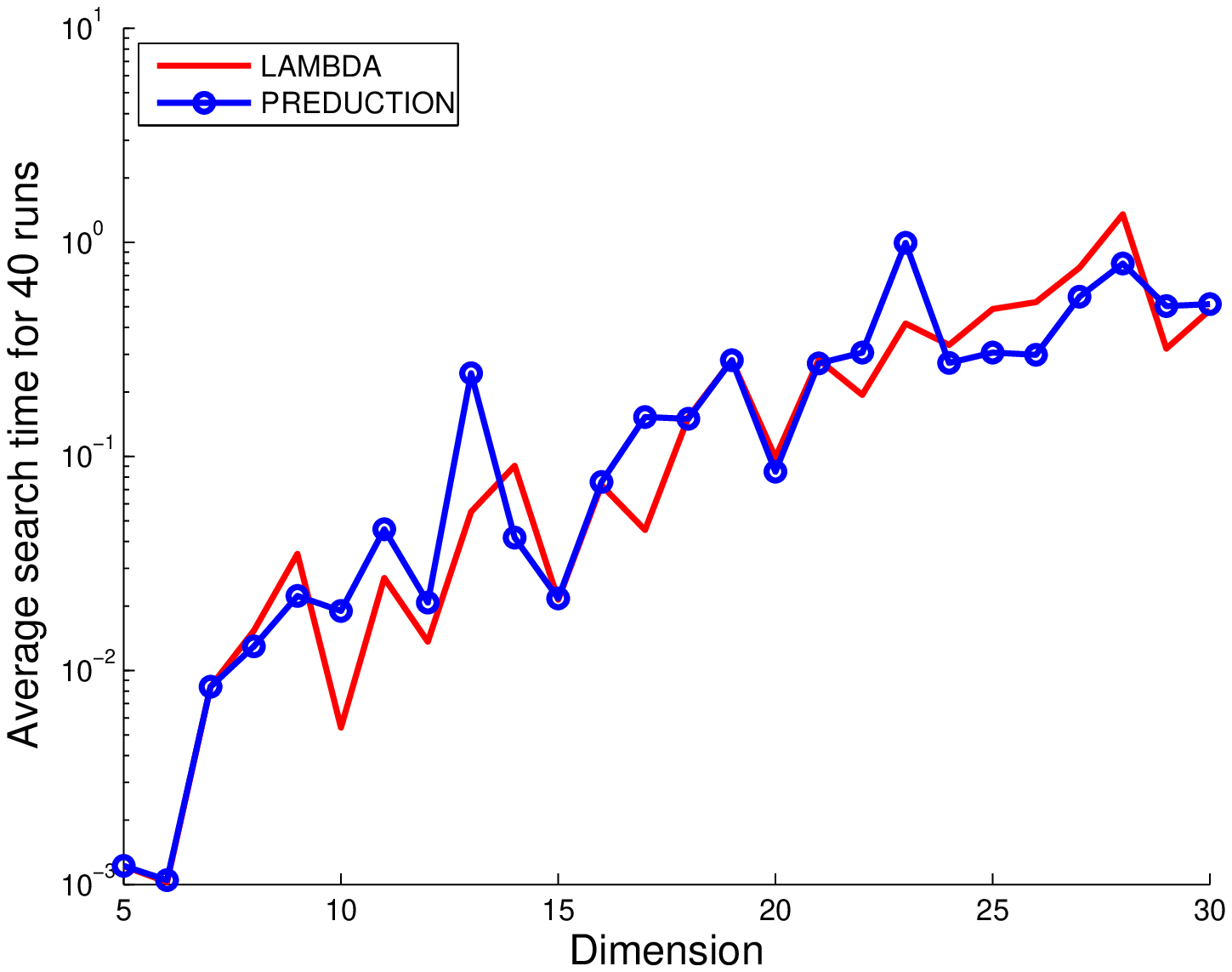}}
\caption{Running time for Case 4} \label{f:case4}
\end{figure}

\begin{figure}[ht!]
\centering
{\includegraphics[scale=0.80]{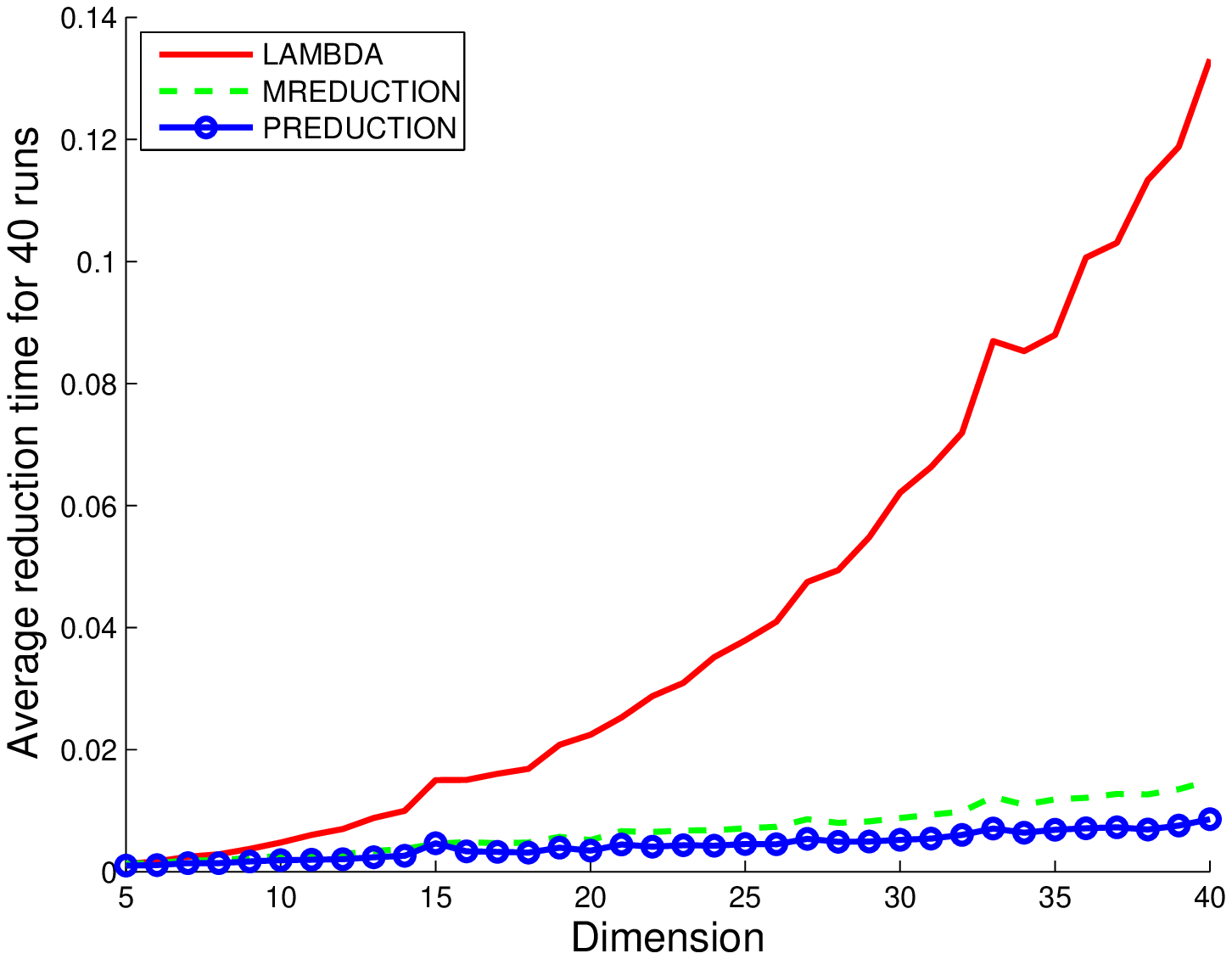}}
{\includegraphics[scale=0.80]{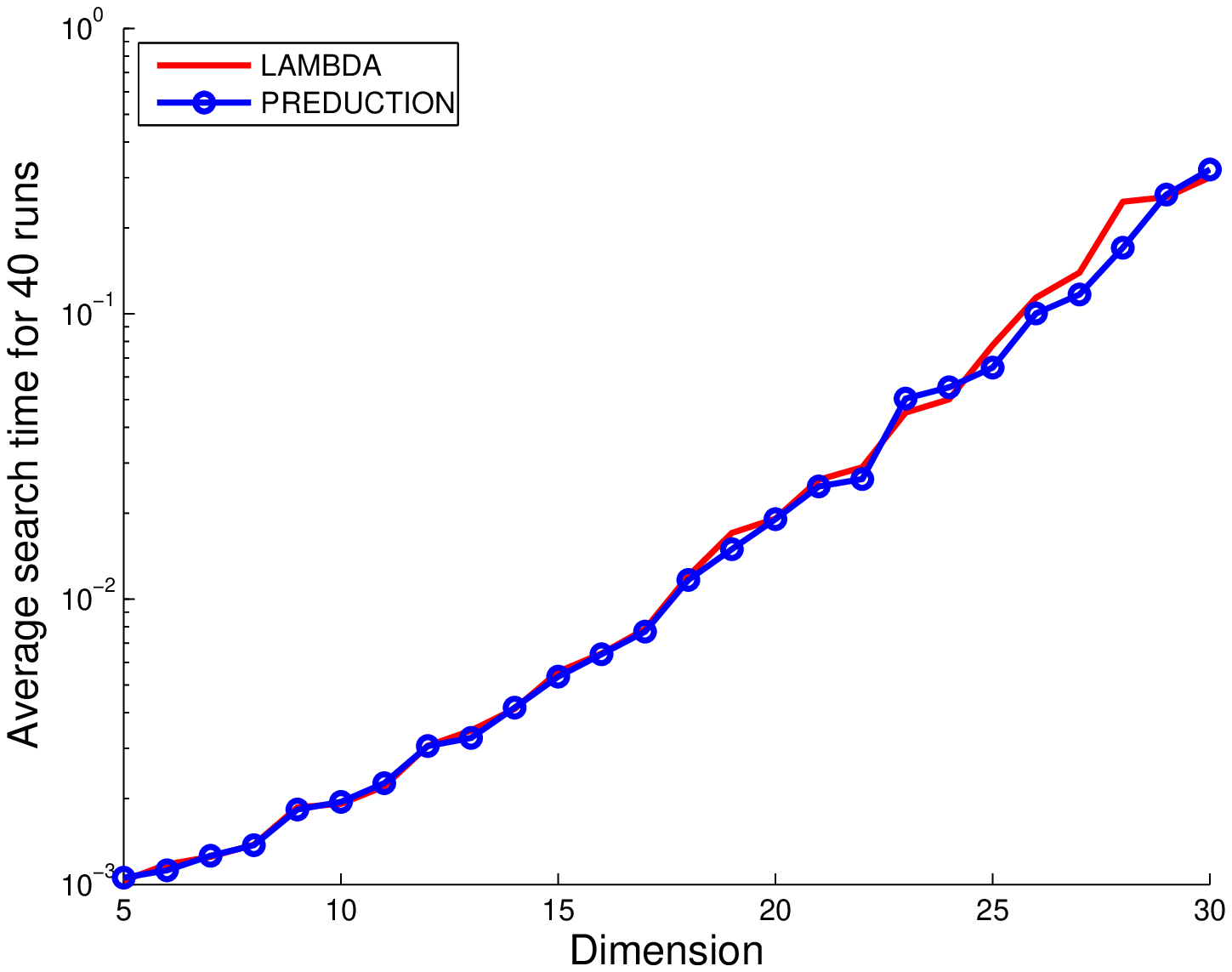}}
\caption{Running time for Case 5} \label{f:case5}
\end{figure}

\begin{figure}[ht!]
\centering
{\includegraphics[scale=0.80]{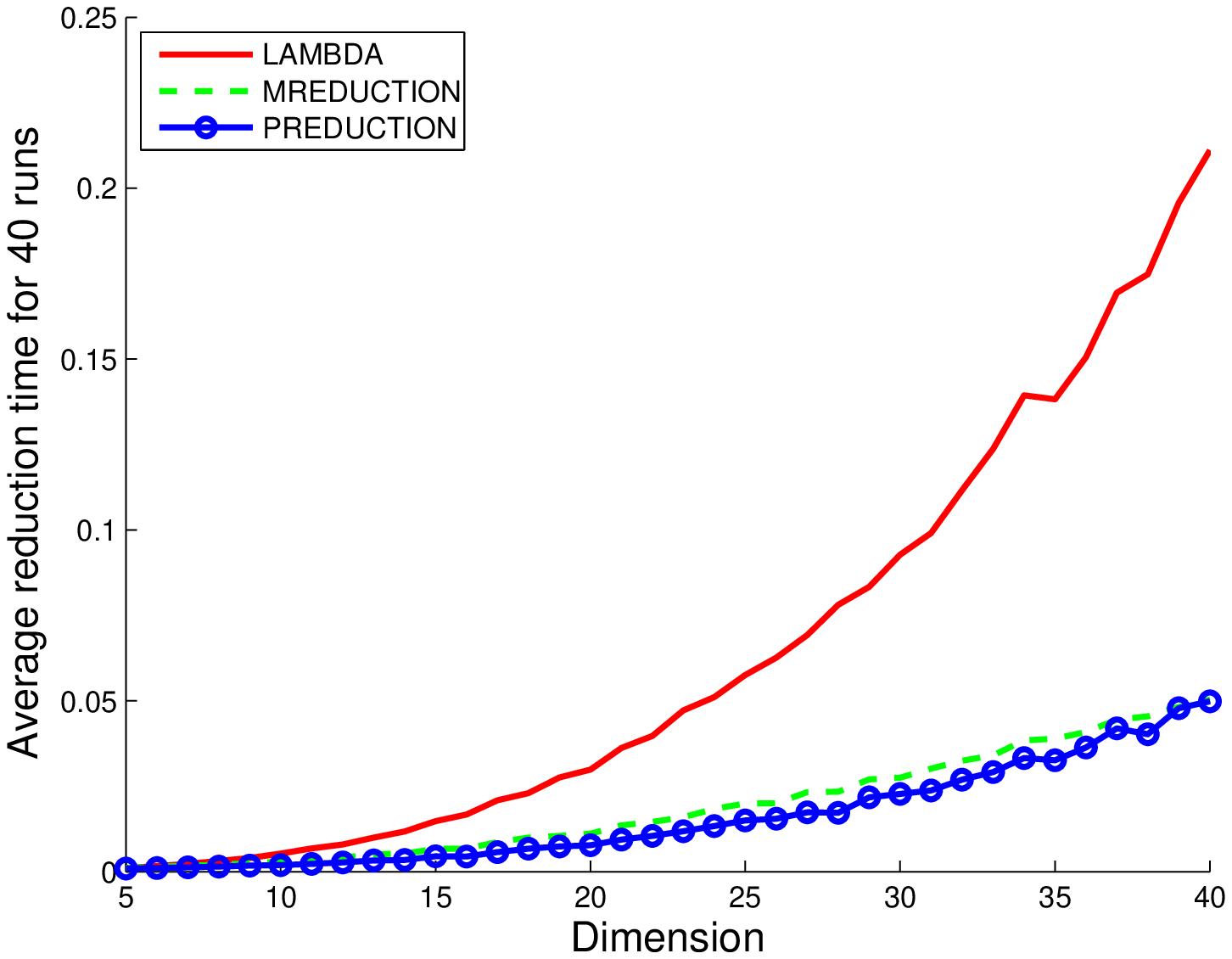}}
{\includegraphics[scale=0.80]{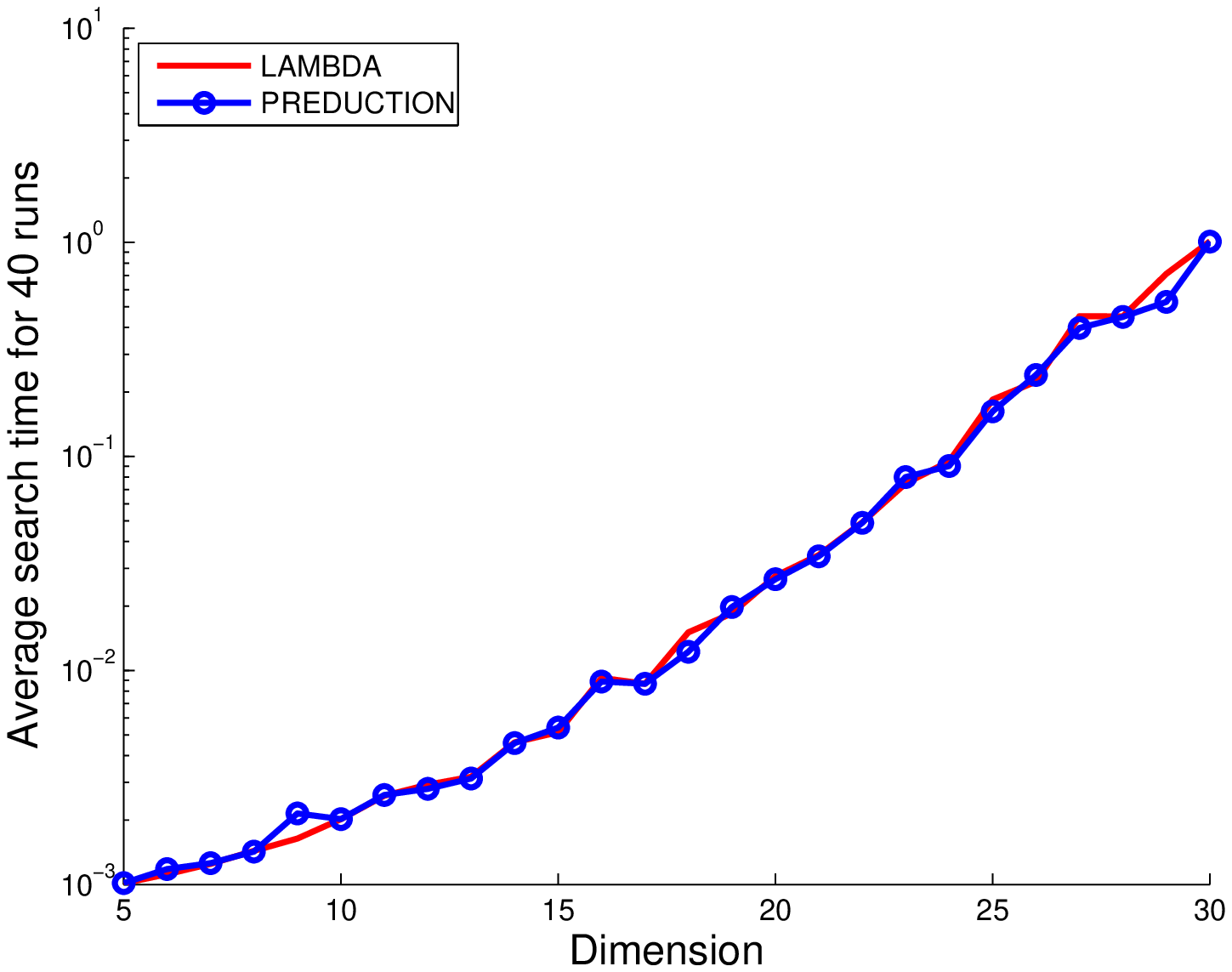}}
\caption{Running time for Case 6} \label{f:case6}
\end{figure}

\begin{figure}[ht!]
\centering
{\includegraphics[scale=0.80]{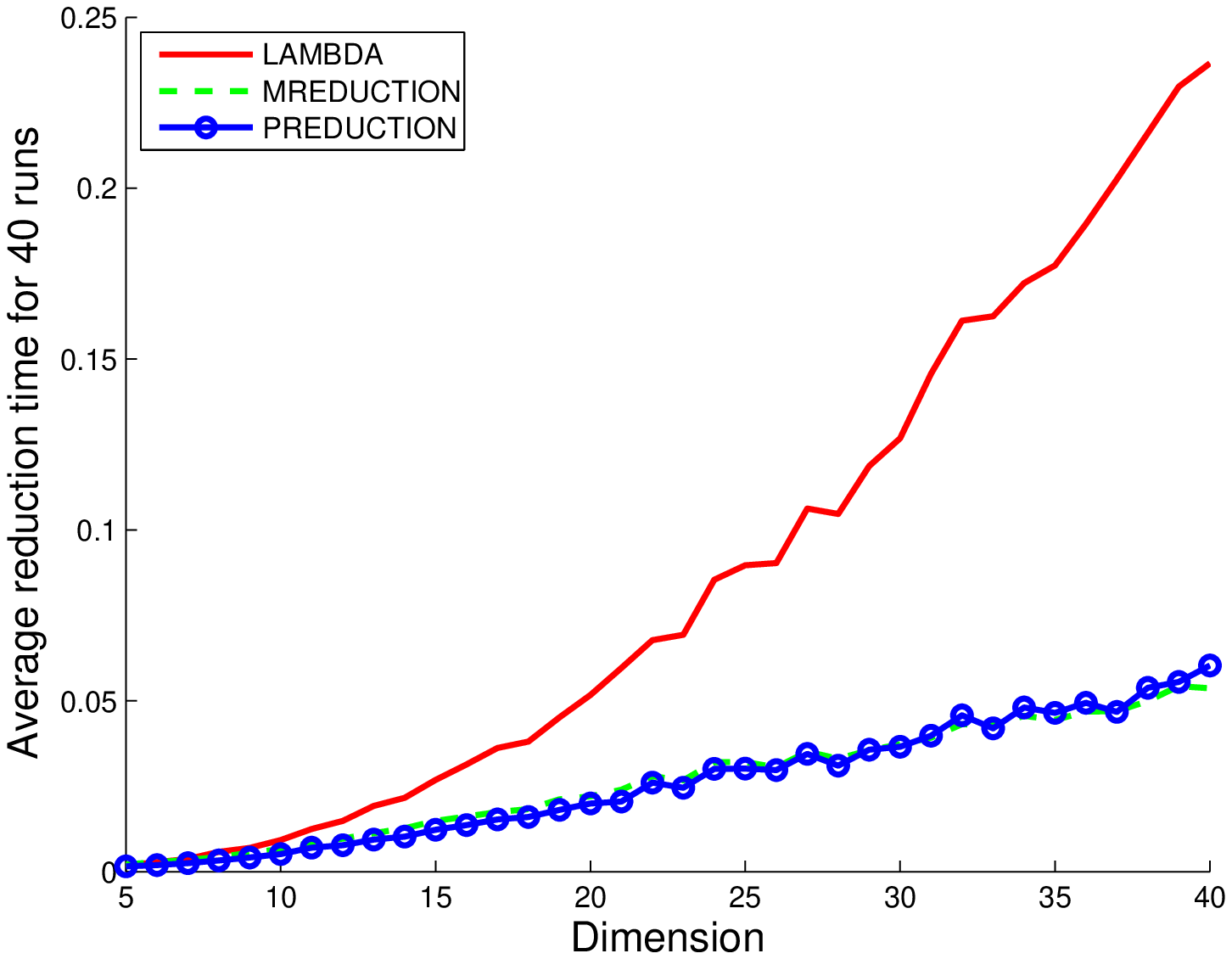}}
{\includegraphics[scale=0.80]{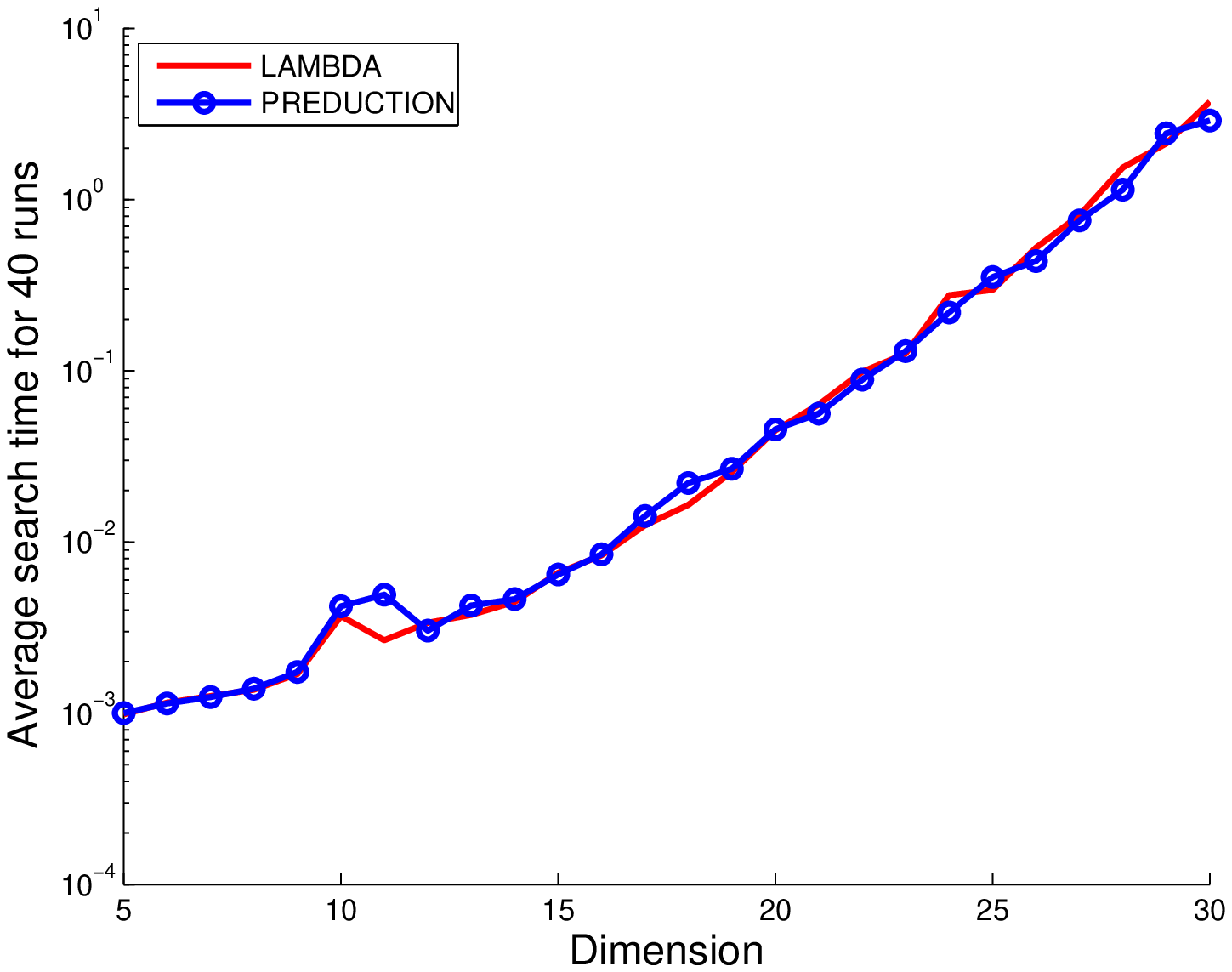}}
\caption{Running time for Case 7} \label{f:case7}
\end{figure}

\begin{figure}[ht!]
\centering
{\includegraphics[scale=0.80]{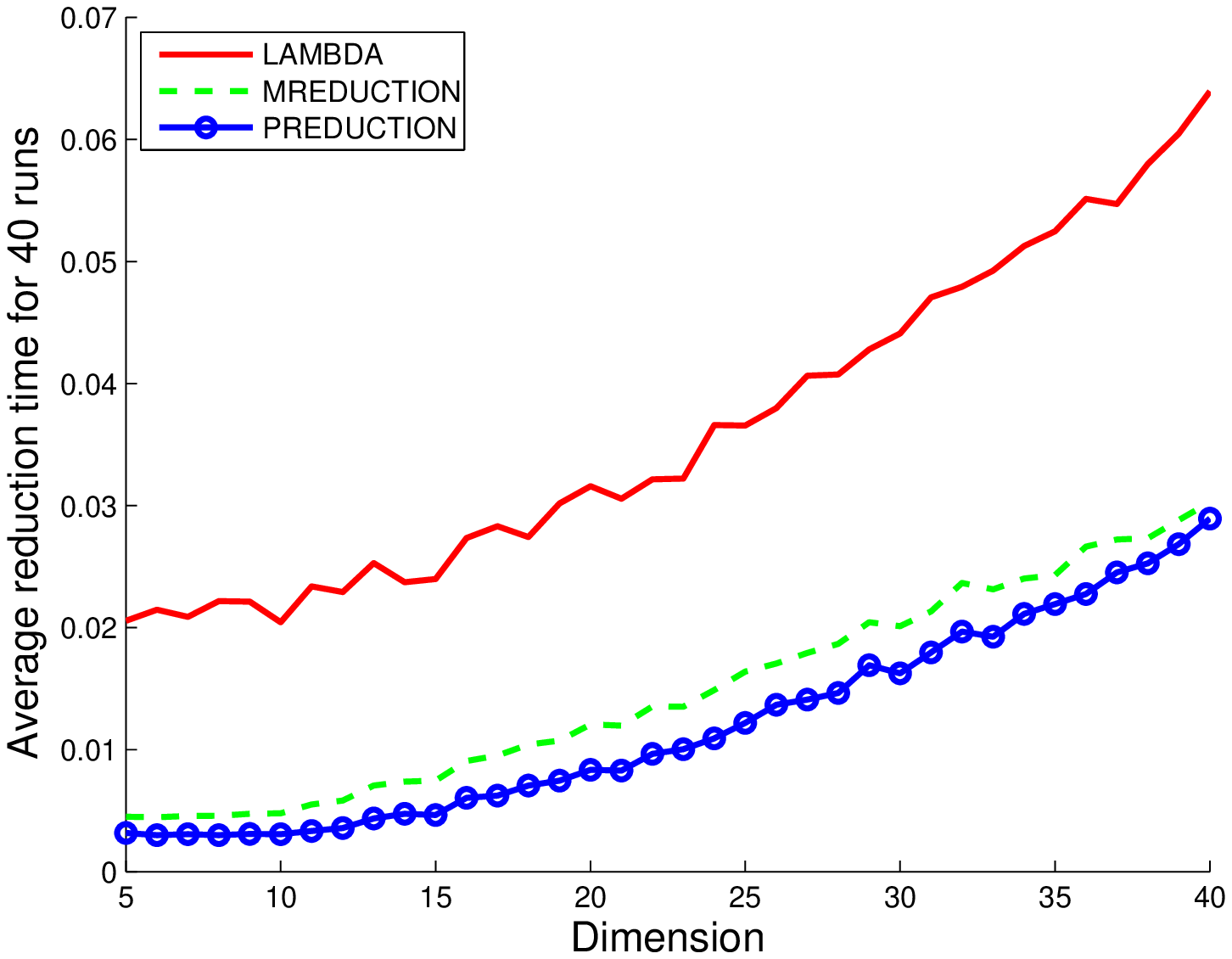}}
{\includegraphics[scale=0.80]{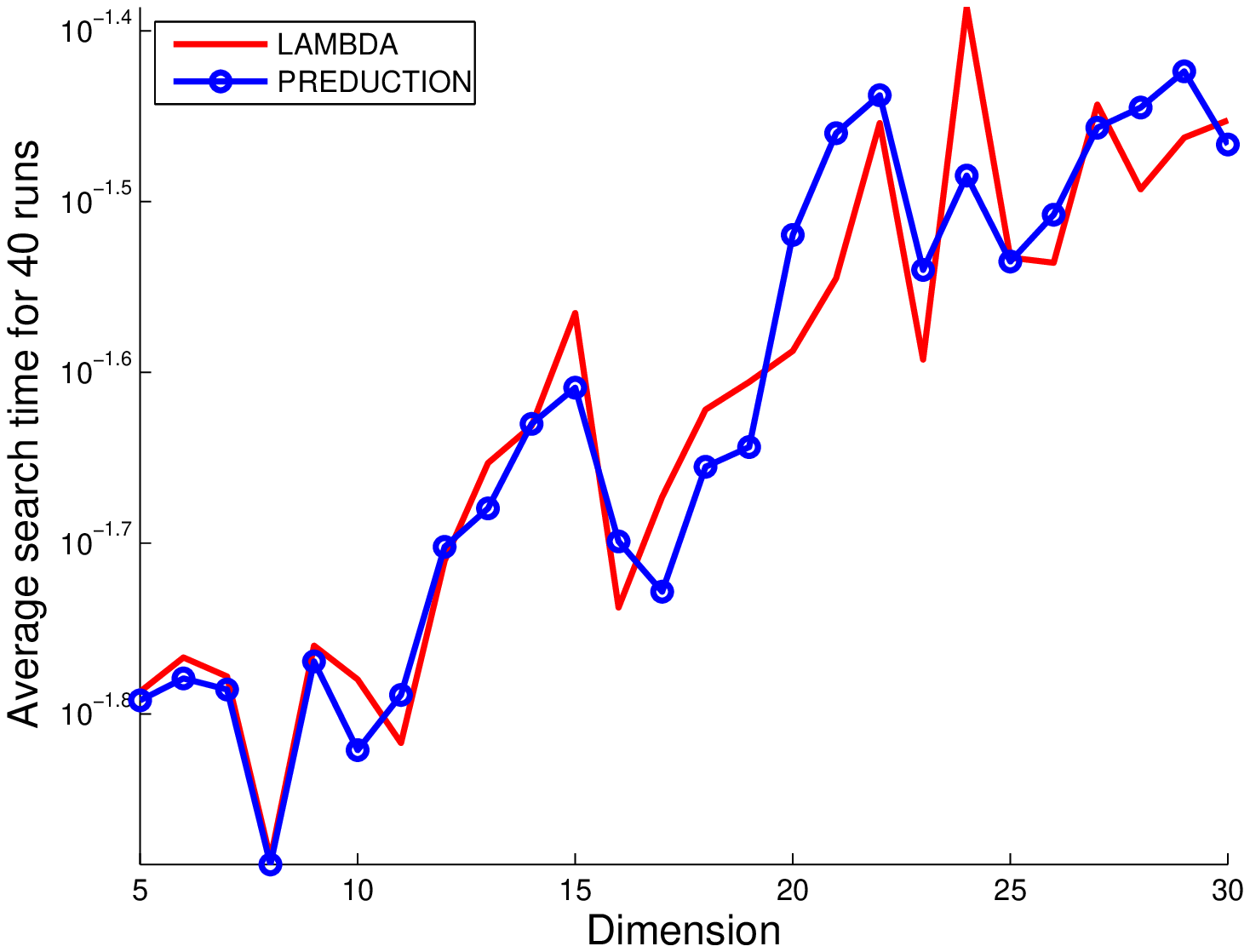}}
\caption{Running time for Case 8} \label{f:case8}
\end{figure}

\begin{figure}[ht!]
\centering
{\includegraphics[scale=0.80]{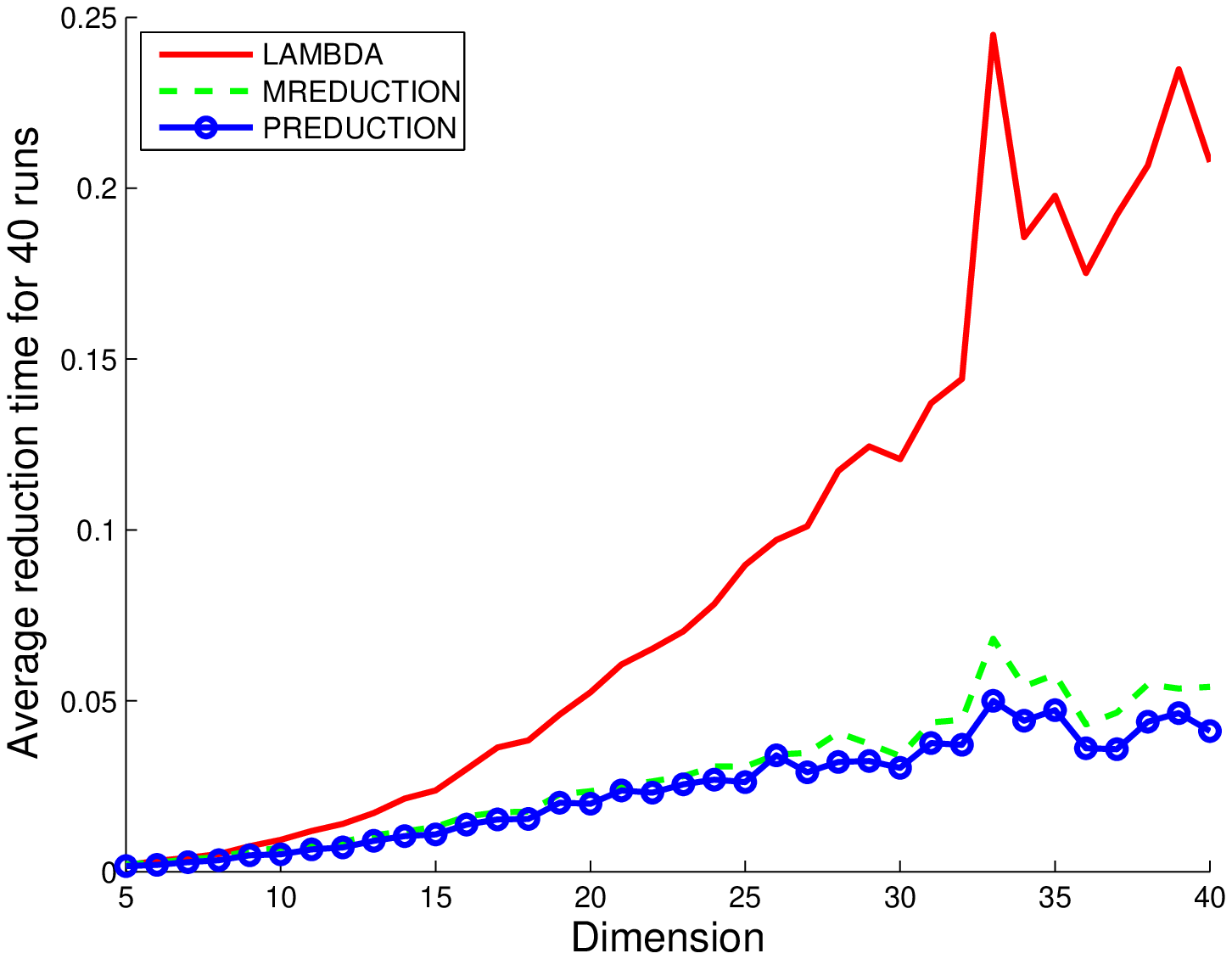}}
{\includegraphics[scale=0.80]{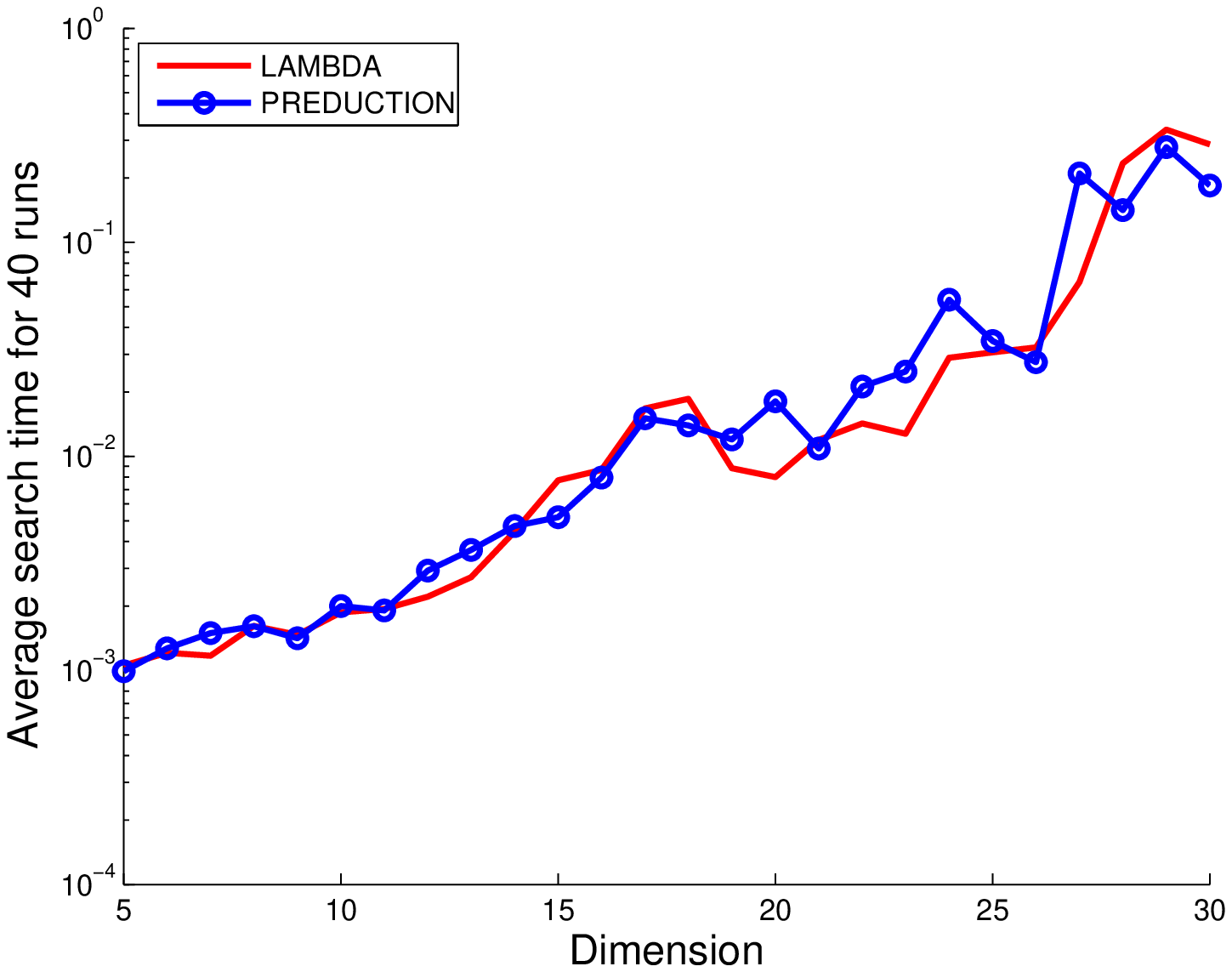}}
\caption{Running time for Case 9} \label{f:case9}
\end{figure}

\begin{figure}[ht!]
\centering
{\includegraphics[scale=0.75]{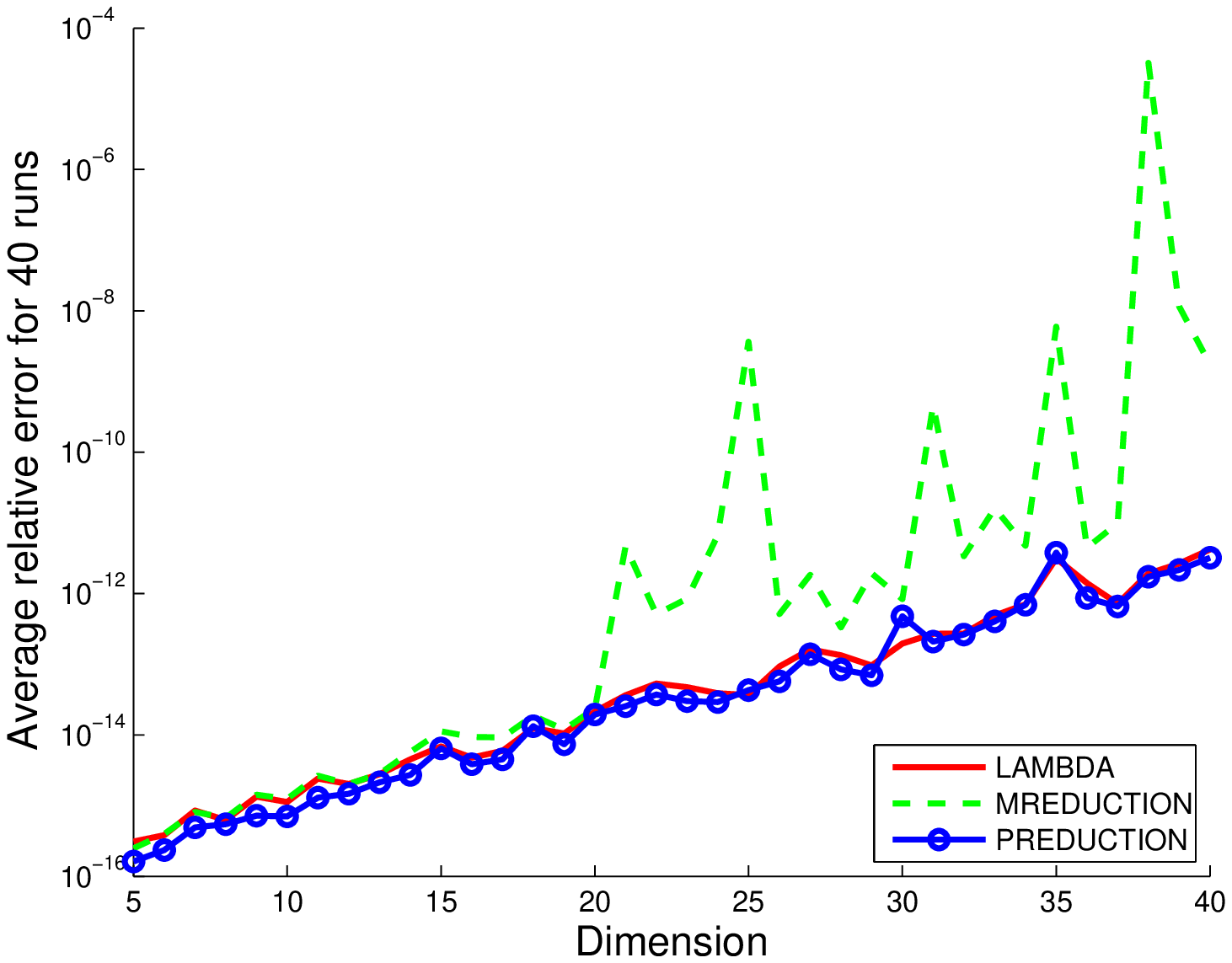}}
\caption{Relative backward error for Case 1} \label{f:case1e}
\end{figure}
\begin{figure}[ht!]
\centering
{\includegraphics[scale=0.75]{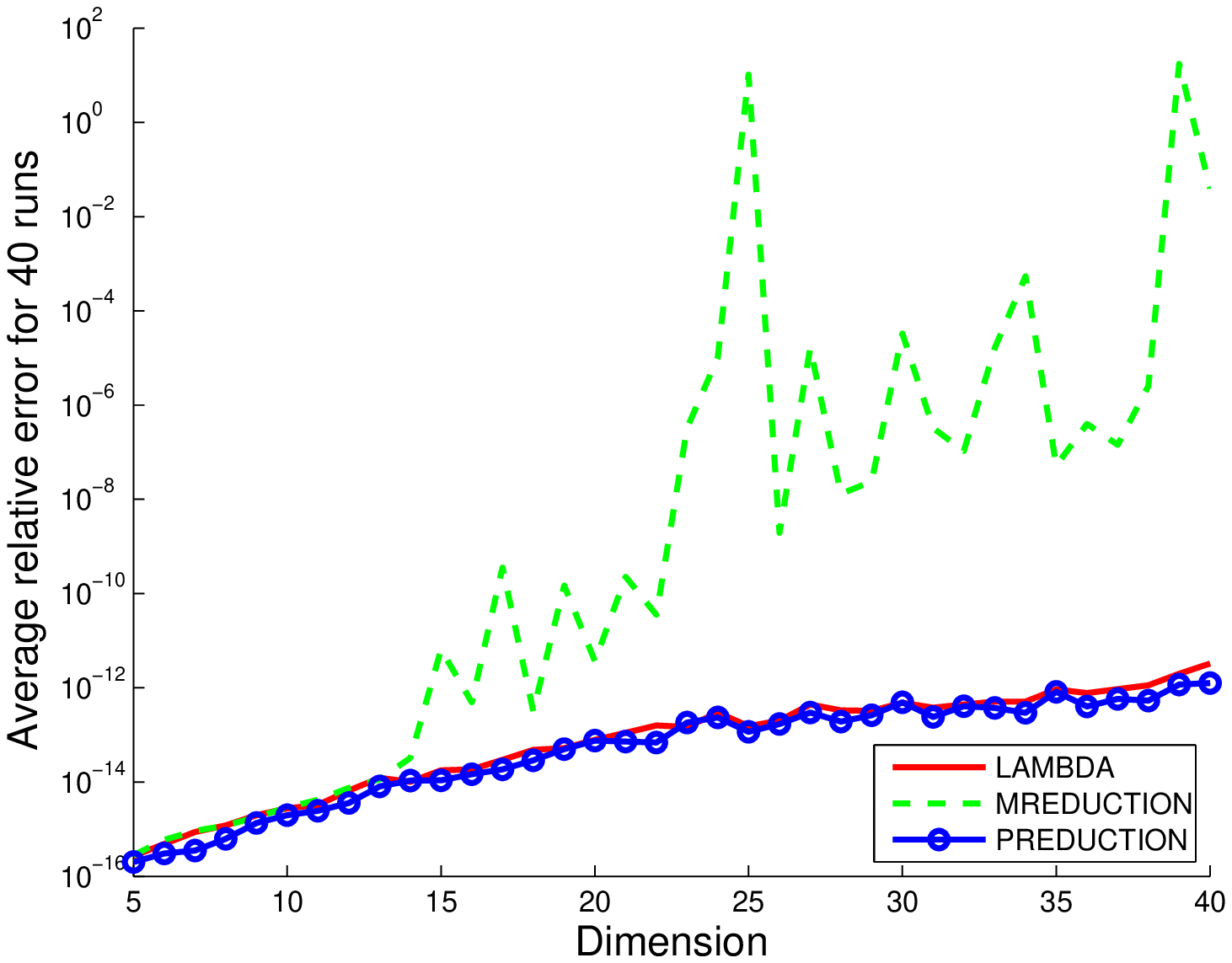}}
\caption{Relative backward error for Case 2} \label{f:case2e}
\end{figure}

\begin{figure}[ht!]
\centering
{\includegraphics[scale=0.75]{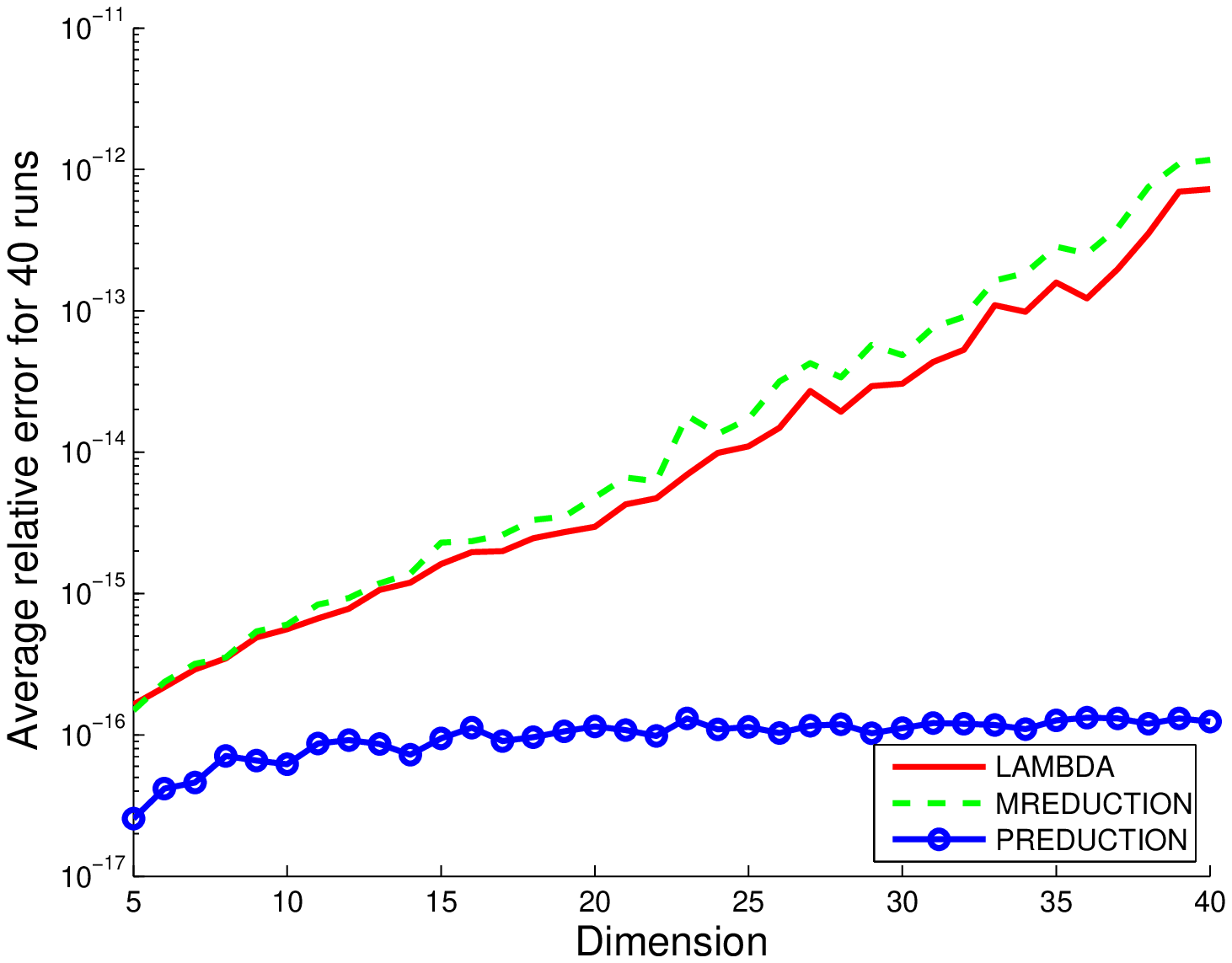}}
\caption{Relative backward error for Case 3} \label{f:case3e}
\end{figure}
\begin{figure}[ht!]
\centering
{\includegraphics[scale=0.75]{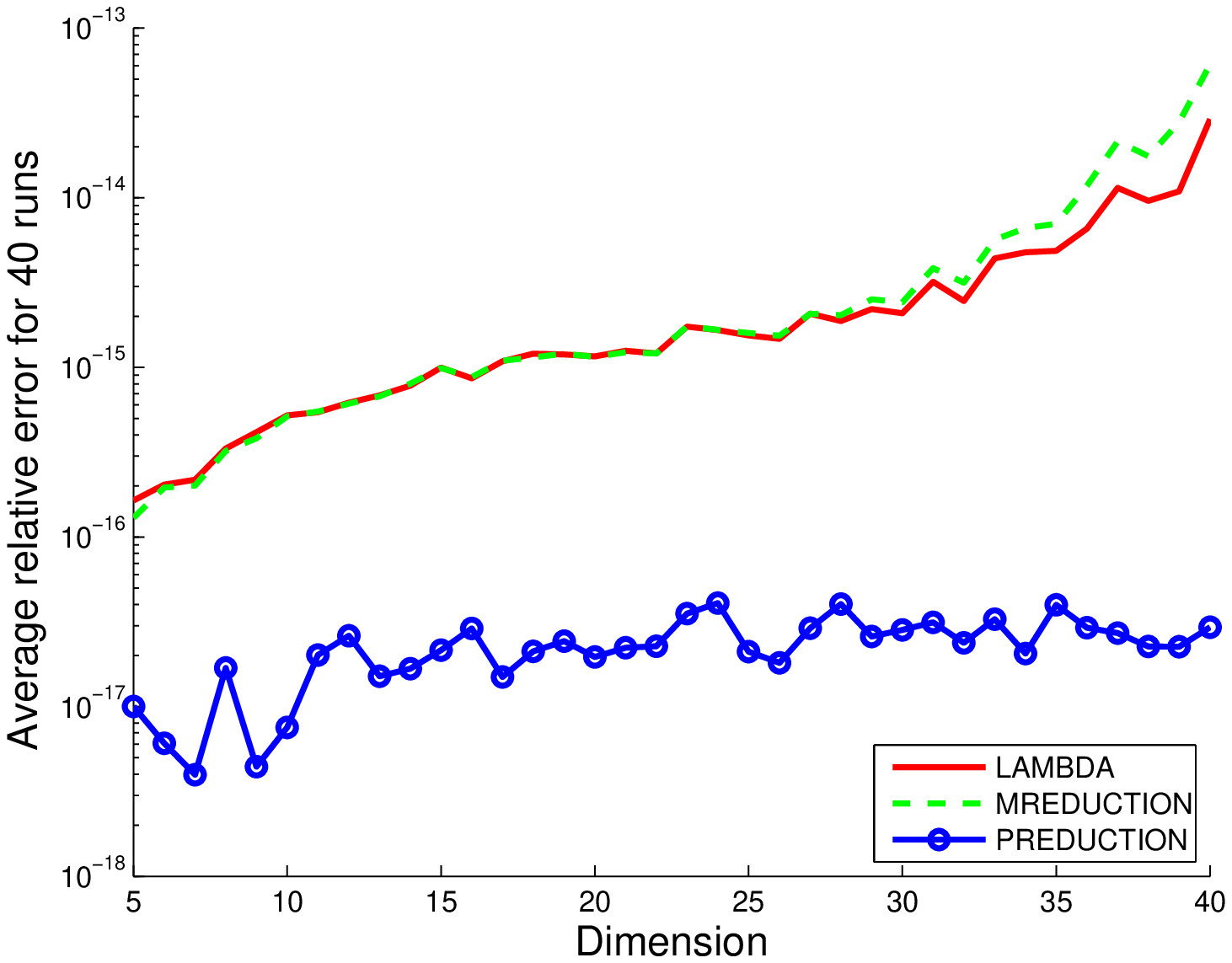}}
\caption{Relative backward error for Case 4} \label{f:case4e}
\end{figure}

\begin{figure}[ht!]
\centering
{\includegraphics[scale=0.75]{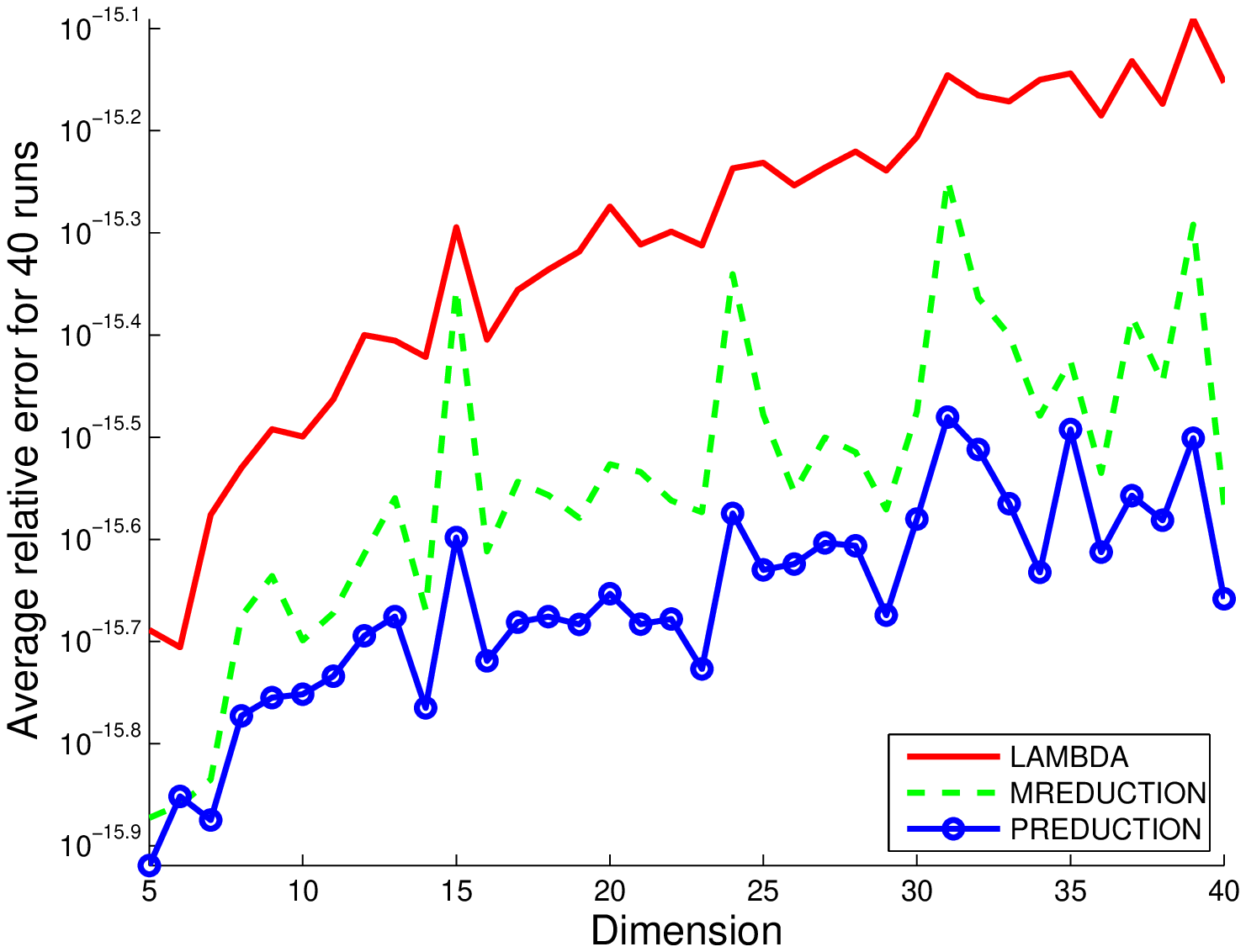}}
\caption{Relative backward error for Case 5} \label{f:case5e}
\end{figure}
\begin{figure}[ht!]
\centering
{\includegraphics[scale=0.75]{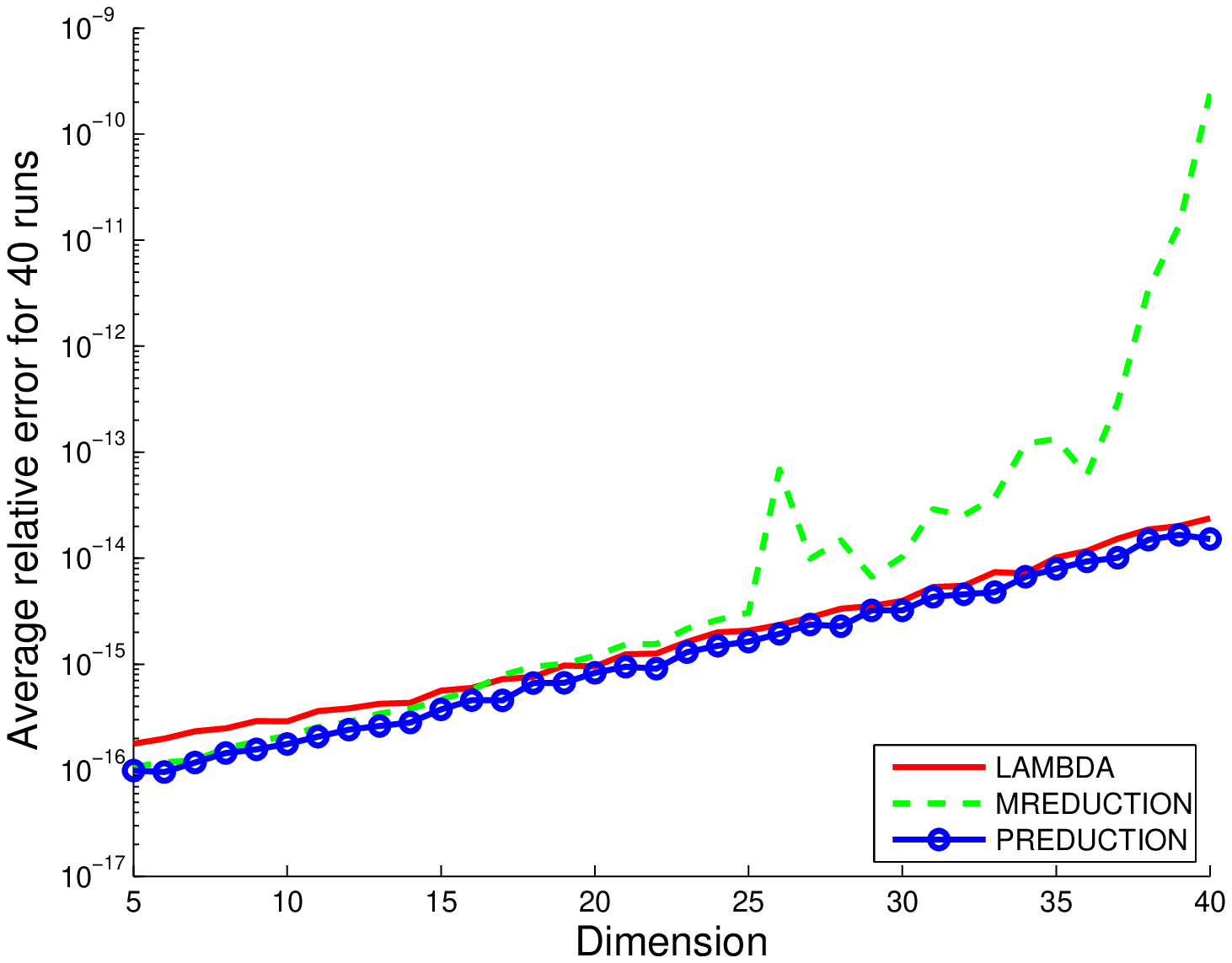}}
\caption{Relative backward error for Case 6} \label{f:case6e}
\end{figure}

\begin{figure}[ht!]
\centering
{\includegraphics[scale=0.75]{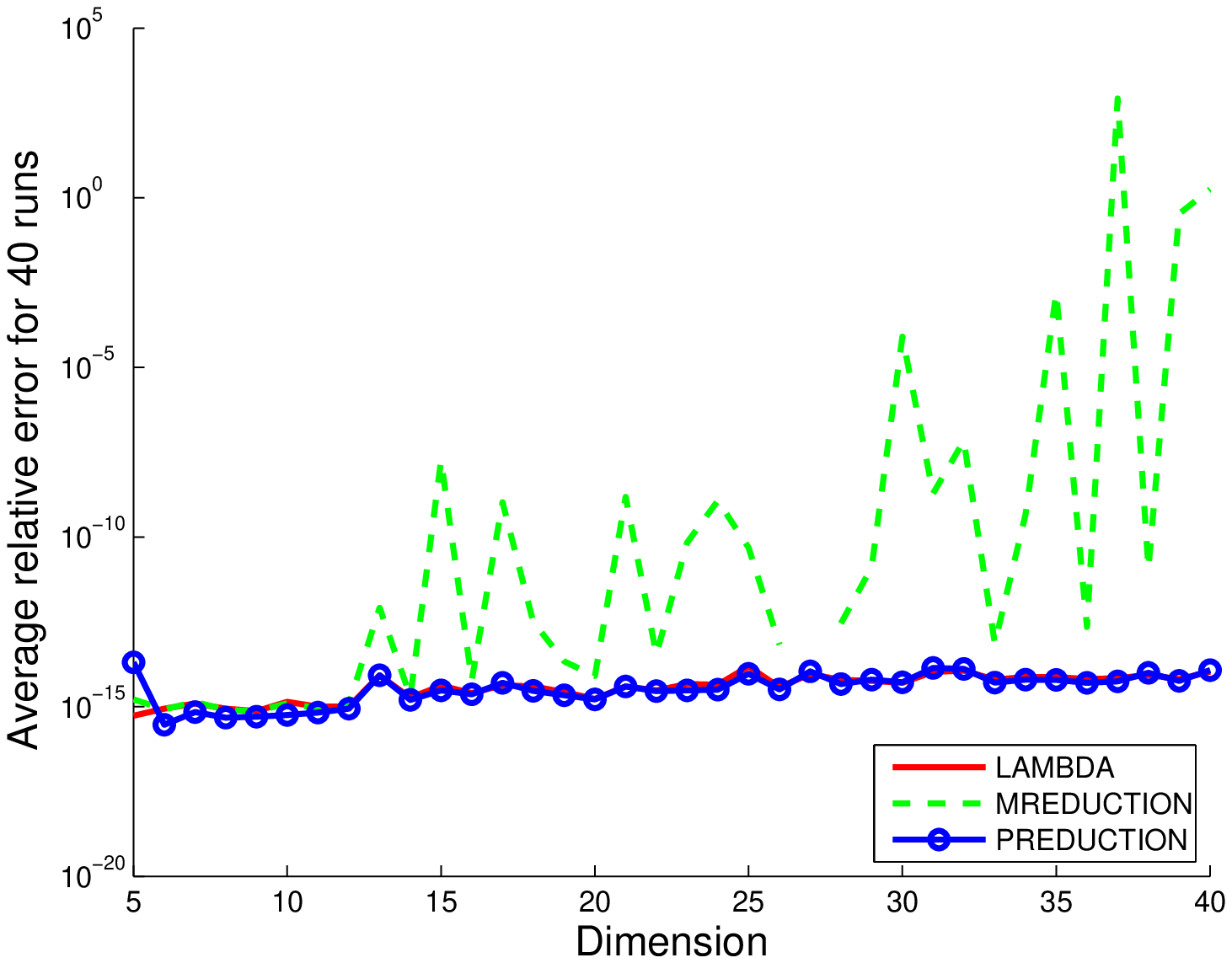}}
\caption{Relative backward error for Case 7} \label{f:case7e}
\end{figure}
\begin{figure}[ht!]
\centering
{\includegraphics[scale=0.75]{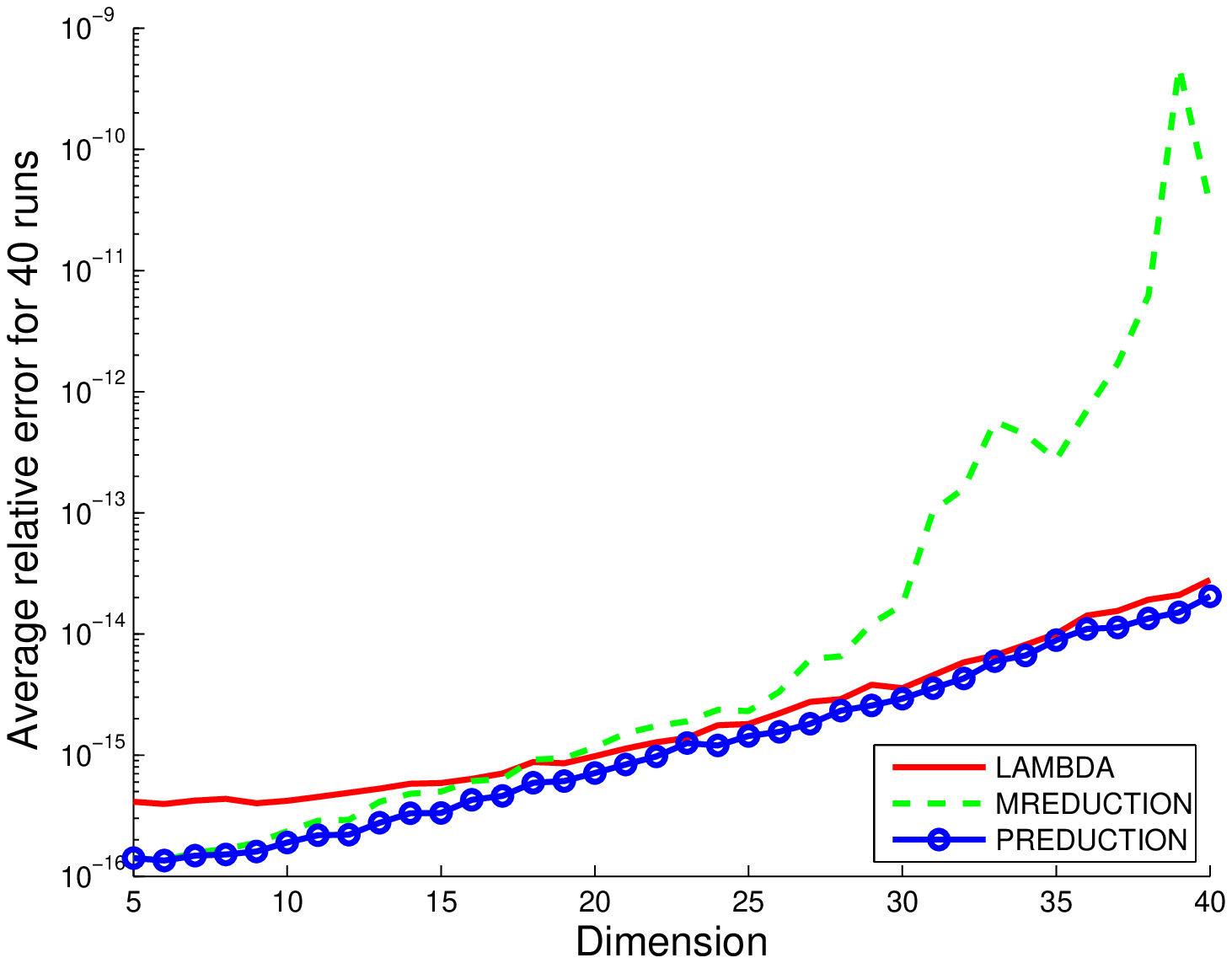}}
\caption{Relative backward error for Case 8} \label{f:case8e}
\end{figure}

\begin{figure}[ht!]
\centering
{\includegraphics[scale=0.75]{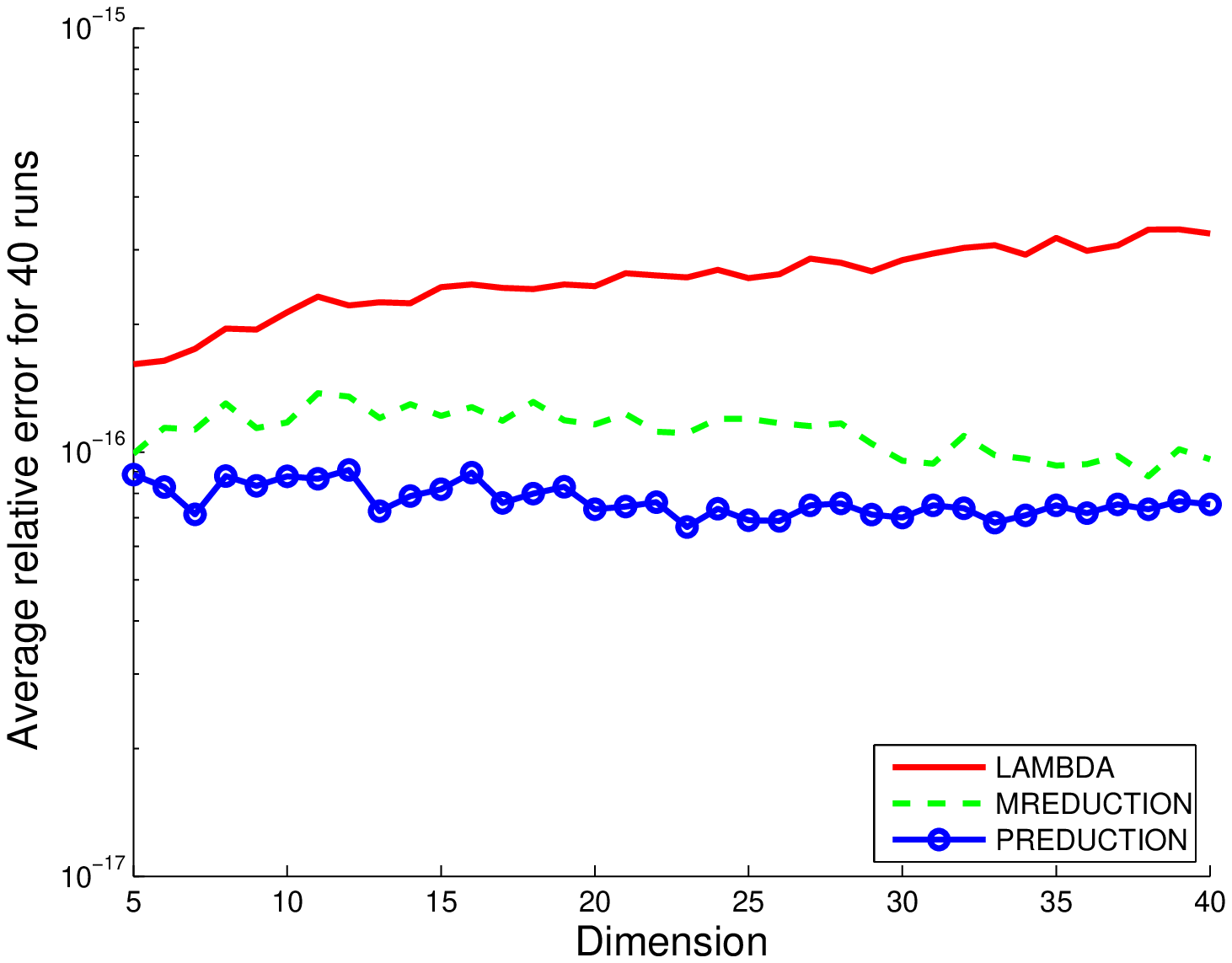}}
\caption{Relative backward error for Case 9} \label{f:case9e}
\end{figure}

\newpage

\section{Condition number criterion}\label{s:conditionNumber}

In some GNSS literature (see, e.g., \cite{LiuHZO99}, \cite{LouG03} and \cite{Xu01}), it is believed that the goal of reduction process is to reduce the condition number of the covariance matrix $\W_{\hbx}$.  The 2-norm condition number of $\W_{\hbx}$ is defined as (see \cite[Sect.\ 2.7]{GolV96})
\begin{equation*}
\kappa(\W_{\hbx}) = \norm{\W_{\hbx}} \norm{\W_{\hbx}^{-1}} = \sigma_1(\W_{\hbx})/\sigma_n(\W_{\hbx}),
\end{equation*}
where $\sigma_1(\W_{\hbx})$ and $\sigma_n(\W_{\hbx})$ are the smallest and largest singular values of $\W_{\hbx}$.  Geometrically, the condition number corresponds to the square of the ratio of the major and minor axes of the search ellipsoid (see \cite[Sect.\ 6.4]{StrB97}).  In other words, the condition number measures the elongation of the search ellipsoid.  As seen in Section \ref{s:geometry}, decorrelating the ambiguities makes the search ellipsoid less elongated.  Thus, a lower condition number of the covariance matrix indicates that the ambiguities are more decorrelated.
However, our contribution has shown that the goal of the reduction process is not to decorrelate the ambiguities as much as possible.     
 We now provide an example which shows that the condition number criterion can be misleading.  Let
\begin{equation*} 
\W_{\hbx} = \L^T\D\L = 
\left(
\begin{array}{cc}
1 & 1000.5 \\
0 &  1
\end{array}
\right)
\left(
\begin{array}{cc}
4 &  0 \\
0 &  0.05
\end{array}
\right)
\left(
\begin{array}{cc}
1 & 0 \\
1000.5 &  1
\end{array}
\right).
\end{equation*}
  
The condition number of $\W_{\hbx}$ is $1.2527 \times 10^{10}$.  If we apply a lower triangular IGT to decorrelate the 1st and 2nd ambiguity, our new covariance matrix is
\begin{equation*} 
\W_{\hbz} = \L^T\D\L = 
\left(
\begin{array}{cc}
1 & 0.5 \\
0 &  1
\end{array}
\right)
\left(
\begin{array}{cc}
4 &  0 \\
0 &  0.05
\end{array}
\right)
\left(
\begin{array}{cc}
1 & 0 \\
0.5 &  1
\end{array}
\right).
\end{equation*}

The condition number of $\W_{\hbz}$ is $80.5071$.  This shows that to solely decorrelate the ambiguities can drastically lower the condition number of the covariance matrix, yet it will not yield any improvement towards the search process as shown in Section \ref{s:proof}.  
If we permute the ambiguities, the $\mathrm{L^TDL}$ factorization becomes
\begin{equation*} 
\P^T \W_{\hbz} \P = \L^T\D\L = 
\left(
\begin{array}{cc}
1 & 0.0062 \\
0 &  1
\end{array}
\right)
\left(
\begin{array}{cc}
0.0498 &  0 \\
0 &  4.0125
\end{array}
\right)
\left(
\begin{array}{cc}
1 & 0 \\
0.0062 &  1
\end{array}
\right).
\end{equation*}

The condition number of $\P^T \W_{\hbz} \P$ is still 80.5071. Yet, numerical simulations indicate that the search process is faster with $\W_{\hbx}$ for randomly chosen $\hat \x$.  This is explained by the fact that order (\ref{eq:order}) is satisfied by $\W_{\hbx}$ and not $\P^T \W_{\hbz} \P$.
For this reason,  order (\ref{eq:order}) is a better criterion to evaluate the reduction process than the condition number.   
 
\section{Implications to the standard OILS form}\label{s:miscLLL}

We now translate our result on the role of lower triangular IGTs in the quadratic OILS form (\ref{eq:qils}) to the standard OILS form (\ref{eq:ILS}). 
Note that in the standard form, it is upper triangular IGTs that make the absolute values of the off-diagonal entries of $\R$ as small as possible.  
Section \ref{s:proof} showed that making $\R$ as close to diagonal as possible will not help the search process.   We know that the purpose of permutations is to strive for $r_{11} \ll \ldots \ll r_{nn}$.  This contribution shows that it is also the purpose of upper triangular IGTs.  In the reduction process, if $r_{k-1,k-1} > \sqrt{r^2_{k-1,k} + r_{kk}^2}$, then permuting columns $k$ and $k-1$ will decrease $r_{k-1,k-1}$ and increase $r_{kk}$ (see Section \ref{s:permutations}).  The purpose of an upper triangulr IGT is two-fold.  First, applying IGT $\Z_{k-1,k}$ makes $|r_{k-1,k}|$ as small as possible, which increases the likelihood of a column permutation.  Second, if a permutation occurs, from (\ref{eq:r_k-1}) and (\ref{eq:r_kk}), a smaller $|r_{k-1,k}|$ ensures a greater decrease in $r_{k-1,k-1}$ and a greater increase in $r_{kk}$.  

\subsection{A partial LLL reduction}
Our result indicates (wrongly) that IGTs have to be applied only on the superdiagonal entries of $\R$.  
Such an approach can be numerically unstable since it can produce large off-diagonal entries in $\R$ relative to the diagonal entry (see \cite{Hig89}).
Some other IGTs are needed to bound the off-diagonal entries of $\R$.  The chosen approach is as follows.  If columns $k$ and $k-1$ are going to be permuted, we first apply IGTs to make the absolute values of $r_{k-1,k}, \ldots, r_{1k}$ as small as possible; otherwise, we do not apply any IGT.  This way, the number of IGTs is minimized.  
Based on this idea, we present a partial LLL 
(PLLL) reduction algorithm.  

\begin{algorithm} \label{a:MLLL}
\textnormal{(PLLL Reduction). Given generator matrix $\A \in \mathbb{R}^{m \times n}$ and the input vector $\y \in \mathbb{R}^{m}$.
The algorithm returns the reduced upper triangular matrix $\R \in \mathbb{R}^{n \times n}$, the 
unimodular matrix $\Z \in \mathbb{Z}^{n \times n}$, and the vector $\by \in \mathbb{R}^{n}$. 
\begin{tabbing}
$\quad$\= {\bf function}: $[\R, \Z,\by] = \mathrm{PLLL}(\A, \y)$ \\
\> Compute the sorted QR decomposition of $\A$ (see \cite{WubBR01})\\
\> and set $\by = \W_{1}^T\y$ \\
\> $\Z=\I$ \\
\> $k=2$ \\ 
\> {\bf while} \= $k \le n$  \\
\>            \>        $r_t = r_{k-1,k} - \round{r_{k-1,k}/r_{k-1,k-1}} \times r_{k-1,k-1} $ \\
\>            \>        {\bf if} \=  $r_{k-1,k-1} > \sqrt{r_t^2 + r_{kk}^2}$ \\ 
\>            \>	\>  {\bf for} \= $i=k-1\!:\!-1\!:\!1$ \\
\>            \>        \>   \> \mbox{Apply IGT $\Z_{ik}$ to $\R$, .i.e., $\R=\R \Z_{ik}$} \\
\>            \>        \>   \> \mbox{Update $\Z$, i.e., $\Z= \Z\Z_{ik}$} \\
\>	      \>        \> {\bf end} \\
\>	      \>                     \> \mbox{Interchange columns $k$ and $k-1$ of $\R$ and $\Z$} \\ 
\>	      \>	             \> \mbox{Transform $\R$ to an upper triangular matrix by a Givens rotation} \\
\>	      \>                     \> \mbox{Apply the same Givens rotation to $\by$ } \\
\>	      \>		     \> {\bf if} \= $k > 2$ \\
\>            \>                     \> 	\> $k = k - 1$ \\
\>	      \>		     \> {\bf end} \\
\>            \> {\bf else}\\
\>            \>         \> $k=k+1$ \\
\>            \> {\bf end} \\
\> {\bf end}
\end{tabbing}
}
\end{algorithm}    

We point out that our final $\R$ might not be LLL-reduced since we do not ensure property (\ref{eq:LLLprop1}).  

\subsection{Ling's and Howgrave-Graham's effective LLL reduction}
The initial version of this thesis was submitted before finding the effective
LLL reduction presented in \cite{LinN07}.
We now point out the similarities and differences with the results
derived in this chapter.
Firstly, they show that IGTs do not affect the Babai point, while we show that IGTs do not affect the entire search process.
Secondly, their modified LLL reduction algorithm uses Gram-Schmidt orthogonalization.
Our modified LLL reduction algorithm uses Householder reflections and Givens rotation. The latter is more stable.
Thirdly, our algorithm  does not apply an IGT if no column permutation is needed as it is unnecessary.
Finally, they do not take numerical stability into account, while we do.
Specifically, our reduction algorithm uses extra IGTs to prevent the serious rounding errors that can occur due to the increase of the off-diagonal elements.

\chapter{Ellipsoid-Constrained Integer Least Squares Problems}\label{s:EILS}

In some applications, one wants to solve
\begin{equation}\label{eq:EILS}
\min_{\x \in \mathcal{E}} \| \y-\A \x \|_2^2, \quad \mathcal{E} = \{ \x \in \mathbb{Z}^n: \| \A \x \|_2^2 \le \alpha^2 \},
\end{equation}
where $\y \in \mathbb{R}^n$ and $\A \in \mathbb{R}^{m \times n}$ has full column rank.
We refer to (\ref{eq:EILS}) as an ellipsoid-constrained integer least squares (EILS) problem.  
In \cite{DamEC03}, the V-BLAST reduction (see \cite{FosGVW99}) was proposed for the reduction process of the EILS problem.  In \cite{ChaG09}, it was shown that the LLL reduction makes the search process more efficient than V-BLAST.   It was also noted that for large noise in the linear model (see \eqref{eq:linearmodel}), the search process becomes extremely time-consuming both with V-BLAST and the LLL reduction.  
In Section \ref{s:searchEILS}, we show how to modify the search process given the ellipsoidal constraint based on the work in \cite{ChaG09}.
In Section \ref{s:NEILS}, we give a new reduction strategy to handle the large noise case, which unlike the LLL reduction and V-BLAST, uses all the available information.  
Finally in Section \ref{s:simulationEILS}, we present simulation results that indicate that our new reduction algorithm is much more effective than the existing algorithms for large noise.

\section{Search process}\label{s:searchEILS} 

Suppose that after the reduction stage, the original EILS problem (\ref{eq:EILS}) is transformed to the following reduced EILS problem
\begin{equation}\label{eq:rEILS}
\min_{\z \in \bar{\mathcal{E}}} \| \by-\R \z \|_2^2, \quad \bar{\mathcal{E}} = \{ \z \in \mathbb{Z}^n: \| \R \z \|_2^2 \le \alpha^2 \}.
\end{equation} 
Without loss of generality, we assume that the diagonal entries of $\R$ are positive. 
Assume that the solution of (\ref{eq:rILS}) satisfies the bound
\begin{equation*}\label{eq:rEILS2}
\| \by-\R \z \|_2^2 < \beta^2.
\end{equation*}
Then, as in the OILS problem \eqref{eq:rILS}, we have the following inequalities (see Sect.\ \ref{s:search})
\begin{align}
& \mbox{level }n : r_{nn}^2 (z_n - c_n)^2 < \beta^2, \label{eq:EILSineq_n}\\ 
& \quad \vdots \nonumber \\
& \mbox{level }k : r_{kk}^2 (z_k - c_k)^2 < \beta^2 - \sum_{i=k+1}^n r_{ii}^2 (z_i - c_i)^2 \label{eq:EILSineq_k}\\ 
& \quad \vdots \nonumber \\
& \mbox{level }1 : r_{11}^2 (z_1 - c_1)^2 < \beta^2 - \sum_{i=2}^n r_{ii}^2 (z_i - c_i)^2,  \label{eq:EILSineq_1}
\end{align}
where $c_k$, for $k=n\!:\!1$, are defined in \eqref{eq:c_k}.
In Section \ref{s:search}, we presented a search process for the OILS problem based on these inequalities.
For the EILS problem \eqref{eq:rEILS}, the search also needs to take the constraint ellipsoid into account. 
In \cite{ChaG09},  the search given in Section \ref{s:search} was modified in order to ensure that the enumerated integer points satisfy the constraint ellipsoid.  Here, we show how to compute the bounds of the constraint.   

The constraint ellipsoid $\|\R \z \|^2_2 \le \alpha^2 $ can be written as
\begin{equation}\label{eq:ceq}
\sum_{k+1}^n (r_{kk} z_k + \sum_{j = k+1}^n r_{kj} z_j)^2 \le \alpha^2
\end{equation}
Define
\begin{equation}\label{eq:bk}
b_n = 0, \quad b_k = \sum_{j=k+1}^n r_{kj} z_j, \quad k = n-1:-1:1.
\end{equation}       
\begin{equation}\label{eq:sk}
s_n = \alpha^2, \quad s_{k-1} = \alpha^2 - \sum_{i=k}^n (r_{ii} z_i + \sum_{j=i+1}^n r_{ik} z_j)^2 = s_k - (r_{kk} z_k + b_k), \quad k=n:-1:2.
\end{equation}
With (\ref{eq:bk}) and (\ref{eq:sk}), we rewrite (\ref{eq:ceq}) as
\begin{equation*}
(r_{kk} z_k + b_k)^2 \le s_k, \quad k = n:-1:1. 
\end{equation*} 
Therefore, at level $k$ in the search process, $z_k$ is constrained to the interval 
\begin{equation} \label{ellipse_bound}
l_k \le z_k \le u_k, \quad l_k = \Big\lceil \frac{- \sqrt{s_k} - b_k}{r_{kk}} \Big\rceil, \quad u_k = \Big\lfloor \frac{ \sqrt{s_k} - b_k}{r_{kk}} \Big\rfloor,
\quad k=n:-1:1.
\end{equation} 
At each level in the search process, we compute $l_k$ and $u_k$ with (\ref{ellipse_bound}).  
If $l_k > u_k$, then no valid integer exists at level $k$ and we move up to level $k+1$.
In the following, the Schnorr-Euchner search algorithm is modified to ensure that $z_k$ is constrained to $[l_k,u_k]$; see \cite{ChaG09}.

\begin{algorithm}\label{alg:searcheils}
\textnormal{(SEARCH-EILS)
Given nonsingular upper triangular matrix $\R \in \mathbb{R}^{n \times n}$ with positive diagonal entries, the vector $\by \in \mathbb{R}^{n}$, the initial search ellipsoid bound $\beta$ and the constraint ellipsoid bound $\alpha$.  The search algorithm finds the solution $\z \in \mathbb{Z}^{n}$ to the EILS problem (\ref{eq:rEILS}).
\begin{tabbing}
$\quad$ \= {\bf function}: $\z = \mathrm{SEARCH\_EILS}(\R, \by, \beta, \alpha)$ \\
\>   1. (Initialization) Set $k=n,  b_k = 0$ and $s_k = \alpha^2$  \\
\>  2. \= Set $lbound_k = 0$ and $ubound_k = 0$. \\
\>  \>  Compute $l_k = \Big\lceil \frac{- \sqrt{s_k} - b_k}{r_{kk}} \Big\rceil, u_k = \Big\lfloor \frac{ \sqrt{s_k} - b_k}{r_{kk}} \Big\rfloor$ \\
\>  \>  {\bf if} \= $u_k < l_k$  \\
\>  \>  \>  Go to Step 4 \\
\>  \>  {\bf end} \\
\>  \>  {\bf if} $u_k = l_k$ \\
\>  \>   \> Set $lbound_k = 1$ and $ubound_k = 1$ \\
\>  \>  {\bf end} \\
\>  \>  Compute $c_k = (\bar{y}_k - b_k)/r_{kk}$. Set $z_k = \round{c_k}$, \\
\>  \> {\bf if} $z_k \le l_k$ \\
\>  \>  \> $z_k = l_k$, set $lbound_k = 1$ and $\Delta_k = 1$ \\
\>  \> {\bf else if } $z_k \ge u_k$ \\
\>  \>  \> $z_k = u_k$, set $ubound_k = 1$ and $\Delta_k = -1$ \\ 
\>  \>  {\bf else } // no bound of the constraint is reached \\
\>  \> \> Set $\Delta_k = \mbox{sgn}(c_k - z_k)$ \\ 
\>  \> \=  {\bf end} \\     
\> 3. (Main step) \\
\>   \> {\bf if } \= $r_{kk}^2 (z_k - c_k)^2 > \beta^2 - \sum_{i=k+1}^n r_{ii}^2 (z_i - c_i)^2$ \\
\>   \>    \>    go to Step 4 \\ 
\>   \> {\bf else if}  $k > 1$ \\ 
\>   \>    \> Compute $b_{k-1} = \sum_{j=k}^n r_{k-1,j} z_j$, $s_{k-1} = s_k - (r_{kk} z_k + b_k)^2$\\
\>   \>    \>  Set $k=k-1$, go to Step 2  \\ 
\>   \>  {\bf else} $\quad$ // case k = 1 \\  
\>   \>     \>  go to Step 5 \\
\>   \>  {\bf end} \\ 
\> 4. (Invalid point) \\
\>    \> {\bf if} \= $k = n $ \\
\>    \>  \>    terminate \\
\>    \> {\bf else } \\
\>    \>    \> $k = k+1$, go to Step 6 \\
\>    \> {\bf end} \\	
\> 5. (Found valid point) \\  
\>   \> Set $\hbz = \z$, $\beta = \sum_{k=1}^n r_{kk}^2 (\hat{z}_k - c_k)^2$  \\
\>    \> $k = k + 1$, go to Step 6 \\   
\> 6. (Enumeration at level k) \\
\>  \> {\bf if} $ubound_k = 1$ and $lbound_k = 1$ \\
\>  \>  \> Go to Step 4 // no integer is available at this level \\
\>  \> {\bf end} \\ 
\>  \> Set $z_k = z_k + \Delta_k$ \\
\>  \> {\bf if} $z_k = l_k$ \\
\>  \>  \> Set $lbound_k = 1$, Compute $\Delta_k = -\Delta_k - \mbox{sgn}(\Delta_k)$ \\
\>  \> {\bf else if} $z_k = u_k$ \\
\>  \>  \> Set $ubound_k = 1$, Compute $\Delta_k = -\Delta_k - \mbox{sgn}(\Delta_k)$ \\ 
\>  \> {\bf else if} $lbound_k = 1$  \\
\>  \>  \> $\Delta_k = 1$ \\
\>  \> {\bf else if} $ubound_k = 1$ \\
\>  \> \> $\Delta_k = -1$ \\
\>  \> {\bf else} \\ 
\>  \> \> Compute $\Delta_k = -\Delta_k - \mbox{sgn}(\Delta_k)$ \\
\> \> {\bf end} \\
\>  \>  Go to Step 3. 
\end{tabbing}
}
\end{algorithm}

\section{Reduction}\label{s:NEILS}
As in the OILS problem, we can apply IGTs and permutations (see Section \ref{s:reduction}) in the reduction process of the EILS problem. 
With (\ref{eq:QRZ1}), (\ref{eq:QRZ2}) and (\ref{eq:QRZ3}), the original EILS problem (\ref{eq:EILS}) is transformed to the reduced EILS problem \eqref{eq:rEILS}.
Our new reduction for the EILS problem is based on a reduction strategy first applied to the box-constrained integer least squares problem.

\subsection{A constraint reduction strategy}\label{s:BILS}

In several applications, one wants to solve
\begin{equation}\label{eq:BILS}
\min_{\x \in \mathcal{B}} \|\y-\A \x\|_2^2, \quad \mathcal{B} = \{ \x \in \mathbb{Z}^n: \l \le \x \le \u, \l \in \mathbb{Z}^n, \u \in \mathbb{Z}^n\}.
\end{equation}  
We refer to (\ref{eq:BILS}) as a box-constrained integer least squares (BILS) problem.  

The transformations applied on $\A$ during the reduction can be described as a QR factorization of $\A$ with column pivoting:
\begin{equation}\label{eq:EILSQRZ1}
\Q^T \A \P = 
\begin{bmatrix}
\R \\
{\bf 0} 
\end{bmatrix}  \quad
\mbox{or } \A \P = \Q_{1}^T \R, 
\end{equation}
where $\Q = [\Q_1, \Q_2] \in \mathbb{R}^{m \times m}$ is orthogonal, $\R \in \mathbb{R}^{n}$ is nonsingular upper triangular and 
$\P \in \mathbb{Z}^{n \times n}$ is a permutation matrix.  This is a special case of the QRZ factorization presented in Section \ref{s:reduction}, where we have a permutation matrix $\P$ instead of a general unimodular matrix $\Z$.  The reason is that a general unimodular matrix will make the box constraint $\mathcal{B}$ very difficult to handle in the search.  Hence, the LLL reduction is usually not used to solve the BILS problem \eqref{eq:BILS} in the literature.   
With \eqref{eq:EILSQRZ1}, we have
\begin{equation}\label{eq:EILSQRZ2}
\| \y - \A \x \|_2^2 = \| \Q_1^T \y - \R \P^{T} \x \|_2^2 + \| \Q_2^{T} \y \|_2^2. 
\end{equation}
Let 
\begin{equation}\label{eq:EILSQRZ3}
\by = \Q_1^T \y, \quad \z = \P^{T} \x, \quad \bbl = \P^T \l, \quad \bu = \P^T \u.
\end{equation}
Then, from \eqref{eq:EILSQRZ2} and \eqref{eq:EILSQRZ3}, we see that (\ref{eq:BILS}) is equivalent to
\begin{equation}\label{eq:EILSrrILS}
\min_{\z \in \mathcal{\bar{B}}} \| \by-\R \z \|_2^2, \quad \mathcal{\bar{B}} = \{ \z \in \mathbb{Z}^n: \bbl \le \z \le \bu, \l \in \mathbb{Z}^n, \bu \in \mathbb{Z}^n\}.
\end{equation}
If $\hbz$ is the solution of the transformed BILS problem (\ref{eq:EILSrrILS}), then $\hbx = \P \hbz$ is the solution of the original BILS problem (\ref{eq:BILS}).   

Most reduction algorithms are solely based on $\A$. In \cite{ChaH08}, it was shown that using the information of $\y$ and the constraint in the reduction process can make the search process much more efficient.
We call their reduction strategy a ``constraint" reduction strategy.
Here, we describe its main idea. 
In the search process at level $i$, we have 
\begin{equation}\label{eq:ineq}
r^2_{ii}(z_i - c_i)^2 < \beta - \sum_{k=i+1}^n r_{kk}^2 (z_k - c_k)^2,
\end{equation}
where $c_i$ is determined when $z_{i+1}, \ldots, z_n$ are fixed (see \eqref{eq:c_k}).  
If we can reduce the search range of $z_i$ for $i=n,n-1,\ldots,1,$ then the search will be more efficient.
Notice that this can be achieved if 
\begin{itemize}
\item The right-hand side of (\ref{eq:ineq}) is as small as possible, which means that each $r_{kk}^2 (z_k - c_k)^2$ is as large as possible.
\item $r_{ii}$ is as large as possible. 
\end{itemize}

The constraint reduction strategy looks for a permutation of $\A$ such that $|r_{kk}(z_k - c_k)|$ is as large as possible for $k=n,n-1,\ldots,1$,
where we also take into account that $r_{kk}$ should be as large as possible.  The algorithm determines the columns of the permuted $\A$ from right to left.  To determine the $k$th column, it chooses from the remaining $k$ columns the one that maximizes $|r_{kk}(z_k - c_k)|$.  
Now, one question that arises is how to choose $z_k$.  The natural approach is to set $z_k$ to be the nearest integer to $c_k$ in $[\bar{l}_k,\bar{u}_k]$.  
This can yield the following problem.  If $z_k$ is very close to $c_k$, then $|r_{kk}(z_k - c_k)|$ is small even though $r_{kk}$ is large.  Since $r_{11} \ldots r_{nn}$ is constant (note that $\mbox{det}^{1/2}(\A^T\A)= \mbox{det}(\R) = r_{11} \ldots r_{nn}$), we might end up with a large $r_{ii}$ for small index $i$ and a small $r_{ii}$ for large index $i$. This does not comply with our requirement; see the first sentence of the paragraph.   
On the other hand, if we choose $z_k$ to be the second nearest integer to $c_k$ in $[\bar{l}_k,\bar{u}_k]$, then $|z_k - c_k|$ is always larger than 0.5.  Thus, if $r_{kk}$ is large,  then $|r_{kk}(z_k - c_k)|$ is also large.  Hence, the previous problem is avoided.  Simulations in \cite{ChaH08} indicate that choosing $z_k$ to be the second nearest integer to $c_k$ in $[\bar{l}_k,\bar{u}_k]$ is more effective than other possible choices of $z_k$.
\subsection{The ellipsoidal constraint}
 
We want a reduction for the EILS problem which uses all the available information in order to improve the search.     
The natural approach would be to use the constraint reduction strategy with the ellipsoidal constraint.
Our experiments show that such an approach is inefficient.  Here, we explain why the constraint reduction strategy is effective for the BILS problem, but  ineffective for the EILS problem.  In the reduction process, we permute $\A$ such that its $k$th column maximizes $|r_{kk}(z_k - c_k)|$, where $c_k$ depends on the chosen values of $z_{k+1}, \ldots, z_n$.
In the EILS problem and unlike the BILS problem, the constraint depends on $\z$ (see (\ref{eq:rEILS})).
In the search process at level $k$, $z_k$ can take values in the interval $[l_k, u_k]$, where $l_k$ and $u_k$ depend on $z_{k+1},\ldots, z_n$.
As $z_{k+1}, \ldots, z_n$ take on different values in the search, it is possible that the $k$th column of $\A$ no longer maximizes $|r_{kk}(z_k - c_k)|$.  Numerical experiments indicate that if we have a box-constraint, it is likely that $|r_{kk}(z_k - c_k)|$ remains large, which means that the search process remains efficient.  If we have an ellipsoidal constraint, it is much less likely that $|r_{kk}(z_k - c_k)|$ remains large since in addition to $c_k$, the constraints $l_k$ and $u_k$ also change as  $z_{k+1},\ldots, z_n$ change.  In other words, the extra uncertainty in the EILS problem makes it more difficult to determine which column maximizes $|r_{kk}(z_k - c_k)|$.
 
To overcome this difficulty, we construct the smallest hyper-rectangle which includes the constraint ellipsoid.  The edges of the hyper-rectangle are parallel to the $\z$-coordinate system.  We suggest to use this new box-constraint instead of the constraint ellipsoid in the reduction.
In the constraint reduction strategy, IGTs are not used since they make the box-constraint too difficult to handle in the search.   This difficulty does not occur with the ellipsoidal constraint as shown in Section \ref{s:searchEILS}.  Hence, we can modify the constraint reduction strategy by introducing IGTs in the reduction stage for the EILS problem. 
Note that the shape of the constraint ellipsoid changes after IGTs are applied, which means that 
the box-constraint needs to be recomputed.
While the box-constraint is less precise than the constraint ellipsoid, it has the advantage of being insensitive to the chosen values of $\z$ in the reduction.  This means that it is now more likely that $|r_{kk}(z_k - c_k)|$ remains large in the search process.  
  
\subsection{Computing the box-constraint}\label{computeBox}
 
In \cite{ChaG09}, it was shown how the smallest box-constraint $[\bbl,\bu]$ that includes the constraint ellipsoid can be efficiently computed. For $k=1:n$, we want to determine
\begin{equation}\label{eq:l_k}
\bar{l}_k = \lceil  \min_{\z \in \mathbb{R}^n} \e_k^T \z \rceil, \quad \mbox{given } \| \R \z \|_2 \le \alpha,
\end{equation}  
\begin{equation}\label{eq:u_k}
\bar{u}_k = \lfloor \max_{\z \in \mathbb{R}^n} \e_k^T \z \rfloor, \quad \mbox{given } \| \R \z \|_2 \le \alpha.
\end{equation}
We first solve for $\bar{u}_k$.
Let $ \p = \R \z$. Substituting in (\ref{eq:u_k}), we obtain
\begin{equation}\label{eq:p}
\bar{u}_k = \lfloor \max_{\p} \e_k^T \R^{-1} \p \rfloor, \quad \mbox{given } \| \p \|_2 \le \alpha.
\end{equation} 
Using the Cauchy-Schwarz inequality (see \cite[p.\ 53]{GolV96}), 
\begin{equation}
\e_k^T \R^{-1}\p \le \| \R^{-T} \e_k \|_2 \| \p \|_2 \le \| \R^{-T} \e_k \|_2 \alpha.
\end{equation}
The inequalities become equalities if and only if $\p$ and $\R^{-T} \e_k$ are linearly dependent, and $\| \p \|_2 = \alpha$, i.e.,  
$\p = \alpha \R^{-T} \e_k / \| \R^{-T} \e_k\|_2 $. Substituting $\p$ in (\ref{eq:p}), we get $\bar{u}_k = \lfloor \alpha \norm{ \R^{-T} \e_k }_2  \rfloor$.
Note that  $\max_{\z \in \mathbb{R}^n} \e_k^T \z = - \min_{\z \in \mathbb{R}^n} \e_k^T \z$ implies that 
$\bar{l}_k =  \lceil - \alpha \norm{ \R^{-T} \e_k }_2 \rceil$.  To efficiently compute $\R^{-T} \e_k$, we solve for $\q$ in the lower triangular system $\R^T \q = \e_k $. 
The algorithm is summarized as follows.

\begin{algorithm} 
\textnormal{(BOX).
Given the nonsingular upper triangular $\R \in \mathbb{R}^{n \times n}$, the constraint ellipsoid bound $\alpha$ 
and an integer $k$, where $k=1:n$. The algorithm computes the interval $[\bar{l}_k,\bar{u}_k]$ of the hyper-rectangle $[\bbl,\bu]$ which  
includes the constraint ellipsoid. 
\begin{tabbing}
$\quad$ \= {\bf function}: $[\bar{l}_k,\bar{u}_k] = \mathrm{BOX}(\R, \alpha, k)$ \\
\> Solve $\R^{T} \q = \e_k$ for $\q$ by forward substitution \\
\> Compute  $\bar{u}_k = \lfloor \alpha \norm{ \q }_2  \rfloor$ and $\bar{l}_k =  \lceil - \alpha \norm{ \q }_2 \rceil$
\end{tabbing}
}
\end{algorithm}

\subsection{A new reduction algorithm}\label{s:CLLLa}

Our new algorithm for the EILS problem merges the ideas of the constraint reduction strategy, which uses all the available information, and of the LLL reduction, which applies IGTs to strive for $r_{11} < \ldots < r_{nn}$.  We call our algorithm Constrained LLL (CLLL) reduction.  In the CLLL reduction, we use constraint information and IGTs to strive for $|r_{11}(z_1 - c_1)| < \ldots < |r_{nn}(z_n - c_n)|$.  

We now describe our reduction algorithm. The CLLL reduction starts by finding the QR decomposition of $\A$ by Householder transformations, then computes $\by$ and works with $\R$ from left to right. At the $k$th column of $\R$, the algorithm applies IGTs to ensure that $|r_{ik}| < \frac{1}{2} r_{ii}$ for $i=k-1\!:\!-1\!:\!1$. 
Then, it computes the constraints $[\bar{l}_k,\bar{u}_k]$ of $z_k$ (see Section \ref{computeBox}) and approximates $c_k$.  The reason that $c_k$ needs to be approximated is that at the $k$th column, $z_{k+1}, \ldots, z_n$ are not yet determined (see \eqref{eq:c_k}).  The algorithm approximates \eqref{eq:c_k}
by $\bar{c}_k = \bar{y}_k/r_{kk}$ for $k=n\!:\!-1\!:\!1$.  As in the constraint reduction strategy (see Sect.\ \ref{s:BILS}),  it sets $z_k$ to be the second nearest integer to $\bar{c}_k$ in $[\bar{l}_k,\bar{u}_k]$. 
Then, if permuting columns $k-1$ and $k$ maximizes $|r_{kk}(z_k - \bar{c}_k)|$, it does so, applies a Givens rotation to $\R$ from the left to bring $\R$ back to an upper triangular form, simultaneously applies the same Givens rotation to $\by$, and moves down to column $k-1$; otherwise it moves up to column $k+1$. We present our implementation of the CLLL reduction.

\begin{algorithm}
\textnormal{(CLLL REDUCTION).
Given the generator matrix $\A \in \mathbb{R}^{m \times n}$, the input vector $\y \in \mathbb{R}^{m}$ and 
the constraint ellipsoid bound $\alpha$.  The algorithm returns the reduced upper triangular matrix $\R \in \mathbb{R}^{n \times n}$, the 
unimodular matrix $\Z \in \mathbb{Z}^{n \times n}$, and the vector $\by \in \mathbb{R}^{n}$. 
\begin{tabbing}
$\quad$\= {\bf function}: $[\R, \Z,\by] = \mathrm{CLLL}(\A, \y,\alpha)$ \\
\> Compute the QR decomposition of $\A$ and set $\by = \Q_{1}^T\y$ \\
\> $\Z := \I_n$ \\
\> $k = 2$ \\
\> {\bf while} \= $k \le n$  \\
\>            \> {\bf for}  \= $i=k-1\!:\!-1\!:\!1$ \\
\>	       \>    \> \mbox{Apply the IGT $\Z_{ik}$ to $\R$, i.e, $\R := \R \Z_{ik}$} \\
\>	       \>    \> \mbox{Update $\Z$, i.e, $\Z := \Z \Z_{ik}$} \\
\>            \>   {\bf end} \\
\>            \> $\R^{\prime} := \R$ \\
\>            \>  \mbox{Interchange columns $k-1$ and $k$ of $\R^{\prime}$ and transform $\R^{\prime}$ to} \\
\>            \>  \mbox{an upper triangular matrix by a Givens rotation, $\G$} \\
\>	      \>  $ \by^{\prime} = \G \by  $ \\
\>            \> \mbox{ // Compute the box constraint of $\R$ for $z_k$} \\
\>            \>  $[\bar{l}_k,\bar{u}_k] = \mbox{BOX}(\R,\alpha,k) $ \\
\>            \> \mbox{ // Compute the box constraint of $\R^{\prime}$ for $z^{\prime}_k$ } \\
\>            \> $[\bar{l}^{\prime}_k,\bar{u}^{\prime}_k]=\mbox{BOX}(\R^{\prime},\alpha,k)$\\
\>	      \>  // \mbox{Compute $|r_{kk}(z_k - \bar{c}_k)|$ } \\
\>            \>   $\bar{c}_k := y_k/r_{kk} $\\
\>            \>   $\bar{c}^{\prime}_k := y^{\prime}_k/r^{\prime}_{kk}$ \\
\>	      \>  \mbox{ Set $z_k$ to be the second nearest integer to $\bar{c}_k$ on $[\bar{l}_k,\bar{u}_k]$ } \\
\>	      \>  \mbox{ Set $z^{\prime}_k$ to be the second nearest integer to $\bar{c}^{\prime}_k$ on $[\bar{l}^{\prime}_k,\bar{u}^{\prime}_k]$ } \\
\> 	      \> {\bf if} $|r^{\prime}_{kk}(z^{\prime}_k-\bar{c}^{\prime}_k)| > |r_{kk}(z_k-\bar{c}_k)|$ \\
\>	      \>     \>  $\R := \R^{\prime}$\\
\>            \>     \>  $\by := \by^{\prime}$ \\
\>	      \>		     \> {\bf if} \= $k > 2$ \\
\>            \>                     \> 	\> $k = k - 1$ \\
\>	      \>		     \> {\bf end} \\
\>            \> {\bf else}\\
\>            \>         \> $k=k+1$ \\ 
\>	      \> {\bf end} \\
\> {\bf end}
\end{tabbing}
}
\end{algorithm}
Note that the CLLL reduction algorithm moves from the left to the right of $\R$, which explains why $z_{k+1}, \ldots, z_n$ are not determined at the $k$th column.  It is straightforward to modify the LLL reduction algorithm for it to move from right to left, which allows us to determine $c_k$ exactly.
Surprisingly, simulation results indicate that such an approach is less effective than the CLLL reduction.  Further investigation is required to understand why $|r_{kk}(z_k - \bar{c}_k)|$ is a better criterion than $|r_{kk}(z_k - c_k)|$ to determine the permutation of $\A$.

\section{Numerical simulations}\label{s:simulationEILS}

In this section, we implemented the CLLL algorithm given in Section \ref{s:CLLLa}. We did numerical simulations
to compare its effectiveness with the LLL reduction.  
Algorithm \ref{alg:searcheils} is used for the search process.  All our simulations were performed in MATLAB 7.9 on a Pentium-4, 2.66 GHz machine with 501 MB memory running Ubuntu 8.10. 

\subsection{Setup}
We took $\A$ to be $n \times n$ matrices drawn from i.i.d. zero-mean, unit variance Gaussian distribution.  We construct $\y$ as follows
\begin{equation*}
\y = \A \x + \v,
\end{equation*}
where the noise vector $\v \sim \mathcal{N}(0,\sigma^2 I)$.   To generate $\x$, we randomly pick an integer point inside some hyper-rectangle.  Then, we set 
$\alpha = \| \A \x \|_2 $.  
In Figs.\ \ref{f:n05} to \ref{f:n10}, we display the average CPU search time in seconds for $\sigma=0.5$ to $\sigma=10$. 
We took dimensions $n=5\!:\!30$ for Figs.\ \ref{f:n05} and \ref{f:n1}, $n=5\!:\!15$ for Fig.\ \ref{f:n2}, $n=5\!:\!10$ for Fig.\ \ref{f:n4}, and performed 20 runs for each case.   
In Fig.\ \ref{f:n10}, we took $n=4\!:\!8$ and performed 5 runs. 
The reason that the dimensions of the experiments get smaller as $\sigma$ gets larger is that the search process with the LLL reduction becomes extremely time-consuming.  For instance, for typical problems with dimension $n$ = 15 and $\sigma=2$, the search time with the LLL reduction is more than $10^3$s, while the search time with the CLLL reduction is around 1s.    In Fig.\ \ref{f:Babai}, we compare the Babai integer points corresponding to the LLL reduction and the CLLL reduction.  We give the ratio between $\beta$ at the Babai integer point and $\beta$ at the EILS solution with different noise, where $\beta = \beta(z) = \| \by-\R \z \|_2^2$. We present the results for dimension $n=5$ with 20 runs.  We obtain similar results for other dimensions.  
As shown in \cite{ChaG09}, the cost of the search time dominates the cost of the whole algorithm when $\sigma \ge 0.5$.  
For this reason, the figures do not take into account the reduction time, which is negligible. 

\begin{figure}[h!]
\centering
\includegraphics[scale=0.75]{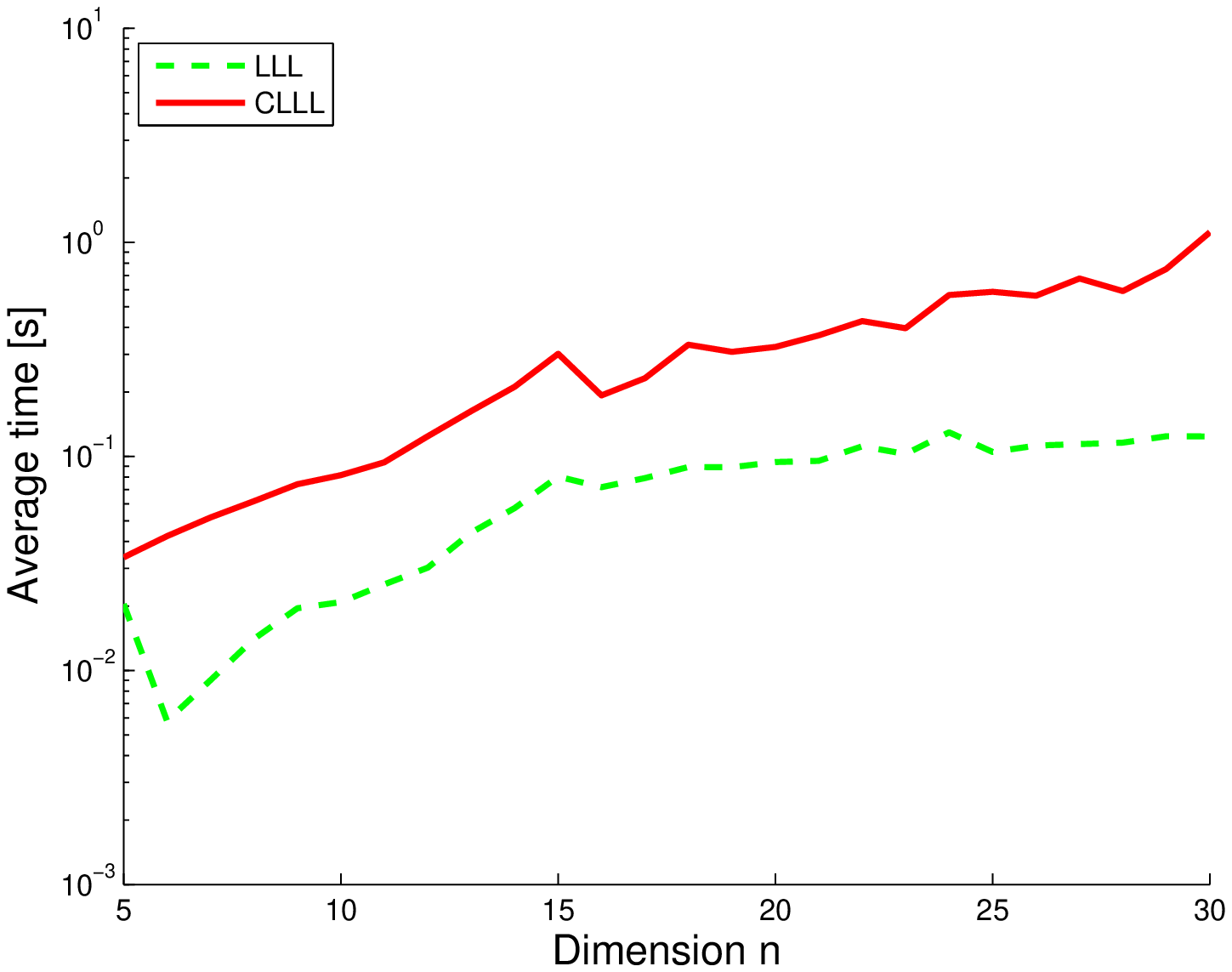}
\caption{Average search time versus dimension, $\sigma = 0.5$.} \label{f:n05}
\end{figure}

\begin{figure}[h!]
\centering
\includegraphics[scale=0.75]{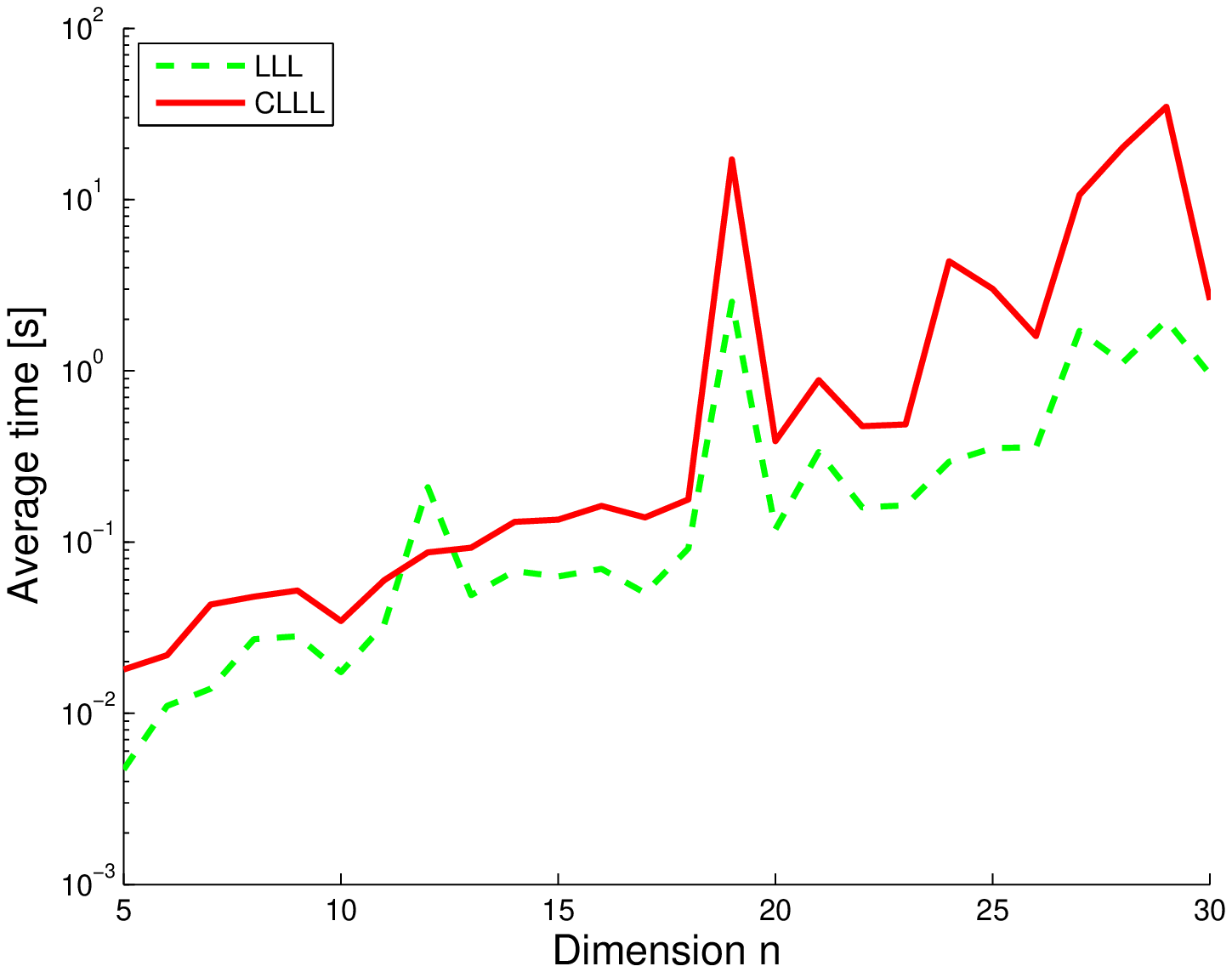}
\caption{Average search time versus dimension, $\sigma = 1$.} \label{f:n1}
\end{figure}

\begin{figure}[h!]
\centering
\includegraphics[scale=0.75]{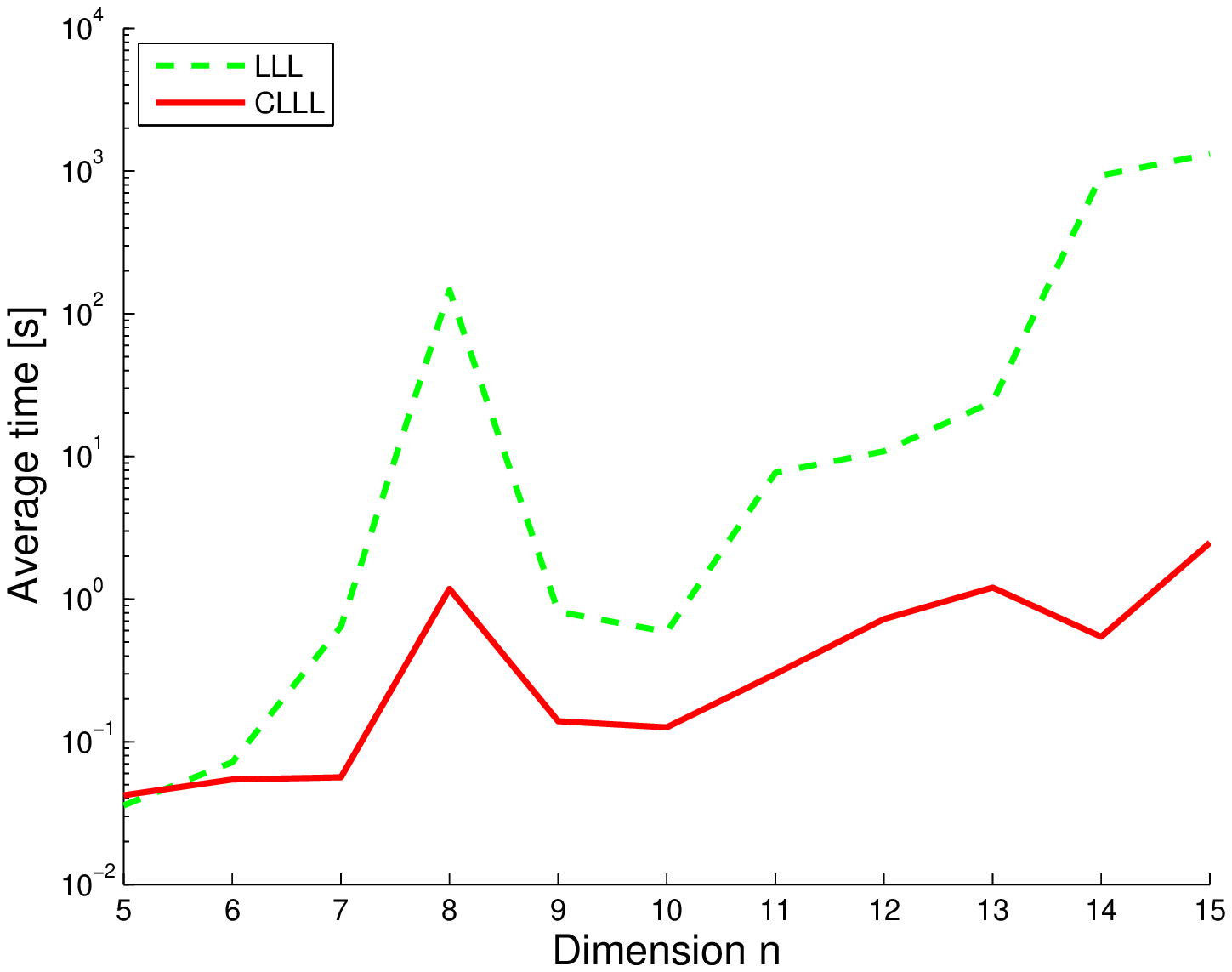}
\caption{Average search time versus dimension, $\sigma = 2$.} \label{f:n2}
\end{figure}

\begin{figure}[h!]
\centering
\includegraphics[scale=0.75]{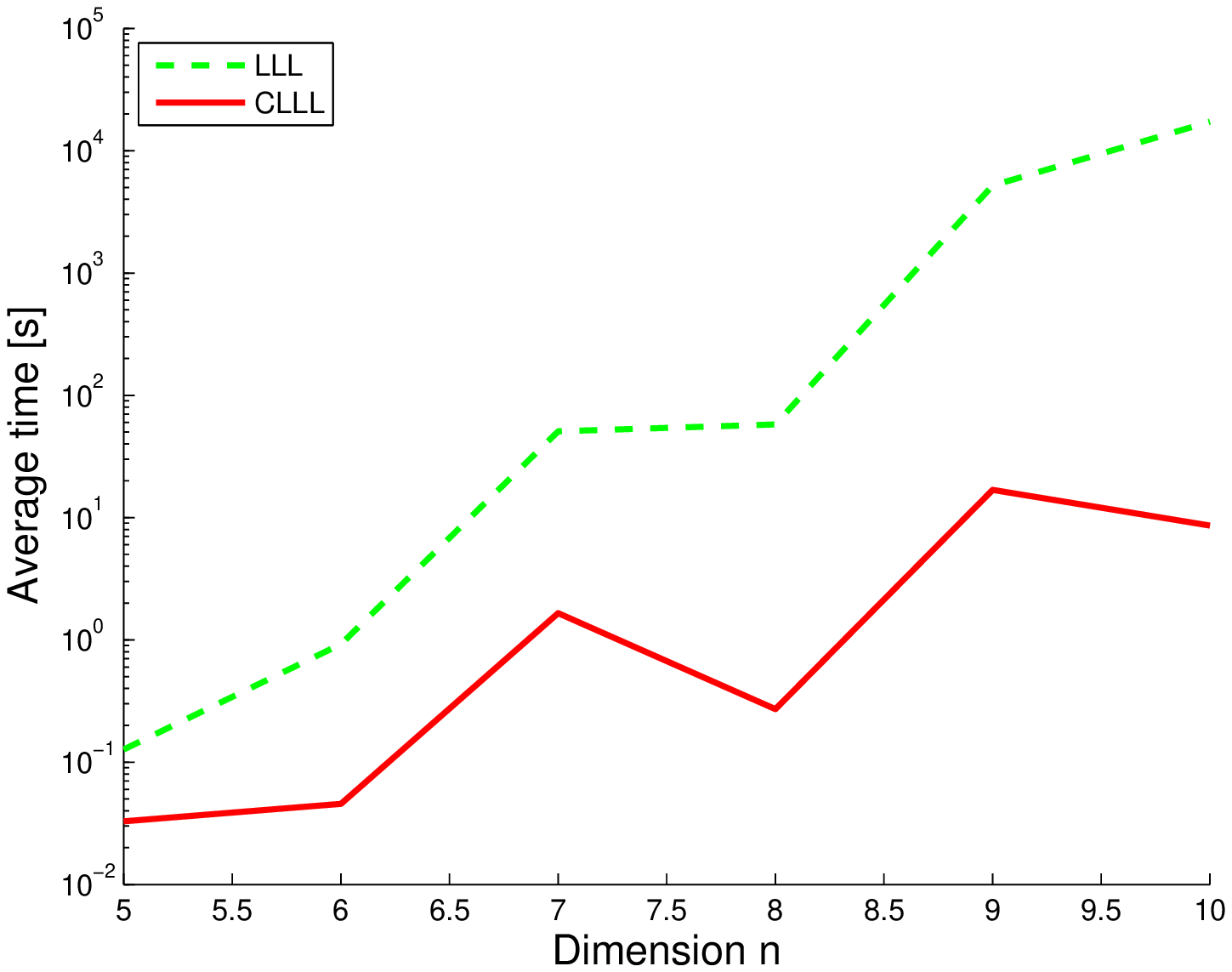}
\caption{Average search time versus dimension, $\sigma = 4$.} \label{f:n4}
\end{figure}

\begin{figure}[h!]
\centering
\includegraphics[scale=0.75]{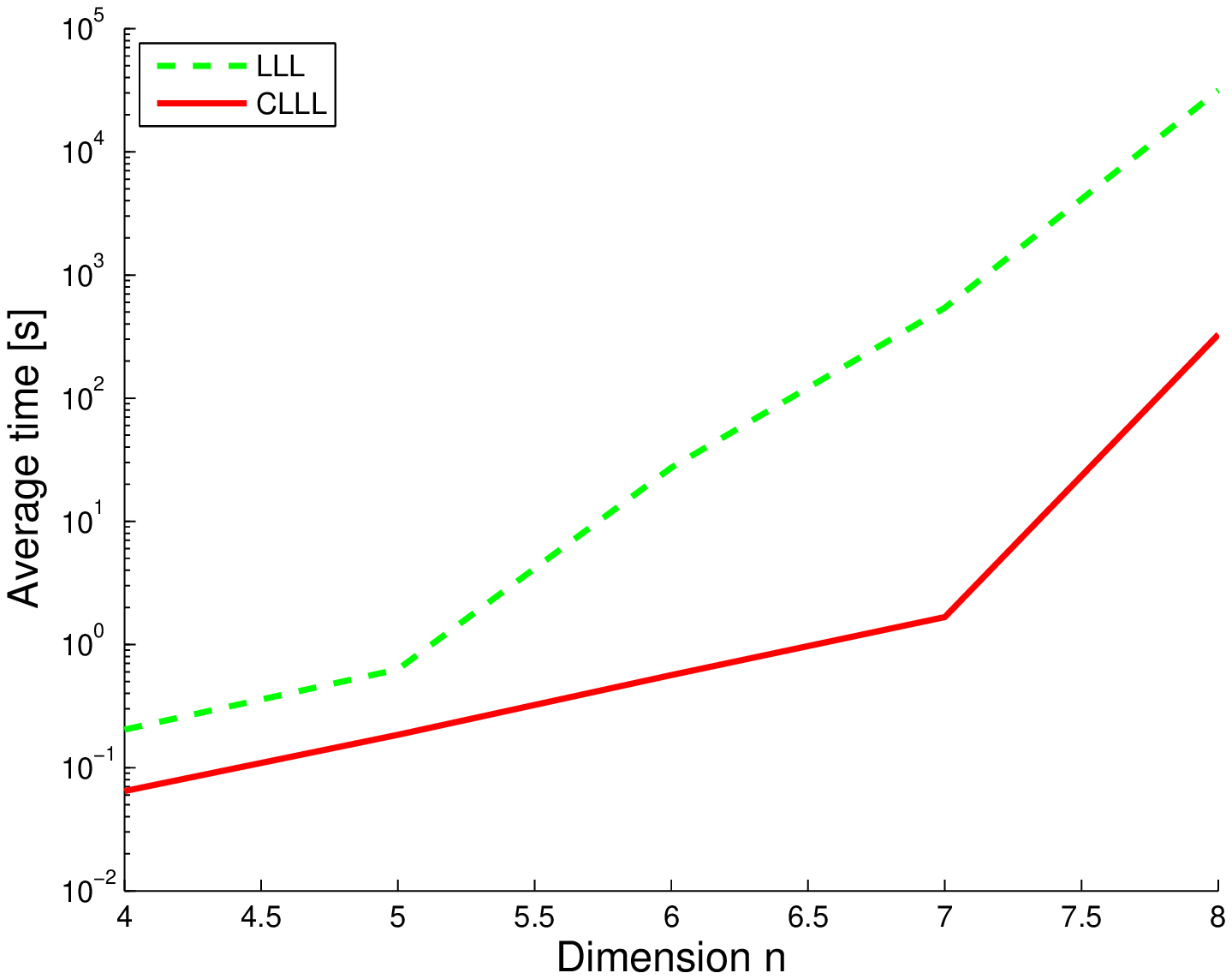}
\caption{Average search time versus dimension, $\sigma = 10$.} \label{f:n10}
\end{figure}

\begin{figure}[h!]
\centering
\includegraphics[scale=0.75]{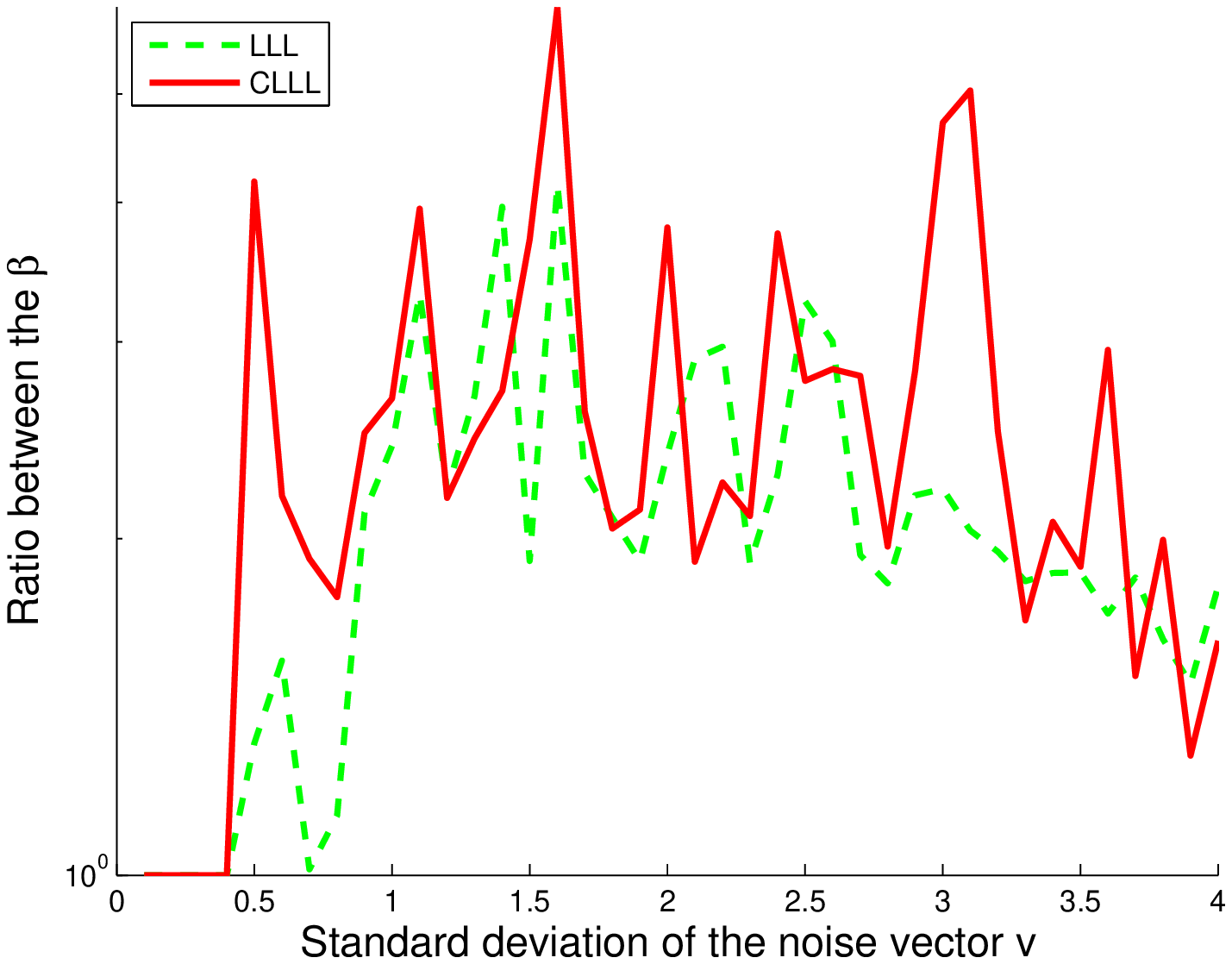}
\caption{Ratio between $\beta$ at the Babai integer point and $\beta$ at the EILS solution, dimension $n=5$.}\label{f:Babai}
\end{figure}

\subsection{Comparison of the reduction strategies}

From the simulations, we see that the most effective reduction algorithm depends on the noise size. In Figs.\ \ref{f:n05} and \ref{f:n1}, we see that when the noise is small, i.e, $\sigma \le 1$, the LLL reduction is more effective than the CLLL reduction. When the noise gets larger, the CLLL reduction becomes much more effective than the LLL reduction.  Notice that the improvement of the CLLL reduction over the LLL reduction becomes more significant with larger noise.   For example, when $n=7$, the CLLL reduction is slightly more effective than the LLL reduction when $\sigma=2$, but close to 1000 times more effective when $\sigma=10$.  We observe that with the LLL reduction, the search time becomes more and more prohibitive as the noise gets larger.
The CLLL reduction provides considerable savings in the search time.  For example when $\sigma=2$ and $n=8$, the search time with the LLL reduction is more than 100s, while it is about 1s with the CLLL reduction.  When $\sigma=4$ and $n=10$, the search time with the LLL reduction is more than $10^4$s, while it is about 10s with the CLLL reduction.   

We now explain why the LLL reduction is preferable over the CLLL reduction in Figs.\ \ref{f:n05} and \ref{f:n1}.     
When the noise is small, it is likely that the OILS solution is close the ellipsoidal constraint. As can be seen in Fig.\ \ref{f:Babai}, in these situations the Babai integer point found with the LLL reduction is usually very close to the EILS solution.  With the CLLL reduction, the goal is to make  $|r_{kk}(z_k - \bar{c}_k)|$ as large as possible for $k=n,\ldots,1$.  While $|\prod_{k=1}^n r_{kk}|$ is constant through the reduction process, $|\prod_{k=1}^n r_{kk}(z_k-\bar{c}_k)|$ is not.  Making $|r_{kk}(z_k - \bar{c}_k)|$ large for some $k$ will not make $|r_{jj}(z_j-\bar{c}_j)|$ smaller for $j < k$.   It is possible that the CLLL reduction permutes $\A$ such that $\sum_{k=1}^{n} r_{kk}^2(z_k - \bar{c}_k)^2$ ends up to be a large number.  While such an ordering allows to prune more points in the search process, it also means that the Babai integer point found with the CLLL reduction is usually worse than the one found with the LLL reduction.  A better Babai integer point makes the intersection between the search ellipsoid and the constraint ellipsoid smaller.
This implies that the search process with the CLLL reduction must find many points before it reaches the EILS solution, while with the LLL reduction, only very few points are found before the EILS solution.  For large noise (see Figs.\ \ref{f:n2} to \ref{f:n10}), it is no longer true that the Babai integer point found with the LLL reduction is very close to the EILS solution.  It is here where the extra search pruning that the CLLL reduction provides becomes advantageous.   
Thus, the LLL reduction is the most effective reduction algorithm for small noise, which includes many communications applications; 
while the CLLL reduction is by far the most effective reduction algorithm for large noise, which includes the applications where the linear model (see \eqref{eq:linearmodel}) is not assumed.

\chapter{Summary and future work}\label{s:summary}

This thesis was concerned with solving the ordinary integer least squares (OILS) problem
\begin{equation*}
\min_{\x \in \mathbb{Z}^n} \| \y-\A \x \|_2^2,
\end{equation*}
where $\y \in \mathbb{R}^n$ and $\A \in \mathbb{R}^{m \times n}$ has full column rank.
In the GNSS literature, one needs to solve the following quadratic form of the OILS problem
\begin{equation*}
\min_{\x\in \mathbb{Z}^n} (\x-\hbx)^T\W_{\hbx}^{-1}(\x-\hbx),
\end{equation*} 
where $\hbx\in \mathbb{R}^n$ is the real-valued least squares (LS) estimate of the double differenced integer ambiguity vector $\x\in \mathbb{Z}^n$, and $\W_{\hbx}\in \mathbb{R}^{n\times n}$ is its covariance matrix, which is symmetric positive definite.

There are two steps in solving an OILS problem: reduction and search.  The main focus of this thesis was on the reduction step. 

In Chapter \ref{s:misconceptions}, we have shown that there are two misconceptions about the reduction in the literature.
The first is that the reduction should decorrelate the ambiguities as much as possible.  We have proved that this is incorrect: only some ambiguities should be decorrelated as much as possible.  
Our new understanding on the role of IGTs in the reduction process led to the PREDUCTION algorithm, a more computationally efficient and stable reduction algorithm than both LAMBDA reduction and MREDUCTION.      The second misconception is that the reduction process should reduce the condition number of the covariance matrix.  We gave  examples which demonstrate that the condition number is an ineffective criterion to evaluate the reduction.  Finally, we translated our result from the quadratic OILS form to the standard OILS form.  Our new understanding on the role of IGTs in the LLL reduction algorithm led to the more efficient PLLL reduction algorithm.

In Chapter \ref{s:EILS}, we discussed how to solve the ellipsoid-constrained integer least squares (EILS) problem
\begin{equation}
\min_{\x \in \mathcal{E}} \| \y-\A \x \|_2^2, \quad \mathcal{E} = \{ \x \in \mathbb{Z}^n: \| \A \x \|_2^2 \le \alpha^2 \},
\end{equation}
where  $\y \in \mathbb{R}^n$ and $\A \in \mathbb{R}^{m \times n}$ has full column rank.  With the existing reduction algorithms for the EILS problem, the search process  is extremely time-consuming for large noise.   
We proposed a new reduction algorithm which, unlike existing algorithms, uses all the available information: the generator matrix, the input vector and the ellipsoidal constraint. 
Simulation results indicate that the new algorithm greatly decreases the computational cost of the search process for large noise. 

In the future, we would like to investigate the following problems:
\begin{itemize}
\item Box-constrained integer least squares (BILS) problems often arise in communications applications.   In the literature of BILS problems, IGTs are usually not applied in the reduction phase since they make the box constraint too difficult to handle in the search phase.   
We would like to see if the modified box constraint can be efficiently approximated by a larger constraint, which can be used easily in the search.     
The search process would have to be modified to ensure that the search ellipsoid is only updated if 
the integer point found satisfies the initial box-constraint.
While the larger constraint makes the search less efficient, the reduction becomes more effective due to the IGTs.   
\item For the BILS and EILS problem, it was shown that using all the available information in the reduction phase can be very effective.
For the OILS problem, the LLL reduction is solely based on the generator matrix $\A$.  We would like to determine if using $\y$ can lead to a more effective reduction algorithm.     
\end{itemize}

\bibHeading{References}

\bibliographystyle{plain}


\end{document}